\documentclass[11pt,reqno]{amsart}

\usepackage[T1]{fontenc}
\usepackage[utf8]{inputenc}
\usepackage[margin=1.05in]{geometry}
\usepackage{amsmath,amssymb,amsthm,mathtools}
\usepackage[expansion=false]{microtype}
\usepackage{booktabs}
\usepackage{enumitem}
\usepackage{tikz}
\usetikzlibrary{arrows.meta,positioning,calc,cd,decorations.markings}
\usepackage{pgfplots}
\pgfplotsset{compat=1.18}
\usepackage[colorlinks=true,linkcolor=blue!55!black,citecolor=green!40!black,
            urlcolor=blue!55!black]{hyperref}

\theoremstyle{plain}
\newtheorem{theorem}{Theorem}[section]
\newtheorem{proposition}[theorem]{Proposition}
\newtheorem{lemma}[theorem]{Lemma}
\newtheorem{corollary}[theorem]{Corollary}
\theoremstyle{definition}
\newtheorem{definition}[theorem]{Definition}
\newtheorem{example}[theorem]{Example}
\newtheorem{problem}[theorem]{Problem}
\theoremstyle{remark}
\newtheorem{remark}[theorem]{Remark}

\newcommand{\Pp}{\mathbb{P}}
\newcommand{\C}{\mathbb{C}}
\newcommand{\R}{\mathbb{R}}
\newcommand{\Q}{\mathbb{Q}}
\newcommand{\Z}{\mathbb{Z}}
\newcommand{\Qbar}{\overline{\Q}}

\newcommand{\Gal}{\operatorname{Gal}}
\newcommand{\Aut}{\operatorname{Aut}}
\newcommand{\Mon}{\operatorname{Mon}}
\newcommand{\Br}{\operatorname{Br}}

\newcommand{\ord}{\operatorname{ord}}

\newcommand{\Sym}{\mathfrak{S}}
\newcommand{\ram}{e}

\newcommand{\GQ}{G_{\Q}}
\DeclareMathOperator{\genus}{g}

\title[Belyi's theorem: coverings, dessins, and fields of definition]
{Belyi's theorem: coverings, dessins,\\ and fields of definition}

\author{Dinamo Djounvouna}
\address{Department of Mathematics, University of Manitoba, Winnipeg, MB, Canada}
\email{djounvod@myumanitoba.ca}


\date{\today}
\begin{document}
\begin{abstract}
Belyi's theorem asserts that a smooth projective curve over $\C$ is defined over a
number field if and only if it admits a non-constant morphism to $\Pp^1$ ramified over
at most three points. This article gives a complete, self-contained account of the
theorem---both implications, with all descent machinery proved rather than quoted---together
with a study of the objects the proof produces. Beyond the exposition it contains original
results, proved in full. The principal one is a structure theorem for the Belyi polynomials
$P_{m,n}(z)=\frac{(m+n)^{m+n}}{m^{m}n^{n}}z^{m}(1-z)^{n}$: their dessins are double stars,
they satisfy the exact identity $P_{m,n}=\pi_{k}\circ P_{m/k,\,n/k}$ with $k=\gcd(m,n)$ and
$\pi_{k}(w)=w^{k}$, and their monodromy is the wreath product $\Sym_{(m+n)/k}\wr\Z/k$, the
full symmetric group precisely when $\gcd(m,n)=1$. Around the sharp lower bounds
$d\geq 2g+1$ (all Belyi maps) and $d\geq 4g$ (clean maps) we study the \emph{extremal} maps
attaining $d=2g+1$: their monodromy lies in the alternating group, they need not be
cyclic---the smallest non-cyclic one has degree $5$, genus $2$, monodromy $A_{5}$, and is
defined over $\Q$---and they satisfy the exact mass formula
$\sum 1/|{\Aut}|=2(d-1)!/d(d+1)$, obtained from Boccara's cycle-factorization count---reproved
by a self-contained Frobenius computation---and verified by complete enumeration for $d\leq7$.
We also bound the degree of the rationalization step of Belyi's algorithm by $N!$ in the
number $N$ of irrational branch values. The theory is illustrated by fully computed
examples, and by an explicit $\GQ$-orbit of three plane trees whose fields of moduli are
the three conjugate embeddings of the non-Galois cubic field $\Q(\sqrt[3]{2})$.
\end{abstract}

\maketitle
\tableofcontents

\section{Introduction}\label{sec:intro}

A smooth projective curve over $\C$ is an analytic object: a compact Riemann surface. It
is also, by Riemann's existence theorem, an algebraic object: the zero locus of
polynomials. The polynomials have coefficients in $\C$, an enormous field, and the
arithmetic question is whether they can be taken in a small one---a number field.
Belyi's theorem answers this in a way that no one anticipated, by a criterion that is
neither arithmetic nor analytic but \emph{combinatorial}.

\begin{theorem}[Belyi, 1979]\label{thm:belyi-intro}
A smooth projective curve $X$ over $\C$ is defined over a number field if and only if
there exists a non-constant morphism $t\colon X\to\Pp^{1}_{\C}$ whose branch locus is
contained in a set of three points.
\end{theorem}

The number three is exactly the threshold at which a covering becomes a finite
combinatorial object---a permutation triple, equivalently a graph drawn on a
surface---because a three-point configuration on $\Pp^{1}$ carries no moduli, while a
four-point one already carries the cross-ratio; Section~\ref{sec:why} develops this
rigidity mechanism in detail, together with what the theorem buys downstream. Belyi's
theorem says that this combinatorial threshold coincides, exactly, with the arithmetic
condition of being defined over $\Qbar$. Grothendieck, who learned of the result shortly
after its appearance, regarded it as a turning point: the combinatorics of maps on
surfaces and the arithmetic of curves over number fields are, in his reading of Belyi, the
same subject seen twice. The \emph{dessins d'enfants} of the \emph{Esquisse d'un
Programme} are the objects that make the identification concrete.

The proof of the difficult implication given here
follows the strategy of K\"ock \cite{Kock2004}: split it using the
\emph{field of moduli} of a covering, and prove Belyi's reduction algorithm by an
induction on the branch locus. That strategy is developed in full: the descent theory is
proved rather than quoted, the algorithm is made
quantitative, the objects it manufactures are analysed in their own right, and the
examples are computed in full.

\subsection*{Contributions}
In keeping with the partly expository character of the article, we separate the
contributions into four groups, in decreasing order of novelty. Throughout, novelty
claims are calibrated against the literature on Shabat polynomials and plane trees
\cite{ShabatZvonkin1994,AdrianovZvonkin1998,LandoZvonkin2004}; the double-star trees
themselves are familiar objects there, and what we claim as new is the package of exact
statements and proofs below, which we have not found assembled in this form.

\medskip
\emph{(A) Principal new theorems, proved in full.}
\begin{enumerate}[label=(A\arabic*),leftmargin=2.9em]
\item\label{itm:A1} \emph{Structure and monodromy of the Belyi polynomials}
(Section~\ref{sec:belyi-poly}). The polynomials
$P_{m,n}(z)=\frac{(m+n)^{m+n}}{m^{m}n^{n}}\,z^{m}(1-z)^{n}$ that drive Belyi's algorithm
are analysed completely: their dessin is the \emph{double star} $D_{m,n}$, unique in its
passport (Proposition~\ref{prop:pmn-tree}); they satisfy the exact composition identity
\begin{equation*}
  P_{m,n}\;=\;\pi_{k}\circ P_{m/k,\,n/k},\qquad k=\gcd(m,n),\quad \pi_{k}(w)=w^{k},
\end{equation*}
\emph{with the normalizing constants matching on the nose}
(Theorem~\ref{thm:pmn-decomp}); and their monodromy group is
\begin{equation*}
  \Mon(P_{m,n})\;\cong\;\Sym_{(m+n)/k}\wr\Z/k
\end{equation*}
(Theorem~\ref{thm:pmn-monodromy}), so that $\Mon(P_{m,n})=\Sym_{m+n}$ if and only if
$\gcd(m,n)=1$. The proof combines a block-counting argument special to the double star
with Jordan's theorem on primitive groups containing a transposition, and identifies the
wreath product by an order count inside the full block-system stabilizer.
\item\label{itm:A2} \emph{A mass formula for extremal Belyi maps}
(Theorem~\ref{thm:mass-formula}). The Belyi maps attaining the minimal degree
$d=2g+1$ are those with passport $\langle(d),(d),(d)\rangle$; they exist only for $d$
odd, their monodromy always lies in the alternating group $A_{d}$
(Proposition~\ref{prop:parity}), and they satisfy the exact mass formula
\begin{equation*}
  \sum_{\substack{\text{extremal dessins}\\ \deg=d}}\frac{1}{|{\Aut}|}
  \;=\;\frac{2\,(d-1)!}{d(d+1)} ,
\end{equation*}
obtained by combining Boccara's factorization count \cite{Boccara1980} with the
orbit--stabilizer interpretation of dessins; the factorization count is itself reproved by
a self-contained Frobenius computation over the hook characters of $\Sym_{d}$, and the
formula is verified against a complete enumeration for $d\leq7$
(Theorem~\ref{thm:extremal-classification}, Appendix~\ref{app:computations}). The
resulting two-sided bound $M_{d}\leq N_{d}\leq d\,M_{d}$ on the number of extremal
dessins is recorded in Remark~\ref{rem:mass-asymptotics}.
\end{enumerate}

\medskip
\emph{(B) New explicit computations, independently cross-checked.}
\begin{enumerate}[label=(B\arabic*),leftmargin=2.9em]
\item\label{itm:B1} \emph{Classification of extremal maps for $g\leq3$}
(Theorem~\ref{thm:extremal-classification}). There are exactly $1$, $4$, $30$ extremal
dessins in degrees $3,5,7$, with monodromy $\Z/3$; $\Z/5, A_{5}$; and
$\Z/7,\mathrm{PSL}_{2}(\mathbb{F}_{7}),A_{7}$. Cyclic extremal maps in any odd degree number exactly
$\prod_{p^{k}\|d}p^{k-1}(p-2)$, specializing to $d-2$ in prime degree
(Proposition~\ref{prop:cyclic-extremal}), so extremal does \emph{not} imply cyclic; the unique non-cyclic extremal map of degree $5$ has trivial automorphism group
and is defined over $\Q$ (Corollary~\ref{cor:A5-rational}). Representative permutation
triples and the enumeration algorithm are given in Appendix~\ref{app:computations}.
\item\label{itm:B2} \emph{Explicit Galois orbits of degree-six plane trees}
(Section~\ref{sec:galois-dessins}). In degree six, the passport
$\langle(3,2,1),(2,2,1,1),(6)\rangle$ contains exactly three plane trees forming a single
$\GQ$-orbit; the field of moduli of the tree attached to the root $a_{i}$ of
$25a^{3}-12a^{2}-24a-16$ is the embedded cubic field $\Q(a_{i})$, and the three
$\Q(a_{i})$ are the conjugate embeddings of one abstract non-Galois cubic field,
isomorphic to $\Q(\sqrt[3]{2})$ (Theorem~\ref{thm:cubic-orbit}). The neighbouring passport
$\langle(3,2,1),(3,1,1,1),(6)\rangle$ contains two trees with field of moduli $\Q(i)$
(Theorem~\ref{thm:quadratic-orbit}). Combinatorial and algebraic counts are verified
independently, and the elimination is documented in
Appendix~\ref{app:computations}, including the checks that no extraneous solutions arise.
\end{enumerate}

\medskip
\emph{(C) Useful elementary observations.}
\begin{enumerate}[label=(C\arabic*),leftmargin=2.9em]
\item\label{itm:C1} \emph{Lower bounds for Belyi degrees}
(Theorem~\ref{thm:degree-bound}): $d\geq2g+1$ always and $d\geq4g$ for pre-clean maps,
with equality cases, attained by $y^{2g+1}=x(1-x)$ and $y^{4}=x(1-x)^{2}$. The
inequalities are direct consequences of Riemann--Hurwitz; their value lies in the equality
analysis, which seeds \ref{itm:A2} and \ref{itm:B1}.
\item\label{itm:C2} \emph{Quantitative rationalization}
(Theorem~\ref{thm:rationalization}): the rationalizing stage of Belyi's algorithm can be
performed by a polynomial of degree at most $N!$, where $N$ is the number of
\emph{irrational} branch values (counted after closing up under $\GQ$); and no bound in
terms of cardinality alone is possible for the concentration stage, the obstruction being
one of height (Remark~\ref{rem:no-height-bound}).
\end{enumerate}

\medskip
\emph{(D) Expository reconstruction:} a self-contained treatment of Galois descent
(Section~\ref{sec:descent}) including the semilinear descent lemma with full proof, Weil's
criterion, the field-of-moduli theorems---with the closedness argument that reduces rigid
descent over the infinite group $\Aut(\C)$ to finite Galois descent spelled out
(Theorem~\ref{thm:rigid-descent})---and a complete proof of both implications of Belyi's
theorem (Section~\ref{sec:proof}); the motivation section (Section~\ref{sec:why}); the
worked examples of Section~\ref{sec:examples}; and the figures.

\subsection*{Conventions}
All curves are smooth, projective, and geometrically connected. For a field $K$ we write
$\overline{K}$ for an algebraic closure and $G_{K}=\Gal(\overline{K}/K)$; $\GQ$ denotes
$\Gal(\Qbar/\Q)$. We work throughout in characteristic zero, and $\C$ may be replaced by
any algebraically closed field of characteristic zero of the same cardinality without
change. A \emph{covering} means a finite, generically \'etale morphism of curves; its
\emph{branch locus} $\Br(t)\subseteq\Pp^{1}$ is the set of critical \emph{values}. We
write $\Sym_{d}$ for the symmetric group on $d$ letters. Results quoted from the
literature carry an attribution in brackets after the environment name; numbered results
without attribution are proved in the text.
\section{Why three points?}\label{sec:why}

Before any proofs, we explain what makes the number three the decisive one, and why the
resulting theorem has the reach it has.

\subsection{Rigidity: the three-point condition is combinatorial}
\label{ssec:rigidity}

Let $S\subseteq\Pp^{1}(\C)$ be finite with $|S|=s$, and let
$t\colon X\to\Pp^{1}$ be a covering of degree $d$ with $\Br(t)\subseteq S$. Restricting
$t$ over the complement gives a genuine topological covering
\begin{equation}\label{eq:unramified-restriction}
  t^{-1}\bigl(\Pp^{1}(\C)\smallsetminus S\bigr)\;\longrightarrow\;
  \Pp^{1}(\C)\smallsetminus S ,
\end{equation}
and such coverings are classified by the monodromy representation of the fundamental
group. Now $\Pp^{1}(\C)\smallsetminus S$ is a sphere with $s$ punctures, whose
fundamental group is free of rank $s-1$:
\begin{equation}\label{eq:pi1}
  \pi_{1}\bigl(\Pp^{1}(\C)\smallsetminus S,\ast\bigr)
  \;=\;\Bigl\langle \gamma_{1},\dots,\gamma_{s}\ \Bigm|\
  \gamma_{1}\gamma_{2}\cdots\gamma_{s}=1\Bigr\rangle .
\end{equation}
The decisive point is that this group depends only on $s$, not on the positions of the
points of $S$: the moduli space of $s$-point configurations on $\Pp^{1}$ modulo
$\mathrm{PGL}_{2}$ has dimension $s-3$. For $s\leq 3$ that dimension is $\leq 0$: any
three distinct points can be moved to $\{0,1,\infty\}$ by a M\"obius
transformation---uniquely once the points are ordered, and in one of $3!=6$ ways if
they are not (Lemma~\ref{lem:mobius-3pts})---so \emph{the branch data carry no
continuous parameters at all}. A degree-$d$ covering branched over three points is
therefore nothing more than a triple of permutations
\begin{equation}\label{eq:triple}
  \sigma_{0}\sigma_{1}\sigma_{\infty}=1\quad\text{in }\Sym_{d},
  \qquad\langle\sigma_{0},\sigma_{1}\rangle\ \text{transitive},
\end{equation}
taken up to simultaneous conjugation---a finite, purely combinatorial datum. For $s=4$
the configuration has one modulus, the cross-ratio, and the coverings acquire a
continuous family; the rigidity is lost.

This is the mechanism behind the easy half of Belyi's theorem: something rigid and finite
cannot move under $\GQ$ except within a finite orbit, and finite orbits mean number
fields. The whole of Section~\ref{sec:descent} is the technology for turning ``finite
orbit'' into ``defined over a number field''.

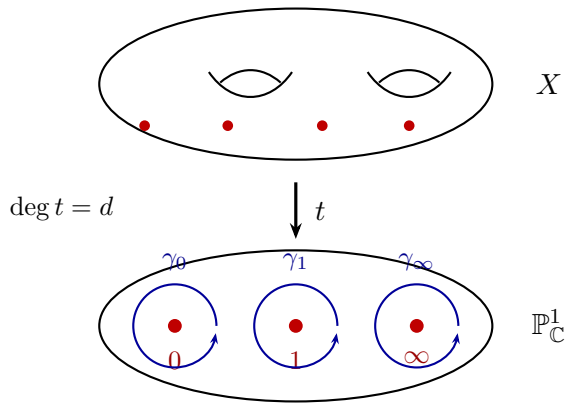
\begin{figure}[t]
\centering
\begin{tikzpicture}[scale=1.0,
  dot/.style={circle,fill,inner sep=1.5pt},
  sheet/.style={draw,thick,rounded corners=10pt}]
\draw[thick] (0,0) ellipse (2.6 and 1.0);
\node at (3.35,0) {$\Pp^{1}_{\C}$};
\foreach \x/\lab in {-1.6/0, 0/1, 1.6/\infty}{
  \fill[red!75!black] (\x,0) circle (2.6pt);
  \node[below=3pt,red!60!black] at (\x,-0.1) {\small $\lab$};
}
\draw[-{Stealth[length=4pt]},blue!60!black,thick]
  (-1.6,0) ++(0.55,0) arc (0:350:0.55);
\node[blue!55!black] at (-1.6,0.85) {\small $\gamma_{0}$};
\draw[-{Stealth[length=4pt]},blue!60!black,thick]
  (0,0) ++(0.55,0) arc (0:350:0.55);
\node[blue!55!black] at (0,0.85) {\small $\gamma_{1}$};
\draw[-{Stealth[length=4pt]},blue!60!black,thick]
  (1.6,0) ++(0.55,0) arc (0:350:0.55);
\node[blue!55!black] at (1.6,0.85) {\small $\gamma_{\infty}$};
\begin{scope}[yshift=3.2cm]
\draw[thick] (0,0) ellipse (2.6 and 1.0);
\draw[thick] (-1.15,0.16) .. controls (-0.85,-0.28) and (-0.35,-0.28) .. (-0.05,0.16);
\draw[thick] (-1.0,0.02) .. controls (-0.75,0.24) and (-0.45,0.24) .. (-0.2,0.02);
\draw[thick] (0.95,0.16) .. controls (1.25,-0.28) and (1.75,-0.28) .. (2.05,0.16);
\draw[thick] (1.1,0.02) .. controls (1.35,0.24) and (1.65,0.24) .. (1.9,0.02);
\node at (3.35,0) {$X$};
\foreach \x in {-2.0,-0.9,0.35,1.5}{\fill[red!75!black] (\x,-0.55) circle (2.0pt);}
\end{scope}
\draw[-{Stealth[length=6pt]},very thick] (0,1.9) -- (0,1.15)
  node[midway,right=3pt] {$t$};
\node at (-3.1,1.55) {\small $\deg t=d$};
\end{tikzpicture}
\caption{A Belyi map. Over the thrice-punctured sphere the covering is unramified and
therefore determined by the images of $\gamma_{0},\gamma_{1},\gamma_{\infty}$ in
$\Sym_{d}$; since $\gamma_{0}\gamma_{1}\gamma_{\infty}=1$, the covering is exactly a
transitive permutation triple \eqref{eq:triple} up to conjugacy. No continuous parameter
survives: three points on $\Pp^{1}$ have no moduli.}
\label{fig:belyi-cover}
\end{figure}

\subsection{What the theorem buys}\label{ssec:applications}

\emph{(1) Arithmetic of curves made combinatorial.} Belyi's theorem converts the
question ``is $X$ defined over $\Qbar$?''---a statement about the existence of models,
quantified over all embeddings and all coordinate systems---into the exhibition of a
single map with three branch points, which is then encoded by
\eqref{eq:triple}. Every curve over $\Qbar$ is thus labelled by a finite graph on a
surface (Section~\ref{sec:monodromy}), and the absolute Galois group $\GQ$ acts on the
set of these labels.

\emph{(2) A faithful action of $\GQ$.} The action of $\GQ$ on dessins is faithful, and
remains faithful on the much smaller set of \emph{plane trees}, that is, on dessins of
polynomial Belyi maps of genus zero \cite{Schneps1994,LandoZvonkin2004}. Since $\GQ$ is
notoriously resistant to direct description, an action on objects a child can draw is a
genuine handle on it; the search for combinatorial invariants that separate Galois orbits
is a large industry, and Section~\ref{sec:galois-dessins} computes two orbits completely.

\emph{(3) Uniformization by triangle groups.} A Belyi map $t\colon X\to\Pp^{1}$ with
ramification indices dividing $(p,q,r)$ over $0,1,\infty$ exhibits $X$ as a quotient of
the $(p,q,r)$-triangle group's uniformizing domain. Consequently every curve defined over
$\Qbar$ is uniformized by a subgroup of finite index in a triangle group---an arithmetic
constraint invisible from the equations. The Klein quartic
(Example~\ref{ex:klein}) is the celebrated case $(2,3,7)$, where the bound is Hurwitz's.

\emph{(4) Diophantine applications.} Belyi maps are the engine of Elkies' proof that the
$abc$ conjecture implies the effective Mordell conjecture \cite{Elkies1991}: a Belyi map
of degree $d$ on $X$ converts a rational point into an $abc$-triple of controlled height,
and the three-point condition is exactly what makes the $abc$ inequality applicable.
Quantitative forms of Belyi's theorem---bounds on $d$ in terms of the curve---are
therefore of direct Diophantine interest, which is the motivation for the degree bounds of
Section~\ref{sec:degree} and for the work of Khadjavi \cite{Khadjavi2002} and
Lit\-can\-u \cite{Litcanu2004}.

\emph{(5) Inverse Galois theory.} A covering branched over three points with monodromy
group $G$, defined over $\Q$, realizes $G$ (or a related group) as a Galois group over
$\Q(u)$ and hence, by Hilbert irreducibility, over $\Q$. The rigidity method of Belyi,
Fried, Matzat, Shih and Thompson is precisely the systematic exploitation of the
three-point rigidity of \S\ref{ssec:rigidity}; see \cite{MalleMatzat1999}.

\emph{(6) Modular curves as the universal example.} The map
$\lambda\mapsto j$ from the $\lambda$-line to the $j$-line is a Belyi map of degree $6$
(Example~\ref{ex:lambda-j}), and its ramification $(3,3\,|\,2,2,2\,|\,2,2,2)$ encodes the
elliptic points of $\mathrm{SL}_{2}(\Z)$. More generally the modular curves $X(N)\to X(1)$
are Belyi maps, which is why Belyi's theorem is sometimes summarized as: \emph{every curve
over $\Qbar$ is a covering of a modular curve, branched over the elliptic points and the
cusp.}
\section{Preliminaries: fields, curves, and Riemann surfaces}
\label{sec:prelim}

This section fixes the language and records the facts used later. Everything here is
standard; we include the statements we actually invoke, with proofs when they are short
and used in an essential way.

\subsection{Fields and Galois theory}\label{ssec:fields}

Let $L/K$ be a field extension. The \emph{Galois group} $\Gal(L/K)$ is the group of
automorphisms of $L$ fixing $K$ pointwise. For an intermediate field $E$ and a subgroup
$H\leq\Gal(L/K)$ we write
\begin{equation}\label{eq:galois-correspondence}
  \Gamma(E)=\{\alpha\in\Aut(L):\alpha|_{E}=\mathrm{id}\},
  \qquad
  \Phi(H)=L^{H}=\{x\in L:\alpha(x)=x\ \forall\alpha\in H\}.
\end{equation}
These maps are inclusion-reversing and satisfy $E\subseteq\Phi\Gamma(E)$ and
$H\subseteq\Gamma\Phi(H)$, but need not be mutually inverse: for
$L=\Q(\sqrt[3]{2})$ and $K=\Q$ one has $\Gal(L/K)=\{1\}$, since any automorphism sends
$\sqrt[3]{2}$ to a cube root of $2$ lying in $L\subset\R$, hence to itself; so
$\Phi\Gamma(\Q)=L\neq\Q$. The extension is not normal. This failure is not a pathology to
be avoided---it is the phenomenon that Section~\ref{sec:galois-dessins} exhibits
combinatorially, in the shape of a $\GQ$-orbit of three plane trees whose fields of
moduli are the three conjugate embedded cubic fields $\Q(a_{i})$, each abstractly
isomorphic to $\Q(\sqrt[3]{2})$.

We use repeatedly the following two facts about the huge group $\Aut(\C)$, for which see
\cite[\S1]{Kock2004}.

\begin{lemma}[{\cite[Lem.~1.2]{Kock2004}}]\label{lem:aut-C}
Let $K$ be a subfield of $\C$. Every automorphism of $K$ extends to an automorphism of
$\C$, and $\C^{\Aut(\C/K)}=K$ whenever $K$ is algebraically closed in $\C$; in general
$\C^{\Aut(\C/K)}$ is the algebraic closure of $K$ in $\C$. In particular
$\C^{\Aut(\C/\Qbar)}=\Qbar$.
\end{lemma}

\begin{lemma}[{\cite[Lem.~1.3, 1.4]{Kock2004}}]\label{lem:index}
Let $U\leq\Aut(\C)$ and let $V\leq U$ be of finite index. Then $\C^{V}/\C^{U}$ is a finite
extension, and $[\C^{V}:\C^{U}]\leq[U:V]$ when $V$ is normal in $U$ or $U$ is closed.
Moreover, if there is a finite extension $K/\C^{U}$ with $\Aut(\C/K)\subseteq U$, then $U$
is closed, i.e.\ $U=\Aut(\C/\C^{U})$.
\end{lemma}

Lemma~\ref{lem:index} is the engine that converts ``finite orbit'' into ``finite
extension'', and hence the technical heart of the easy direction of Belyi's theorem.

\subsection{Curves and fields of definition}\label{ssec:curves}

\begin{definition}\label{def:defined-over}
Let $K\subseteq\C$ be a subfield and $X$ a variety over $\C$. We say $X$ is
\emph{defined over $K$}, and call $K$ a \emph{field of definition}, if either of the
following equivalent conditions holds:
\begin{enumerate}[label=\textup{(\alph*)},leftmargin=2.4em]
\item there is a variety $X_{K}$ over $K$ with $X\cong X_{K}\times_{K}\C$ over $\C$;
\item $X$ admits a covering by affine charts cut out by polynomials with coefficients in
$K$.
\end{enumerate}
A morphism $t\colon X\to Y$ is defined over $K$ if $X$, $Y$ and $t$ all descend
compatibly.
\end{definition}

Being defined over $K$ is a property of the isomorphism class, not of the equations. Thus
$S_{1}\colon y^{2}=x^{3}-\pi^{3}$ appears to require the transcendental field $\Q(\pi)$,
but the substitution $(x,y)\mapsto(x/\pi,\,y/\pi^{3/2})$ identifies it with
$S_{2}\colon y^{2}=x^{3}-1$, so $S_{1}$ is defined over $\Q$. Recognizing such hidden
descents is precisely the problem Belyi's theorem solves.

\begin{lemma}[Three-point normalization]\label{lem:mobius-3pts}
Given ordered triples $(z_{1},z_{2},z_{3})$ and $(w_{1},w_{2},w_{3})$ of distinct points
of $\Pp^{1}$, there is a unique M\"obius transformation $M$ with $M(z_{i})=w_{i}$. If both
triples consist of $\Qbar$-rational points, $M$ has coefficients in $\Qbar$; if both lie
in $\Pp^{1}(\Q)$, in $\Q$.
\end{lemma}

\begin{proof}
The cross-ratio $(z,z_{1},z_{2},z_{3})=\frac{(z-z_{1})(z_{2}-z_{3})}{(z-z_{3})(z_{2}-z_{1})}$
is the unique M\"obius map sending $(z_{1},z_{2},z_{3})$ to $(0,1,\infty)$; composing the
one for $(z_{i})$ with the inverse of the one for $(w_{i})$ gives $M$, whose coefficients
are visibly polynomial in the $z_{i},w_{i}$. Uniqueness: a M\"obius transformation fixing
$0,1,\infty$ fixes $\Pp^{1}$ pointwise, since $az+b=z(cz+d)$ for three values of $z$
forces $c=b=0$, $a=d$.
\end{proof}

\begin{theorem}[Riemann--Hurwitz]\label{thm:RH}
Let $t\colon X\to Y$ be a covering of degree $d$ of smooth projective curves in
characteristic zero. Then
\begin{equation}\label{eq:RH}
  2\genus(X)-2\;=\;d\bigl(2\genus(Y)-2\bigr)\;+\;\sum_{P\in X}\bigl(\ram_{P}-1\bigr),
\end{equation}
where $\ram_{P}$ is the ramification index at $P$. For $Y=\Pp^{1}$ and $\Br(t)\subseteq
S$, writing $n_{s}=\#t^{-1}(s)$ for $s\in S$ and using $\sum_{P\in t^{-1}(s)}\ram_{P}=d$,
\begin{equation}\label{eq:RH-belyi}
  2\genus(X)-2\;=\;-2d\;+\;\sum_{s\in S}\bigl(d-n_{s}\bigr).
\end{equation}
In particular, for a Belyi map ($S=\{0,1,\infty\}$),
$2\genus(X)-2=d-(n_{0}+n_{1}+n_{\infty})$.
\end{theorem}

Identity \eqref{eq:RH-belyi} with $|S|=3$ is used so often below that we single it out:
\begin{equation}\label{eq:belyi-RH}
  \boxed{\;2g-2\;=\;d-\bigl(n_{0}+n_{1}+n_{\infty}\bigr)\;}
\end{equation}
where $g=\genus(X)$ and $n_{s}$ is the number of points above $s$. Every degree bound in
Section~\ref{sec:degree} is read off from \eqref{eq:belyi-RH}.

\subsection{The analytic--algebraic dictionary}\label{ssec:dictionary}

A \emph{Riemann surface} is a connected $2$-manifold with a maximal atlas of charts whose
transition functions are biholomorphic. The examples we need are the plane $\C$; the
Riemann sphere $\Pp^{1}=\C\cup\{\infty\}$ with the two charts $z$ and $1/z$; the
quotients $\C/\Z\cong\C^{*}$ and $\C/\Lambda$ (complex tori); and the unit disc $\Delta$.
The uniformization theorem states that every simply connected Riemann surface is
isomorphic to exactly one of $\Pp^{1}$, $\C$, $\Delta$; passing to quotients by the
fundamental group acting on the universal cover, every Riemann surface is a quotient of
one of these three by a discrete group of automorphisms.

We shall use the dictionary in the following form; see \cite{Miranda1995,GGD2012}.

\begin{theorem}[Riemann existence theorem and the GAGA dictionary]
\label{thm:dictionary}
The following categories are equivalent:
\begin{enumerate}[label=\textup{(\roman*)},leftmargin=2.4em]
\item compact Riemann surfaces and non-constant holomorphic maps;
\item smooth projective curves over $\C$ and non-constant morphisms;
\item field extensions of $\C$ of transcendence degree $1$, finitely generated, with
$\C$-algebra homomorphisms reversed.
\end{enumerate}
Under this equivalence a curve $X$ corresponds to its function field $\C(X)$, a morphism
$t\colon X\to\Pp^{1}$ to the finite extension $\C(X)/\C(t)$ of degree $\deg t$, and a
meromorphic function to a morphism to $\Pp^{1}$. Moreover, every finite topological
covering of $\Pp^{1}(\C)\smallsetminus S$ extends uniquely to a covering of curves
branched only in $S$.
\end{theorem}

The last sentence---Riemann's existence theorem proper---is what makes
\S\ref{ssec:rigidity} into mathematics: it guarantees that \emph{every} permutation
triple \eqref{eq:triple} is realized by an actual algebraic curve, so that the
combinatorial classification is not merely necessary but sufficient.
\section{Coverings, monodromy, and dessins}\label{sec:monodromy}

We now make precise the correspondence sketched in \S\ref{ssec:rigidity}, in the form used
throughout the rest of the article.

\subsection{Monodromy}\label{ssec:monodromy}

Let $t\colon X\to\Pp^{1}$ be a covering of degree $d$ with $\Br(t)\subseteq S=\{0,1,\infty\}$
and fix a base point $\ast\in\Pp^{1}(\C)\smallsetminus S$ together with a labelling
$t^{-1}(\ast)=\{1,\dots,d\}$. Analytic continuation of the sheets along loops defines the
\emph{monodromy representation}
\begin{equation}\label{eq:monodromy-rep}
  \rho\colon\pi_{1}\bigl(\Pp^{1}(\C)\smallsetminus S,\ast\bigr)\longrightarrow\Sym_{d},
  \qquad \sigma_{0}:=\rho(\gamma_{0}),\ \ \sigma_{1}:=\rho(\gamma_{1}),\ \
  \sigma_{\infty}:=\rho(\gamma_{\infty}),
\end{equation}
whose image $\Mon(t):=\rho(\pi_{1})$ is the \emph{monodromy group}. By \eqref{eq:pi1},
$\sigma_{0}\sigma_{1}\sigma_{\infty}=1$, and $\Mon(t)$ is transitive because $X$ is
connected. The cycle type of $\sigma_{s}$ records the ramification over $s$: the cycles of
$\sigma_{s}$ correspond bijectively to the points of $t^{-1}(s)$, a cycle of length $\ell$
corresponding to a point with $\ram_{P}=\ell$. The triple of cycle types,
\begin{equation}\label{eq:passport}
  \Pi(t)\;=\;\bigl\langle \text{type}(\sigma_{0}),\ \text{type}(\sigma_{1}),\
  \text{type}(\sigma_{\infty})\bigr\rangle ,
\end{equation}
is the \emph{passport} of $t$. Riemann--Hurwitz \eqref{eq:belyi-RH} is then the statement
that the genus is determined by the passport.

\begin{theorem}[Classification of Belyi maps]\label{thm:classification}
The assignment $t\mapsto(\sigma_{0},\sigma_{1},\sigma_{\infty})$ induces a bijection
between
\begin{enumerate}[label=\textup{(\roman*)},leftmargin=2.4em]
\item isomorphism classes of pairs $(X,t)$ with $X$ a curve over $\C$ and
$t\colon X\to\Pp^{1}$ a covering of degree $d$ with $\Br(t)\subseteq\{0,1,\infty\}$, and
\item transitive triples $(\sigma_{0},\sigma_{1},\sigma_{\infty})\in\Sym_{d}^{3}$ with
$\sigma_{0}\sigma_{1}\sigma_{\infty}=1$, modulo simultaneous conjugation.
\end{enumerate}
Here $(X_{1},t_{1})\cong(X_{2},t_{2})$ means an isomorphism $f\colon X_{1}\to X_{2}$ with
$t_{2}\circ f=t_{1}$. In particular there are only finitely many such pairs for each $d$.
\end{theorem}

\begin{proof}
The restriction \eqref{eq:unramified-restriction} is a degree-$d$ topological covering of
$\Pp^{1}(\C)\smallsetminus S$, and such coverings are classified by conjugacy classes of
index-$d$ subgroups of $\pi_{1}$, equivalently by transitive permutation representations
of degree $d$ up to conjugacy---that is, by triples as in (ii) via \eqref{eq:pi1}.
Conversely a topological covering of the punctured sphere carries a unique Riemann surface
structure making the projection holomorphic; it is compactified by filling in one point
for each cycle of each $\sigma_{s}$, and the resulting compact Riemann surface is
algebraic and the map a morphism, by Theorem~\ref{thm:dictionary}. Finiteness for fixed
$d$ is clear from (ii).
\end{proof}

\subsection{Dessins}\label{ssec:dessins}

Theorem~\ref{thm:classification} is already combinatorial, but the combinatorics becomes
visible when drawn.

\begin{definition}\label{def:dessin}
Let $t\colon X\to\Pp^{1}$ be a Belyi map. Its \emph{dessin} is the bipartite graph
$\mathcal{D}(t)\subset X$ with
\begin{itemize}[leftmargin=1.6em]
\item black vertices $t^{-1}(0)$, white vertices $t^{-1}(1)$,
\item edges the connected components of $t^{-1}\bigl((0,1)\bigr)$,
\end{itemize}
that is, $\mathcal{D}(t)=t^{-1}\bigl([0,1]\bigr)$ with its induced embedding in $X$. A
\emph{dessin d'enfant} is, abstractly, a connected bipartite graph embedded in a compact
oriented surface whose complementary faces are discs.
\end{definition}

The dictionary reads as follows, and is immediate from Theorem~\ref{thm:classification}
once one observes that $\sigma_{0}$ (resp.\ $\sigma_{1}$) rotates the edges around a
black (resp.\ white) vertex counterclockwise, while $\sigma_{\infty}$ rotates the edges
around a face.

\begin{proposition}\label{prop:dessin-dictionary}
Let $t$ be a Belyi map of degree $d$ on a curve of genus $g$. Then $\mathcal{D}(t)$ has
exactly $d$ edges; the black (resp.\ white) vertex degrees are the ramification indices
over $0$ (resp.\ over $1$); the faces of $\mathcal{D}(t)$ in $X$ correspond to the points
of $t^{-1}(\infty)$, a point of index $\ell$ giving a face bounded by $2\ell$ edge-sides;
and Euler's formula
$\#V-\#E+\#F=2-2g$ is exactly \eqref{eq:belyi-RH}. Conversely every dessin arises from a
unique Belyi map. Two Belyi maps are isomorphic if and only if their dessins are
isomorphic as embedded graphs.
\end{proposition}

\begin{proof}
Only the last assertions need comment. Given the embedded graph, the cyclic order of edges
at each vertex determines $\sigma_{0},\sigma_{1}$, hence the triple, hence $(X,t)$ by
Theorem~\ref{thm:classification}; and the identity $\#V-\#E+\#F=2-2g$ reads
$(n_{0}+n_{1})-d+n_{\infty}=2-2g$, which is \eqref{eq:belyi-RH}.
\end{proof}

Two special cases matter later. If $X$ has genus $0$ and $t$ is a \emph{polynomial}, then
$t^{-1}(\infty)=\{\infty\}$ is a single totally ramified point, so $\mathcal{D}(t)$ has
exactly one face; a connected graph with $d$ edges and one face on the sphere has
$d+1$ vertices and is a \emph{plane tree}. Conversely every plane tree with $d$ edges is
the dessin of a degree-$d$ polynomial Belyi map, unique up to affine change of variable.
This is the case in which Section~\ref{sec:galois-dessins} computes.

\begin{definition}\label{def:clean}
A Belyi map $t$ is \emph{pre-clean} if every ramification index over $1$ is at most $2$,
and \emph{clean} if every ramification index over $1$ equals $2$. Equivalently, the dessin
of a clean Belyi map has all white vertices of degree $2$; erasing them turns
$\mathcal{D}(t)$ into an ordinary (non-bipartite) graph with $d/2$ edges.
\end{definition}

\begin{proposition}[Cleaning]\label{prop:cleaning}
If $t$ is a Belyi map on $X$, then $4t(1-t)$ is a clean Belyi map on $X$, of degree
$2\deg t$. Consequently $X$ is defined over $\Qbar$ if and only if it admits a clean Belyi
map.
\end{proposition}

\begin{proof}
Let $q(w)=4w(1-w)=P_{1,1}(w)$. Then $q'(w)=4-8w$ vanishes only at $w=1/2$, where
$q=1$, so $\Br(q)\subseteq\{1,\infty\}$ and $q$ maps $\{0,1,\infty\}$ into
$\{0,\infty\}$. Hence $\Br(q\circ t)\subseteq \Br(q)\cup q(\Br(t))\subseteq\{0,1,\infty\}$,
and $q\circ t$ is a Belyi map. Over $1$ the fibre of $q\circ t$ is
$t^{-1}(1/2)$, which is unramified for $t$ (as $1/2\notin\Br(t)$) while $q$ is ramified of
order $2$ there; so every point over $1$ has index exactly $2$, i.e.\ $q\circ t$ is clean.
\end{proof}
\section{Galois descent and fields of moduli}\label{sec:descent}

The hard direction of Belyi's theorem produces, from a three-point covering, a
\emph{finite} set of Galois conjugates, and must then descend the curve to the fixed field.
This section develops the descent machinery in the form we need. Because we work with
smooth projective curves, where birational equivalence is isomorphism, all descent can be
carried out on function fields; this keeps the argument elementary and self-contained.

\subsection{Semilinear descent}\label{ssec:semilinear}

Let $L$ be a field, $G\leq\Aut(L)$ a finite subgroup, and $K=L^{G}$; by Artin's theorem
$L/K$ is Galois with group $G$ and $[L:K]=|G|$.

\begin{definition}\label{def:G-structure}
Let $W$ be an $L$-vector space. A map $r\colon W\to W$ is \emph{$\sigma$-linear}
($\sigma\in G$) if it is additive and $r(aw)=\sigma(a)r(w)$ for $a\in L$. A
\emph{$G$-structure} on $W$ is a family $(r_{\sigma})_{\sigma\in G}$ of $\sigma$-linear
bijections with $r_{1}=\mathrm{id}_{W}$ and $r_{\sigma}\circ r_{\tau}=r_{\sigma\tau}$.
\end{definition}

\begin{lemma}[Linear independence of characters]\label{lem:dedekind}
Let $\sigma_{1},\dots,\sigma_{n}$ be distinct homomorphisms from a group $\Gamma$ to
$L^{\times}$ and $w_{1},\dots,w_{n}$ elements of an $L$-vector space $W$. If
$\sum_{i}\sigma_{i}(\gamma)w_{i}=0$ for all $\gamma\in\Gamma$, then all $w_{i}=0$.
\end{lemma}

\begin{proof}
Induct on $n$; the case $n=1$ is clear since $\sigma_{1}(\gamma)\in L^{\times}$. For
$n>1$ pick $\delta$ with $\sigma_{1}(\delta)\neq\sigma_{n}(\delta)$. Subtracting
$\sigma_{n}(\delta)^{-1}$ times the relation evaluated at $\delta\gamma$ from the relation
at $\gamma$ eliminates the $n$-th term and yields
$\sum_{i<n}\bigl(1-\sigma_{n}(\delta)^{-1}\sigma_{i}(\delta)\bigr)\sigma_{i}(\gamma)w_{i}=0$
for all $\gamma$. By induction each coefficient vanishes; for $i=1$ the scalar is nonzero,
so $w_{1}=0$, and symmetrically all $w_{i}=0$ for $i<n$; then $\sigma_{n}(\gamma)w_{n}=0$
gives $w_{n}=0$.
\end{proof}

\begin{lemma}[Nonvanishing of the trace]\label{lem:trace}
Let $W$ carry a $G$-structure and set $T(w)=\sum_{\sigma\in G}r_{\sigma}(w)$. Then
$T(W)\subseteq W^{G}:=\{w: r_{\sigma}(w)=w\ \forall\sigma\}$, and for every $w\neq 0$
there is $a\in L$ with $T(aw)\neq 0$. In particular $W\neq 0$ implies $W^{G}\neq 0$.
\end{lemma}

\begin{proof}
For $\tau\in G$, $r_{\tau}(T(w))=\sum_{\sigma}r_{\tau\sigma}(w)=T(w)$, so
$T(W)\subseteq W^{G}$. If $T(aw)=0$ for all $a\in L$, then
$\sum_{\sigma}\sigma(a)\,r_{\sigma}(w)=0$ for all $a\in L^{\times}$; by
Lemma~\ref{lem:dedekind} applied to the distinct characters $\sigma|_{L^{\times}}$ we get
$r_{\sigma}(w)=0$ for every $\sigma$, and $r_{1}=\mathrm{id}$ gives $w=0$.
\end{proof}

\begin{lemma}[Galois descent]\label{lem:galois-descent}
Let $W$ be an $L$-vector space with a $G$-structure and $K=L^{G}$. Then the canonical map
\begin{equation*}
  r\colon L\otimes_{K}W^{G}\longrightarrow W,\qquad a\otimes w\longmapsto aw,
\end{equation*}
is an isomorphism of $L$-vector spaces. Equivalently, every $K$-basis of $W^{G}$ is an
$L$-basis of $W$.
\end{lemma}

\begin{proof}
\emph{Injectivity.} It suffices to show that $K$-linearly independent elements
$w_{1},\dots,w_{n}\in W^{G}$ remain $L$-linearly independent in $W$. If not, choose a
nontrivial relation $\sum_{i=1}^{n}a_{i}w_{i}=0$ with $n$ minimal; then all $a_{i}\neq 0$,
$n\geq 2$, and after scaling $a_{n}=1$. Applying $r_{\sigma}$ and using
$r_{\sigma}(w_{i})=w_{i}$ gives $\sum_{i}\sigma(a_{i})w_{i}=0$; subtracting,
$\sum_{i=1}^{n-1}\bigl(a_{i}-\sigma(a_{i})\bigr)w_{i}=0$, a shorter relation, hence
trivial by minimality. So $a_{i}=\sigma(a_{i})$ for all $\sigma$, i.e.\ $a_{i}\in K$,
contradicting $K$-independence.

\emph{Surjectivity.} The image $W_{0}=r(L\otimes_{K}W^{G})$ is an $L$-subspace stable
under every $r_{\sigma}$, because $r_{\sigma}(aw)=\sigma(a)w$ for $w\in W^{G}$. Hence the
quotient $\overline{W}=W/W_{0}$ inherits a $G$-structure, and every element of
$\overline{W}$ has trace $\overline{T(w)}=\overline{0}$, since $T(w)\in W^{G}\subseteq
W_{0}$. By Lemma~\ref{lem:trace}, $\overline{W}=0$, i.e.\ $W_{0}=W$.
\end{proof}

\begin{theorem}[Weil's descent criterion for curves]\label{thm:weil}
Let $L$ be a field, $G\leq\Aut(L)$ finite, $K=L^{G}$, and let $X$ be a smooth projective
curve over $L$. Suppose that for every $\sigma\in G$ there is an isomorphism
$f_{\sigma}\colon X^{\sigma}\to X$ of curves over $L$ such that the induced maps on
function fields, $f_{\sigma}^{*}\colon L(X)\to L(X^{\sigma})=L(X)$, define a
$G$-structure:
\begin{equation}\label{eq:cocycle}
  r_{\sigma}:=\bigl(f_{\sigma}^{-1}\bigr)^{*}\ \text{is $\sigma$-linear},
  \qquad r_{\sigma}\circ r_{\tau}=r_{\sigma\tau},\qquad r_{1}=\mathrm{id}.
\end{equation}
Then there is a curve $X_{K}$ over $K$ with $X_{K}\times_{K}L\cong X$.
\end{theorem}

\begin{proof}
Put $W=L(X)$ with the $G$-structure \eqref{eq:cocycle}, so each $r_{\sigma}$ is a
$\sigma$-linear field automorphism. Then $V:=W^{G}$ is a subfield of $W$ containing
$K$, and Lemma~\ref{lem:galois-descent} gives $L\otimes_{K}V\xrightarrow{\ \sim\ }W$. In
particular $V$ has transcendence degree $1$ over $K$; it is finitely generated over $K$
because $W$ is finitely generated over $L$ and, by the isomorphism, generators of $W$ as
an $L$-field can be replaced by finitely many elements of $V$. Let $X_{K}$ be the smooth
projective model of $V/K$, which exists and is unique. Then $K(X_{K})=V$ and
$L(X_{K}\times_{K}L)=L\otimes_{K}V=L(X)$, whence $X_{K}\times_{K}L\cong X$, since for
smooth projective curves an isomorphism of function fields over $L$ comes from a unique
isomorphism of curves.
\end{proof}

\begin{remark}\label{rem:cocycle-twist}
Condition \eqref{eq:cocycle} is the cocycle condition of Weil descent, and the twist is
essential: the naive requirement $f_{\sigma\tau}=f_{\sigma}\circ f_{\tau}$ is not
$\sigma$-linear-compatible. In terms of the maps $f_{\sigma}$ themselves the condition
reads $f_{\sigma\tau}=f_{\sigma}\circ f_{\tau}^{\sigma}$, where $f_{\tau}^{\sigma}$ is the
$\sigma$-conjugate of $f_{\tau}$. Verifying it is the crux of every descent argument; the
point of Theorem~\ref{thm:rigid-descent} below is that in the rigid case it is
\emph{automatic}.
\end{remark}

\subsection{Fields of moduli}\label{ssec:moduli-field}

\begin{definition}\label{def:moduli-field}
Let $X$ be a curve over $\C$. For $\sigma\in\Aut(\C)$ let $X^{\sigma}$ be the curve
obtained by applying $\sigma$ to the coefficients of defining equations. Set
\begin{equation*}
  U(X)=\{\sigma\in\Aut(\C): X^{\sigma}\cong X \text{ over }\C\},
  \qquad
  M(X)=\C^{U(X)} .
\end{equation*}
$U(X)$ is a subgroup of $\Aut(\C)$ and $M(X)$ is the \emph{field of moduli} of $X$.
Similarly, for a covering $t\colon X\to\Pp^{1}_{\C}$ let $U(X,t)$ consist of those
$\sigma$ for which there is an isomorphism $f_{\sigma}\colon X^{\sigma}\to X$ with
\begin{equation}\label{eq:covering-conjugate}
  \begin{tikzcd}[column sep=2.6em,row sep=2.0em]
    X^{\sigma} \arrow[r,"f_{\sigma}"] \arrow[d,"t^{\sigma}"'] & X \arrow[d,"t"]\\
    (\Pp^{1}_{\C})^{\sigma} \arrow[r,"\mathrm{Proj}(\sigma)"'] & \Pp^{1}_{\C}
  \end{tikzcd}
\end{equation}
commutative, where $\mathrm{Proj}(\sigma)$ is the automorphism induced by $\sigma$ on
$\Pp^{1}_{\C}=\operatorname{Proj}\C[T_{0},T_{1}]$. The \emph{field of moduli of the
covering} is $M(X,t)=\C^{U(X,t)}$.
\end{definition}

Every field of definition contains the field of moduli, and $M(X)\subseteq M(X,t)$ since
$U(X,t)\subseteq U(X)$. The converse---is the field of moduli a field of definition?---is
the central difficulty, and the answer is: not always, but always up to a finite
extension (Theorem~\ref{thm:moduli-covering}), and always when the covering has no
automorphisms (Theorem~\ref{thm:rigid-descent}, which builds on the former).

\begin{theorem}[Field of moduli of a covering; {\cite[Thm.~2.2]{Kock2004}}]
\label{thm:moduli-covering}
Let $X$ be a curve over $\C$ and $t\colon X\to\Pp^{1}_{\C}$ a covering. Then $X$ and $t$
are defined over a finite extension of $M(X,t)$. If $t$ is a Galois covering, they are
defined over $M(X,t)$ itself.
\end{theorem}

\begin{proof}
The strategy: use Riemann--Roch to manufacture a generator $z$ of the function field over
$\C(t)$, normalized so rigidly---by its pole, its leading coefficient, and its constant
term---that the semilinear action coming from the field of moduli has no choice but to fix
it; the coefficients of its minimal polynomial then land in the desired fixed field.
Choose a $\Q$-rational point $Q\in\Pp^{1}(\Q)$ with $Q\notin\Br(t)$ and a point
$P\in t^{-1}(Q)$. Applying Riemann--Roch to the divisor $(g+1)P$, where $g=\genus(X)$,
gives $\ell\bigl((g+1)P\bigr)\geq g+2-g=2$, so there is a non-constant meromorphic
function $z\in\C(X)$ whose only pole is at $P$. Choose such a $z$ with the pole order
$m=-\ord_{P}(z)$ minimal.

We claim $\C(X)=\C(t,z)$. Indeed, let $Y$ be the curve with function field $\C(t,z)$, so
that $t$ and $z$ factor through a morphism $\varphi\colon X\to Y$. Since $Q\notin\Br(t)$,
$\varphi$ is unramified at $P$; but $P$ is the unique pole of $z$, so $z$, viewed on $Y$,
has a unique pole below $P$ and $\varphi$ is totally ramified at $P$. An index that is
simultaneously $1$ and $\deg\varphi$ forces $\deg\varphi=1$, proving the claim.

Now normalize $z$. The space
$V=\{h\in\C(X):\ord_{P}(h)\geq -m,\ \ord_{R}(h)\geq0\ \forall R\neq P\}$
equals $\C\oplus\C z$ by minimality of $m$: two elements with pole order exactly $m$
differ, after scaling, by an element of smaller pole order, hence by a constant. Since
$Q$ is unramified, $t-Q$ is a local parameter at $P$, and there is a \emph{unique}
$z\in V$ whose Laurent expansion in $t-Q$ has leading coefficient $1$ in degree $-m$ and
vanishing constant term. Fix this $z$.

Let $U(X,t,P)\subseteq U(X,t)$ be the subgroup of those $\sigma$ for which $f_{\sigma}$ can
be chosen with $f_{\sigma}(P^{\sigma})=P$. Since $\Aut(X,t)$ acts freely on the fibre
$t^{-1}(Q)$ (an automorphism fixing a point of an unramified fibre and commuting with $t$
is the identity), such $f_{\sigma}$ is unique, and $\sigma\mapsto f_{\sigma}^{*}$ is a
semilinear action of $U(X,t,P)$ on $\C(X)$ fixing $t$. The subgroup $U(X,t,P)$ is the
stabilizer of $[P]$ for the action of $U(X,t)$ on the finite set
$t^{-1}(Q)/\Aut(X,t)$, hence has finite index in $U(X,t)$; if $t$ is Galois then
$\Aut(X,t)$ is transitive on $t^{-1}(Q)$ and $U(X,t,P)=U(X,t)$.

The three normalizing conditions defining $z$ are preserved by this action, so $z$---and
hence the minimal polynomial of $z$ over $\C(t)$---is invariant under $U(X,t,P)$. Its
coefficients therefore lie in $\C^{U(X,t,P)}=:k$, and $\C(X)=\C(t,z)$ is defined over
$k$, as is $t$. By Lemma~\ref{lem:index}, $k$ is a finite extension of
$\C^{U(X,t)}=M(X,t)$, and $k=M(X,t)$ in the Galois case.
\end{proof}

\begin{theorem}[Rigid descent]\label{thm:rigid-descent}
Let $t\colon X\to\Pp^{1}_{\C}$ be a covering whose automorphism group
\begin{equation*}
  \Aut(X,t)=\{f\in\Aut(X): t\circ f=t\}
\end{equation*}
is trivial. Then $X$ and $t$ are defined over the field of moduli $M(X,t)$.
\end{theorem}

\begin{proof}
Write $U=U(X,t)$ and $M=\C^{U}=M(X,t)$. The group $U$ is a subgroup of the enormous group
$\Aut(\C)$, so finite Galois descent does not apply to it directly; the proof proceeds in
three steps, reducing to the finite case.

\emph{Step 1: $U$ is closed, i.e.\ $U=\Aut(\C/M)$.} By
Theorem~\ref{thm:moduli-covering}---whose proof is independent of the present
theorem---the pair $(X,t)$ is defined over some \emph{finite} extension $L$ of $M$ inside
$\C$. Fix a model $(X_{L},t_{L})$ over $L$. Every $\sigma\in\Aut(\C/L)$ then fixes the
defining equations of the model, so $(X^{\sigma},t^{\sigma})\cong(X,t)$ canonically and
$\Aut(\C/L)\subseteq U$. By the closedness criterion of Lemma~\ref{lem:index}, $U$ is
closed: $U=\Aut(\C/M)$. In particular, \emph{every} automorphism of $\C$ fixing $M$
pointwise lies in $U$.

\emph{Step 2: reduction to a finite Galois group.} The idea of this step: the descent
datum a priori lives over the huge group $\Aut(\C)$, but the covering itself lives over a
finite extension, so the datum should too---uniqueness of the comparison isomorphisms
forces their coefficients down into that finite extension, after which only a finite
Galois group is in play. Concretely: let $L'/M$ be the Galois closure of
$L/M$ inside $\C$, a finite Galois extension with group $G=\Gal(L'/M)$, and let
$X'=X_{L}\times_{L}L'$, $t'=t_{L}\times_{L}L'$, a model of $(X,t)$ over $L'$. Fix
$\sigma\in G$ and extend it to $\tilde\sigma\in\Aut(\C)$ (Lemma~\ref{lem:aut-C}); by
Step 1, $\tilde\sigma\in U$, so there is an isomorphism
$f_{\tilde\sigma}\colon X^{\tilde\sigma}\to X$ over $\C$ compatible with $t$ as in
\eqref{eq:covering-conjugate}, and since $\tilde\sigma$ preserves $L'$ ($L'/M$ being
normal), $(X')^{\sigma}$ is again a curve over $L'$. The isomorphism is in fact defined
over $L'$, in three short steps: (a) by the uniqueness proved in Step 3 below,
$f_{\tilde\sigma}$ is invariant under conjugation by $\Aut(\C/L')$, which fixes the
models and their conjugates; (b) its coefficients therefore lie in the fixed field
$\C^{\Aut(\C/L')}$, the algebraic closure of $L'$ in $\C$ (Lemma~\ref{lem:aut-C})---in
particular they are \emph{algebraic} over $L'$; (c) an algebraic element fixed by
$\Gal(\overline{L'}/L')$ lies in $L'$, and every element of that Galois group is the
restriction of an automorphism of $\C$ over $L'$ (Lemma~\ref{lem:aut-C} again), so the
coefficients lie in $L'$. Thus for each $\sigma\in G$ there is a unique isomorphism
$f_{\sigma}\colon (X')^{\sigma}\to X'$ over $L'$ compatible with $t'$.

\emph{Step 3: uniqueness, cocycle condition, and finite descent.} The isomorphism
$f_{\sigma}$ is unique: if $f_{\sigma}'$ were another, then
$f_{\sigma}^{-1}\circ f_{\sigma}'\in\Aut\bigl((X')^{\sigma},(t')^{\sigma}\bigr)$, which is
trivial because $\Aut(X,t)$ is trivial and triviality is preserved under base change and
conjugation. Uniqueness forces the cocycle condition: both $f_{\sigma\tau}$ and
$f_{\sigma}\circ f_{\tau}^{\sigma}$ fit into the diagram
\eqref{eq:covering-conjugate} for $\sigma\tau$, hence coincide, which is
\eqref{eq:cocycle} for the finite group $G=\Gal(L'/M)$ acting on the function field
$W=L'(X')$. Theorem~\ref{thm:weil}---finite Galois descent, fully proved
above---now produces a curve $X_{M}$ over $M=(L')^{G}$ with $X_{M}\times_{M}L'\cong X'$;
and since each $f_{\sigma}$ commutes with $t'$, the element $t'\in W$ is fixed by the
semilinear action, hence lies in the descended function field $W^{G}=M(X_{M})$, so $t$
descends to $M$ as well.
\end{proof}

\begin{remark}\label{rem:rigid-descent-structure}
The logical order matters: Theorem~\ref{thm:moduli-covering} (finite extension,
unconditional) is proved first and is an input to Theorem~\ref{thm:rigid-descent}
(descent to $M$ itself, under rigidity), through the closedness of $U(X,t)$ in Step 1.
Without some such input, descent along the infinite group $\Aut(\C)$ is not covered by
Lemma~\ref{lem:galois-descent}, whose proof uses the finiteness of $G$ in an essential way
(the trace of Lemma~\ref{lem:trace} is a finite sum). This is the standard subtlety in
``field of moduli versus field of definition'' arguments; see
\cite{Weil1956,JonesWolfart2016}.
\end{remark}

\begin{remark}\label{rem:moduli-genus01}
For the curve alone the situation is better in low genus, though it requires care. A curve
of genus $0$ is a smooth conic; its field of moduli is a field of definition, but the
conic need not be $\Pp^{1}$ over that field---it may be a nonsplit conic, split by a
quadratic extension. A curve $X$ of genus $1$ has a $j$-invariant satisfying
$j(X^{\sigma})=\sigma(j)$, so $M(X)=\Q(j)$ by Lemma~\ref{lem:aut-C}; but here a second
subtlety intervenes, and it is the one most often elided. As an \emph{unpointed} genus-one
curve, $X$ is a torsor under its Jacobian $E$, and it need not carry a rational point over
$M(X)$: the elliptic curve with $j$-invariant $j$ realizing the standard Weierstrass model
over $\Q(j)$ is $\mathrm{Jac}(X)=E$, not $X$ itself. Thus $M(X)=\Q(j)$ \emph{is} a field of
definition of $X$ as a curve (genus one is not obstructed), but the specific representative
one writes down over $\Q(j)$ is the Weierstrass model of $\mathrm{Jac}(X)$; $X$ is a
possibly-nontrivial element of the Weil--Ch\^atelet group $H^{1}(\Q(j),E)$, and coincides
with its Jacobian only when it has a rational point. The pointed curve $(X,O)$, by contrast,
is always the Weierstrass curve over $\Q(j)$. For $g\geq 2$ one has $\#\Aut(X)\leq 84(g-1)$ by Hurwitz's theorem, so the
descent obstruction is at least finite in nature; but genuine counterexamples to
``field of moduli $=$ field of definition'' exist for curves and for coverings, which is
why Theorem~\ref{thm:moduli-covering} claims only a finite extension. Theorem
\ref{thm:rigid-descent} isolates the hypothesis---triviality of $\Aut(X,t)$---under which
the obstruction vanishes identically.
\end{remark}
\section{Proof of Belyi's theorem}\label{sec:proof}

We now prove Theorem~\ref{thm:belyi-intro}. The two implications have entirely different
characters. The implication ``three branch points $\Rightarrow$ defined over $\Qbar$'' is
a rigidity-and-descent argument: it uses the finiteness of
Theorem~\ref{thm:classification} and the machinery of Section~\ref{sec:descent}. The
converse is Belyi's algorithm: an explicit, constructive reduction of the branch locus,
and the source of all the explicit maps studied later.

\subsection{Three points imply arithmetic}\label{ssec:easy}

\begin{proposition}[Finiteness]\label{prop:finiteness}
Let $S\subset\Pp^{1}(\C)$ be finite and $d\geq1$. Then there are only finitely many
isomorphism classes of pairs $(X,t)$ with $X$ a curve over $\C$ and
$t\colon X\to\Pp^{1}_{\C}$ a covering of degree $d$ with $\Br(t)\subseteq S$.
\end{proposition}

\begin{proof}
By Theorem~\ref{thm:dictionary} the pair $(X,t)$ is determined up to isomorphism by the
topological covering \eqref{eq:unramified-restriction}: two pairs whose restrictions are
isomorphic as coverings of $\Pp^{1}(\C)\smallsetminus S$ have a biholomorphism between the
punctured surfaces commuting with the projections, which extends over the punctures by
Riemann's removable singularity theorem and is algebraic by
Theorem~\ref{thm:dictionary}. Degree-$d$ coverings of
$\Pp^{1}(\C)\smallsetminus S$ correspond to index-$d$ subgroups of
$\pi_{1}(\Pp^{1}(\C)\smallsetminus S)$ up to conjugacy; by \eqref{eq:pi1} that group is
free of rank $|S|-1$, in particular finitely generated, and a finitely generated group has
only finitely many subgroups of each finite index (a subgroup of index $d$ is the
stabilizer of a point in a transitive action on $d$ letters, of which there are at most
$(d!)^{|S|-1}$).
\end{proof}

\begin{corollary}\label{cor:moduli-number-field}
Let $t\colon X\to\Pp^{1}_{\C}$ be a covering whose critical values are $K$-rational for a
subfield $K\subseteq\C$. Then $M(X,t)$ is contained in a finite extension of $K$. In
particular, if $\Br(t)\subseteq\{0,1,\infty\}$ then $M(X,t)$ is a number field.
\end{corollary}

\begin{proof}
For $\sigma\in\Aut(\C/K)$ the conjugate $t^{\sigma}\colon X^{\sigma}\to\Pp^{1}_{\C}$ has
the same degree and the same critical values as $t$, the latter because they are
$K$-rational and $\sigma$ fixes $K$. By Proposition~\ref{prop:finiteness} the orbit of the
isomorphism class of $(X,t)$ under $\Aut(\C/K)$ is finite, so its stabilizer has finite
index in $\Aut(\C/K)$; the stabilizer is contained in $U(X,t)$ by definition. Thus
$\Aut(\C/K)\cap U(X,t)$ has finite index in $\Aut(\C/K)$, and
Lemmas~\ref{lem:aut-C} and~\ref{lem:index} give that
$M(X,t)=\C^{U(X,t)}$ is contained in a finite extension of $\C^{\Aut(\C/K)}=K$. For
$K=\Q$ this is a number field.
\end{proof}

\begin{proof}[Proof of the implication ``$\Leftarrow$'' of Theorem~\ref{thm:belyi-intro}]
Suppose $t\colon X\to\Pp^{1}_{\C}$ is non-constant with at most three branch points. By
Lemma~\ref{lem:mobius-3pts} we may compose with a M\"obius transformation and assume
$\Br(t)\subseteq\{0,1,\infty\}$. By Corollary~\ref{cor:moduli-number-field}, $M(X,t)$ is a
number field, and by Theorem~\ref{thm:moduli-covering} the pair $(X,t)$ is defined over a
finite extension $L$ of $M(X,t)$. Then
$[L:\Q]=[L:M(X,t)]\cdot[M(X,t):\Q]<\infty$, so $L$ is a number field and $X$ is defined
over it.
\end{proof}

\subsection{Belyi's algorithm}\label{ssec:algorithm}

The converse is constructive. Start with any curve $X$ over $\Qbar$ and any non-constant
morphism $t_{0}\colon X\to\Pp^{1}$ defined over $\Qbar$---for instance a coordinate
function. Its branch locus is a finite set of algebraic points, and the task is to push
that set into $\{0,1,\infty\}$ by post-composition, without ever leaving $\Qbar$. This is
done in two stages, illustrated in Figure~\ref{fig:algorithm}:
\begin{enumerate}[label=\textup{Stage \arabic*.},leftmargin=4.2em]
\item \emph{Rationalization.} Push the branch locus into $\Pp^{1}(\Q)$, using
minimal polynomials; see Lemma~\ref{lem:rationalize}.
\item \emph{Concentration.} Push a finite set of rational points into
$\{0,1,\infty\}$, using the Belyi polynomials $P_{m,n}$; see
Lemma~\ref{lem:concentrate}.
\end{enumerate}
Both stages rely on the elementary composition rule for branch loci: for
$f,g$ non-constant morphisms of curves,
\begin{equation}\label{eq:branch-composition}
  \Br(g\circ f)\;\subseteq\;\Br(g)\cup g\bigl(\Br(f)\bigr),
\end{equation}
which follows from the chain rule $(g\circ f)'=(g'\circ f)\cdot f'$: a critical point of
$g\circ f$ is a critical point of $f$ or maps to one of $g$.

\begin{lemma}[Elimination of a critical value; {\cite[Lem.~4.6]{Kock2004}}]
\label{lem:rationalize}
Let $S\subset\Qbar$ be a finite set. There is a non-constant polynomial $p\in\Q[z]$ such
that $p(S)\subseteq\Q$ and $\Br(p)\subseteq\Pp^{1}(\Q)$.
\end{lemma}

\begin{proof}
Induct on $N=|S\smallsetminus\Q|$, the number of irrational points of $S$; we may assume
$S$ closed under $\GQ$. If $N=0$ take $p=z$. Otherwise let $p_{1}\in\Q[z]$ be the product
of the distinct minimal polynomials of the irrational elements of $S$, so
$\deg p_{1}=N$ and $p_{1}$ maps every irrational element of $S$ to $0$ and every rational
element to a rational number. As a map $\Pp^{1}\to\Pp^{1}$, $p_{1}$ is branched over
$\infty$ and over $p_{1}(Z)$ where $Z$ is the zero set of $p_{1}'$; since
$\deg p_{1}'=N-1$, the set $S_{1}:=p_{1}(Z)$ has at most $N-1$ elements, is closed under
$\GQ$ (as $p_{1}\in\Q[z]$), and hence contains at most $N-1$ irrational points. By
induction there is $p_{2}\in\Q[z]$ with $p_{2}(S_{1})\subseteq\Q$ and
$\Br(p_{2})\subseteq\Pp^{1}(\Q)$. Put $p=p_{2}\circ p_{1}$. Then $p(S)\subseteq
p_{2}(\Q)\subseteq\Q$, and by \eqref{eq:branch-composition}
\begin{equation*}
  \Br(p)\subseteq\Br(p_{2})\cup p_{2}\bigl(\Br(p_{1})\bigr)
  \subseteq\Pp^{1}(\Q)\cup p_{2}\bigl(S_{1}\cup\{\infty\}\bigr)\subseteq\Pp^{1}(\Q).\qedhere
\end{equation*}
\end{proof}

\begin{lemma}[Concentration; {\cite[Lem.~4.7]{Kock2004}}]\label{lem:concentrate}
Let $T\subset\Pp^{1}(\Q)$ be a finite set. There is a non-constant morphism
$q\colon\Pp^{1}\to\Pp^{1}$ defined over $\Q$ with $q(T)\subseteq\{0,1,\infty\}$ and
$\Br(q)\subseteq\{0,1,\infty\}$.
\end{lemma}

\begin{proof}
Induct on $r=|T|$. If $r\leq3$, a M\"obius transformation over $\Q$
(Lemma~\ref{lem:mobius-3pts}) does it. If $r>3$, compose first with a M\"obius
transformation over $\Q$ so that $\{0,1,\infty\}\subseteq T$ and some fourth point of $T$
lies strictly between $0$ and $1$; being rational it is $m/(m+n)$ with $m,n$ positive
integers. Consider the \emph{Belyi polynomial}
\begin{equation}\label{eq:belyi-poly}
  P_{m,n}(z)\;=\;\frac{(m+n)^{m+n}}{m^{m}\,n^{n}}\;z^{m}(1-z)^{n}\ \in\ \Q[z].
\end{equation}
Its derivative is
\begin{equation}\label{eq:belyi-poly-derivative}
  P_{m,n}'(z)\;=\;-\frac{(m+n)^{m+n+1}}{m^{m}n^{n}}\;z^{m-1}(1-z)^{n-1}
  \Bigl(z-\tfrac{m}{m+n}\Bigr),
\end{equation}
so the critical points in $\C$ are $0$, $1$ and $m/(m+n)$, with critical values
$0$, $0$ and
\begin{equation}\label{eq:pmn-value}
  P_{m,n}\!\left(\frac{m}{m+n}\right)
  =\frac{(m+n)^{m+n}}{m^{m}n^{n}}\cdot\frac{m^{m}}{(m+n)^{m}}\cdot\frac{n^{n}}{(m+n)^{n}}
  =1 .
\end{equation}
Hence $\Br(P_{m,n})\subseteq\{0,1,\infty\}$, and $P_{m,n}$ maps
$\{0,\,m/(m+n),\,1,\,\infty\}$ onto $\{0,1,\infty\}$. Therefore
$P_{m,n}(T)$ has at most $r-1$ elements, all rational, and the induction hypothesis
supplies $Q$ with $Q(P_{m,n}(T))\subseteq\{0,1,\infty\}$ and
$\Br(Q)\subseteq\{0,1,\infty\}$. Then $q=Q\circ P_{m,n}$ satisfies, by
\eqref{eq:branch-composition},
$\Br(q)\subseteq\Br(Q)\cup Q(\{0,1,\infty\})\subseteq\{0,1,\infty\}$ and
$q(T)\subseteq\{0,1,\infty\}$.
\end{proof}

\begin{proof}[Proof of the implication ``$\Rightarrow$'' of Theorem~\ref{thm:belyi-intro}]
Let $X$ be defined over a number field, hence over $\Qbar$, and let
$t_{0}\colon X\to\Pp^{1}$ be any non-constant morphism defined over $\Qbar$. Because $X$
and $t_{0}$ are defined over the algebraically closed field $\Qbar$, the critical values of
$t_{0}$ are $\Qbar$-rational: they are the images of the zeros of a differential form with
$\Qbar$-coefficients, hence solutions of polynomial equations over $\Qbar$. Put
$S=\Br(t_{0})\subset\Pp^{1}(\Qbar)$ and let $p$ be as in Lemma~\ref{lem:rationalize}, then
$T=\Br(p)\cup p(S)\subset\Pp^{1}(\Q)$ and let $q$ be as in Lemma~\ref{lem:concentrate}.
Then $t=q\circ p\circ t_{0}$ satisfies, by two applications of
\eqref{eq:branch-composition},
\begin{align*}
  \Br(t)&\subseteq\Br(q)\cup q\bigl(\Br(p\circ t_{0})\bigr)
        \subseteq\Br(q)\cup q\bigl(\Br(p)\cup p(S)\bigr)
        =\Br(q)\cup q(T)\subseteq\{0,1,\infty\}. \qedhere
\end{align*}
\end{proof}

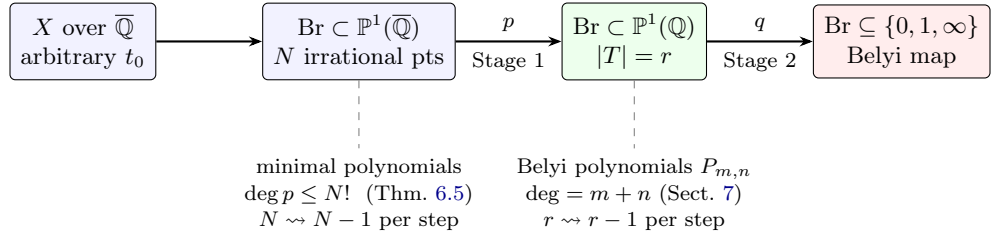
\begin{figure}[t]
\centering
\begin{tikzpicture}[
  node distance=8mm,
  box/.style={draw,rounded corners=2pt,align=center,inner sep=4pt,
              font=\footnotesize,minimum height=9mm},
  arr/.style={-{Stealth[length=5pt]},thick}]
\node[box,fill=blue!5] (A) {$X$ over $\Qbar$\\ arbitrary $t_{0}$};
\node[box,fill=blue!5,right=14mm of A] (B) {$\Br\subset\Pp^{1}(\Qbar)$\\ $N$ irrational pts};
\node[box,fill=green!7,right=14mm of B] (C) {$\Br\subset\Pp^{1}(\Q)$\\ $|T|=r$};
\node[box,fill=red!7,right=14mm of C] (D) {$\Br\subseteq\{0,1,\infty\}$\\ Belyi map};
\draw[arr] (A) -- (B);
\draw[arr] (B) -- node[above,font=\scriptsize]{$p$} node[below,font=\scriptsize]{Stage 1} (C);
\draw[arr] (C) -- node[above,font=\scriptsize]{$q$} node[below,font=\scriptsize]{Stage 2} (D);
\node[align=center,font=\scriptsize,below=9mm of B.south] (B2)
  {minimal polynomials\\ $\deg p\leq N!$ \ (Thm.~\ref{thm:rationalization})\\
   $N\rightsquigarrow N-1$ per step};
\node[align=center,font=\scriptsize,below=9mm of C.south] (C2)
  {Belyi polynomials $P_{m,n}$\\ $\deg=m+n$ (Sect.~\ref{sec:belyi-poly})\\
   $r\rightsquigarrow r-1$ per step};
\draw[gray,dashed] (B.south) -- (B2.north);
\draw[gray,dashed] (C.south) -- (C2.north);
\end{tikzpicture}
\caption{Belyi's algorithm. Stage~1 is bounded purely in terms of the cardinality of the
branch locus; Stage~2 reduces the cardinality by one at each step but at a cost governed
by the heights of the rational points involved, not by their number
(Remark~\ref{rem:no-height-bound}).}
\label{fig:algorithm}
\end{figure}

\subsection{How large is the Belyi map?}\label{ssec:quantitative}

The proof just given is effective, and it is worth extracting what it actually bounds. The
following is the precise accounting of Stage~1.

\begin{theorem}[Rationalization degree]\label{thm:rationalization}
Let $S\subset\Pp^{1}(\Qbar)$ be a finite set, closed under the action of $\GQ$, and let
$N=|S\smallsetminus\Pp^{1}(\Q)|$ be the number of its \emph{irrational} points. (Any
finite set may first be replaced by its $\GQ$-closure; this changes $N$ by at most a
factor equal to the largest degree of a point.) Then the polynomial $p$ of
Lemma~\ref{lem:rationalize} may be chosen with
\begin{equation}\label{eq:rationalization-bound}
  \deg p\;\leq\;N\,(N-1)\,(N-2)\cdots 2\;=\;N!\,,
\end{equation}
and the resulting rational branch locus $T=\Br(p)\cup p(S)$ satisfies $|T|\leq N!-1+|p(S)|$,
where $|p(S)|\leq|S|$ counts the distinct $p$-images of the points of $S$. If $S$ has no
rational points other than $\{0,1,\infty\}$, this simplifies to $|T|\leq N!$.
\end{theorem}

\begin{proof}
Let $D(N)$ denote the least degree achievable by the inductive construction, starting from
a $\GQ$-stable set with $N$ irrational points. The construction takes $p_{1}$ of degree
exactly $N$ (the product of the distinct minimal polynomials of the irrational elements,
whose roots are precisely those elements), and then applies the construction to
$S_{1}:=\Br(p_{1})\cup p_{1}(S)$. The key point is that $p_{1}$ has \emph{rational}
coefficients and maps every irrational element of $S$ to a \emph{rational} value---that is
the purpose of taking $p_{1}$ to be the minimal-polynomial product---so the irrational part
of $S_{1}$ can only come from $\Br(p_{1})$, of which there are at most $N-1$ finite values.
Hence $S_{1}$ has at most $N-1$ irrational points, and $D(N)\leq N\cdot D(N-1)$ with
$D(0)=D(1)=1$, giving $D(N)\leq N!$.

For the cardinality of the terminal rational set $T$ we account for both sources
separately, since $T=\Br(p)\cup p(S)$ and the second term is \emph{not} controlled by
$\deg p$ alone. Write $p=p_{r}\circ\cdots\circ p_{1}$ for the composition produced by the
$r\leq N$ stages. At each stage $p(S)$ can only shrink: $p_{1}$ already collapses the $N$
irrational points onto at most $N$ rational values and fixes $\{0,1,\infty\}\subseteq S$
setwise up to the finitely many rational points of $S$, and each subsequent $p_{i}$ maps a
finite rational set to a smaller-or-equal one. Thus $|p(S)|\leq|S_{1}|\leq|\Br(p_{1})|+
|p_{1}(S)|$. The number of \emph{rational} points of $S$ is not bounded by $N$; but they
contribute to $T$ only through their images, and $p_{1}(S)$ already lies in the rational set
$S_{1}$ whose irrational part we bounded above. Iterating, $|T|\leq|\Br(p)|+|p(S)|$ with
$|\Br(p)|\leq\deg p-1\leq N!-1$ and $|p(S)|\leq|S_{1}|$, so that the terminal set is finite
and effectively bounded; when $S$ has no rational points beyond $\{0,1,\infty\}$ the two
contributions merge and $|T|\leq N!$. In general the bound reads
$|T|\leq(N!-1)+|p_{1}(S)|$, the correction $|p_{1}(S)|\leq|S|$ being the number of distinct
$p_{1}$-images of the rational points of $S$---a quantity depending on $S$, not on $N$
alone, as it must.
\end{proof}

\begin{remark}[No bound from cardinality alone in Stage~2]\label{rem:no-height-bound}
It is tempting to hope for a companion bound $\deg q\leq F(r)$ in
Lemma~\ref{lem:concentrate}. There is none of the kind produced by the induction: at each
step the algorithm reduces $|T|$ by one at the cost of a factor $m+n$, where $m/(m+n)$ is
the chosen fourth point of $T$, and $m+n$ is a height, not a cardinality. Concretely, for
$T=\{0,1,\infty,\,1/M\}$ with $M$ large the first step already costs $\deg P_{1,M-1}=M$.
The induction therefore controls the \emph{number} of steps ($r-3$ of them) but not their
size, and the bound it yields is a product of heights of intermediate points, which the
recursion does not keep small in a way uniform in $|T|$. Genuine effective statements
must bring in the arithmetic of the branch points; see Khadjavi \cite{Khadjavi2002}, who
bounds the Belyi degree in terms of the degree and height of the defining data, and
Lit\-can\-u \cite{Litcanu2004} for the case of covers of $\Pp^{1}$.
\end{remark}

\begin{remark}[Collapsing at the last step]\label{rem:collapsing}
The final M\"obius step in Lemma~\ref{lem:concentrate} requires $|T|\leq 3$; when
$|T|=4$ one further Belyi polynomial is needed. The following elementary observation
sometimes saves a step: if $T=\{0,1,\infty,\lambda\}$ with $\lambda=m/(m+n)$ and
$\gcd(m,n)=k>1$, then by Theorem~\ref{thm:pmn-decomp} below the map $P_{m,n}$ factors as
$\pi_{k}\circ P_{m/k,n/k}$, and the inner factor already collapses $T$ to three points; the
outer $\pi_{k}(w)=w^{k}$ is redundant. Hence one may always take $\gcd(m,n)=1$ in
Lemma~\ref{lem:concentrate}, reducing the degree of that step from $m+n$ to
$(m+n)/\gcd(m,n)$.
\end{remark}
\section{The structure of the Belyi polynomials}\label{sec:belyi-poly}

The polynomials $P_{m,n}$ of \eqref{eq:belyi-poly} are the workhorses of Belyi's
algorithm, and they appear in every treatment of the theorem; but they are invariably used
and discarded. They are, however, interesting objects in their own right---an explicit
two-parameter family of Belyi maps whose dessins, decompositions and monodromy can all be
determined completely. That is the content of this section, announced as \ref{itm:A1}. Double-star trees
are familiar objects in the theory of Shabat polynomials---they appear among the first
examples in \cite{ShabatZvonkin1994} and \cite[Ch.~2]{LandoZvonkin2004}, and compositions
of plane trees are studied systematically in \cite{AdrianovZvonkin1998}---so what we claim
as new is not the objects but the exact statements: the on-the-nose compatibility of the
normalizing constants in Theorem~\ref{thm:pmn-decomp}, and the complete wreath-product
determination of the monodromy in Theorem~\ref{thm:pmn-monodromy}, with a self-contained
proof. Throughout we write
\begin{equation*}
  d=m+n,\qquad k=\gcd(m,n),\qquad m'=m/k,\quad n'=n/k,\quad d'=d/k=m'+n' .
\end{equation*}

\subsection{The dessin is a double star}\label{ssec:pmn-tree}

\begin{definition}\label{def:double-star}
For $m,n\geq1$ the \emph{double star} $D_{m,n}$ is the plane tree with two black vertices
$b_{0},b_{1}$ of degrees $m$ and $n$, joined by a path $b_{0}-w_{\ast}-b_{1}$ through a
white vertex $w_{\ast}$ of degree $2$, together with $m-1$ pendant edges at $b_{0}$ and
$n-1$ pendant edges at $b_{1}$, each ending in a white vertex of degree $1$.
\end{definition}

\begin{proposition}\label{prop:pmn-tree}
The dessin of $P_{m,n}$ is the double star $D_{m,n}$. Its passport is
\begin{equation}\label{eq:pmn-passport}
  \Pi(P_{m,n})=\bigl\langle (m,n),\ (2,1^{\,d-2}),\ (d)\bigr\rangle ,
\end{equation}
and $P_{m,n}$ is the \emph{unique} Belyi map with this passport, up to isomorphism.
\end{proposition}

\begin{proof}
By \eqref{eq:belyi-poly}, $P_{m,n}^{-1}(0)=\{0,1\}$ with multiplicities $m$ and $n$, so
there are two black vertices of degrees $m,n$; by
\eqref{eq:belyi-poly-derivative}--\eqref{eq:pmn-value} the fibre over $1$ contains
$m/(m+n)$ with multiplicity $2$ and, since $P_{m,n}'$ has no other zeros in
$\C\smallsetminus\{0,1\}$, $d-2$ further simple points. Being a polynomial of degree $d$,
$P_{m,n}$ is totally ramified over $\infty$. This gives \eqref{eq:pmn-passport}, and
\eqref{eq:belyi-RH} yields $2g-2=d-(2+(d-1)+1)=-2$, so $g=0$ and the dessin is a plane
tree with $d$ edges (\S\ref{ssec:dessins}).

The combinatorial structure is then forced. The white vertex $w_{\ast}$ of degree $2$ has
two neighbours, necessarily black, and there are only two black vertices; if both edges at
$w_{\ast}$ went to the same black vertex the graph would contain a cycle, contradicting
that it is a tree. So $w_{\ast}$ joins $b_{0}$ to $b_{1}$. Every remaining edge is
incident to a white vertex of degree $1$, i.e.\ is a pendant edge, and $b_{0}$ (resp.\
$b_{1}$) carries $m-1$ (resp.\ $n-1$) of them. This is exactly $D_{m,n}$. Since a plane
tree is determined by its abstract tree together with the cyclic orders at the vertices,
and here all vertices other than $b_{0},b_{1}$ have degree $\leq2$ while the cyclic orders
at $b_{0}$ and $b_{1}$ consist of one distinguished edge among otherwise indistinguishable
pendant edges, the plane tree $D_{m,n}$ is unique; uniqueness of the Belyi map follows from
Proposition~\ref{prop:dessin-dictionary}.
\end{proof}

\begin{corollary}\label{cor:pmn-rigid}
$P_{m,n}$ is defined over $\Q$, and its field of moduli is $\Q$. More precisely, the
$\GQ$-orbit of the dessin $D_{m,n}$ is a single point.
\end{corollary}

\begin{proof}
$\GQ$ permutes the dessins with a given passport, since conjugation preserves degrees and
ramification indices. By Proposition~\ref{prop:pmn-tree} the passport
\eqref{eq:pmn-passport} contains one dessin, so the orbit is trivial and the field of
moduli is $\Q$; and $P_{m,n}$ visibly has rational coefficients.
\end{proof}

This is the extreme opposite of the behaviour found in Section~\ref{sec:galois-dessins},
where a passport containing three trees has field of moduli a cubic field. Rigidity of a
passport---containing exactly one dessin---is precisely what forces definability over
$\Q$, and it is the mechanism behind the rigidity method in inverse Galois theory
\cite{MalleMatzat1999}.

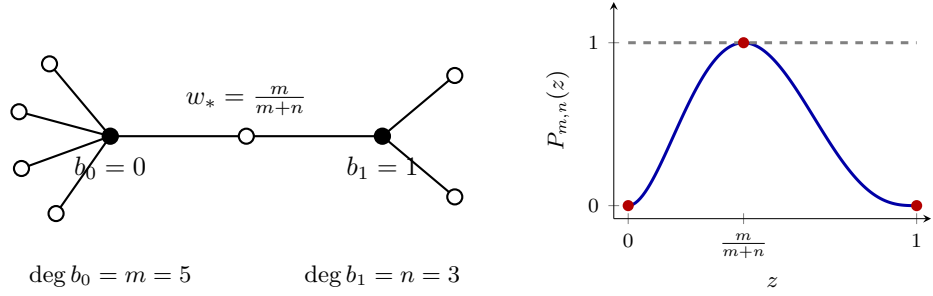
\begin{figure}[t]
\centering
\begin{tikzpicture}[scale=1.0,
  bl/.style={circle,fill=black,inner sep=2.3pt},
  wh/.style={circle,draw=black,fill=white,thick,inner sep=2.0pt},
  ed/.style={thick}]
\coordinate (b0) at (0,0);
\coordinate (b1) at (3.6,0);
\coordinate (ws) at (1.8,0);
\draw[ed] (b0) -- (ws) -- (b1);
\foreach \a in {130,165,200,235}{
  \draw[ed] (b0) -- ++(\a:1.25);
  \node[wh] at ($(b0)+(\a:1.25)$) {};
}
\foreach \a in {-40,40}{
  \draw[ed] (b1) -- ++(\a:1.25);
  \node[wh] at ($(b1)+(\a:1.25)$) {};
}
\node[bl] at (b0) {};
\node[bl] at (b1) {};
\node[wh] at (ws) {};
\node[below=4pt] at (b0) {\small $b_{0}=0$};
\node[below=4pt] at (b1) {\small $b_{1}=1$};
\node[above=4pt] at (ws) {\small $w_{\ast}=\tfrac{m}{m+n}$};
\node[font=\footnotesize,align=center] at (0.0,-1.85)
  {$\deg b_{0}=m=5$};
\node[font=\footnotesize,align=center] at (3.6,-1.85)
  {$\deg b_{1}=n=3$};
\end{tikzpicture}
\qquad
\begin{tikzpicture}[scale=0.95]
\begin{axis}[
  width=6.0cm,height=4.6cm,
  xlabel={\footnotesize $z$}, ylabel={\footnotesize $P_{m,n}(z)$},
  xmin=-0.05,xmax=1.05,ymin=-0.08,ymax=1.25,
  xtick={0,0.4,1},xticklabels={$0$,$\tfrac{m}{m+n}$,$1$},
  ytick={0,1},
  tick label style={font=\scriptsize},
  label style={font=\scriptsize},
  axis lines=left,
  every axis plot/.append style={very thick}]
\addplot[domain=0:1,samples=120,blue!65!black]
  {(3125/108)*x^2*(1-x)^3};
\addplot[dashed,gray,domain=0:1,samples=2]{1};
\addplot[only marks,mark=*,mark size=1.6pt,red!70!black]
  coordinates {(0,0) (1,0) (0.4,1)};
\end{axis}
\end{tikzpicture}
\caption{Left: the dessin $D_{5,3}$ of $P_{5,3}$, a double star. The two black vertices
sit at $z=0$ and $z=1$; the unique white vertex of degree $2$ sits at the interior
critical point $m/(m+n)$. Right: the graph of $P_{2,3}(z)=\frac{3125}{108}z^{2}(1-z)^{3}$
on $[0,1]$; the normalizing constant is exactly what makes the interior critical value
equal to $1$, by \eqref{eq:pmn-value}.}
\label{fig:double-star}
\end{figure}

\subsection{An exact decomposition}\label{ssec:pmn-decomp}

The following identity is exact---not merely up to constants---and this is what makes it
useful: the normalizations built into \eqref{eq:belyi-poly} are compatible with
composition.

\begin{theorem}[Decomposition of Belyi polynomials]\label{thm:pmn-decomp}
Let $k=\gcd(m,n)$ and $\pi_{k}(w)=w^{k}$. Then
\begin{equation}\label{eq:pmn-decomp}
  P_{m,n}\;=\;\pi_{k}\circ P_{m/k,\;n/k}
  \qquad\text{as polynomials in }\Q[z].
\end{equation}
More generally $P_{m,n}=\pi_{e}\circ P_{m/e,\,n/e}$ for every common divisor $e$ of
$m$ and $n$, and $P_{m,n}$ is indecomposable as a composition of polynomials of the
form $\pi_{e}\circ(\,\cdot\,)$ precisely when $\gcd(m,n)=1$.
\end{theorem}

\begin{proof}
Let $e\mid\gcd(m,n)$ and put $\mu=m/e$, $\nu=n/e$, $\delta=\mu+\nu=d/e$. Then
\begin{equation*}
  \bigl(P_{\mu,\nu}(z)\bigr)^{e}
  =\left(\frac{\delta^{\delta}}{\mu^{\mu}\nu^{\nu}}\right)^{\!e} z^{\mu e}(1-z)^{\nu e}
  =\frac{\delta^{\delta e}}{\mu^{\mu e}\nu^{\nu e}}\;z^{m}(1-z)^{n}
  =\frac{(d/e)^{d}}{\mu^{m}\nu^{n}}\;z^{m}(1-z)^{n},
\end{equation*}
using $\delta e=d$ and $\mu^{\mu e}=\mu^{m}$, $\nu^{\nu e}=\nu^{n}$. On the other hand
\begin{equation*}
  \frac{d^{d}}{m^{m}n^{n}}
  =\frac{d^{d}}{(e\mu)^{m}(e\nu)^{n}}
  =\frac{d^{d}}{e^{m+n}\,\mu^{m}\nu^{n}}
  =\frac{(d/e)^{d}}{\mu^{m}\nu^{n}},
\end{equation*}
because $e^{m+n}=e^{d}$. The two right-hand sides agree, so
$\bigl(P_{\mu,\nu}\bigr)^{e}=P_{m,n}$, which is \eqref{eq:pmn-decomp}. For the last
assertion, if $\gcd(m,n)=1$ then $z^{m}(1-z)^{n}$ is not a proper power in
$\C[z]$---the multiplicity $m$ of the root $0$ and $n$ of the root $1$ would both have
to be divisible by the exponent---so no factorization $\pi_{e}\circ(\,\cdot\,)$ with
$e>1$ exists.
\end{proof}

\begin{remark}[Relation to Ritt decomposition]\label{rem:ritt}
The qualifier in Theorem~\ref{thm:pmn-decomp} is genuine and should not be shortened to
``indecomposable''. We have shown $P_{m,n}$ is indecomposable \emph{as a tower of outer
power maps} $\pi_{e}\circ(\,\cdot\,)$ when $\gcd(m,n)=1$; this is weaker than Ritt
indecomposability \cite{Ritt1922}, which forbids \emph{all} nontrivial functional
factorizations $f\circ g$. For $z^{m}(1-z)^{n}$ with $\gcd(m,n)=1$ the two notions in fact
coincide---any decomposition of a polynomial with exactly two finite critical values, one
totally ramified point over each of $0$ and $1$, and the remaining ramification simple, is
forced by the L\"uroth/Ritt analysis of block systems to be an outer power, which is the
content of Theorem~\ref{thm:pmn-monodromy}: a primitive monodromy group admits no proper
block system, hence no proper decomposition. But the two notions are logically distinct,
and only the outer-power statement is proved directly here. Monodromy of Shabat
(one- or two-critical-value) polynomials and of decomposable dessins is studied by
Adrianov--Zvonkin \cite{AdrianovZvonkin1998} and in the survey \cite{LandoZvonkin2004};
the double stars $D_{m,n}$ are the diameter-three trees of that literature.
\end{remark}

\begin{remark}\label{rem:decomp-dessin}
Decomposition \eqref{eq:pmn-decomp} is visible in the dessins: $D_{m,n}$ is the
``$k$-fold inflation'' of $D_{m',n'}$, in which each edge of $D_{m',n'}$ is replaced by
$k$ parallel edges around the black vertices. On monodromy this is the passage from
$\Sym_{d'}$ to a wreath product, which is Theorem~\ref{thm:pmn-monodromy}. In
Remark~\ref{rem:collapsing} it is what allows one to replace $P_{m,n}$ by the cheaper
$P_{m',n'}$ in Belyi's algorithm.
\end{remark}

\begin{figure}[t]
\centering
\begin{tikzpicture}[scale=0.9,
  bl/.style={circle,fill=black,inner sep=2.1pt},
  wh/.style={circle,draw=black,fill=white,thick,inner sep=1.7pt},
  edA/.style={very thick,blue!60!black},
  edB/.style={very thick,red!70!black,densely dashed}]
\begin{scope}[xshift=0cm]
  \coordinate (b0) at (0,0); \coordinate (b1) at (2.1,0); \coordinate (ws) at (1.05,0);
  \draw[edA] (b0)--(ws); \draw[edA] (ws)--(b1);
  \draw[edA] (b0)--++(140:1.0); \node[wh] at ($(b0)+(140:1.0)$) {};
  \draw[edA] (b0)--++(220:1.0); \node[wh] at ($(b0)+(220:1.0)$) {};
  \draw[edA] (b1)--++(0:1.0);  \node[wh] at ($(b1)+(0:1.0)$) {};
  \node[bl] at (b0) {}; \node[bl] at (b1) {}; \node[wh] at (ws) {};
  \node[font=\footnotesize,align=center] at (1.3,-1.55)
    {$D_{3,2}$: dessin of $P_{3,2}$\\ $\Mon=\Sym_{5}$ (primitive)};
\end{scope}
\begin{scope}[xshift=6.6cm]
  \coordinate (b0) at (0,0); \coordinate (b1) at (2.6,0); \coordinate (ws) at (1.3,0);
  \draw[edA] (b0)--(ws); \draw[edA] (ws)--(b1);
  \foreach \a/\st in {60/edB,120/edA,180/edB,240/edA,300/edB}{
    \draw[\st] (b0)--++(\a:1.05); \node[wh] at ($(b0)+(\a:1.05)$) {};}
  \foreach \a/\st in {90/edB,0/edA,-90/edB}{
    \draw[\st] (b1)--++(\a:1.05); \node[wh] at ($(b1)+(\a:1.05)$) {};}
  \node[bl] at (b0) {}; \node[bl] at (b1) {}; \node[wh] at (ws) {};
  \node[font=\footnotesize,align=center] at (1.3,-1.75)
    {$D_{6,4}$: dessin of $P_{6,4}=\pi_{2}\circ P_{3,2}$\\
     $\Mon=\Sym_{5}\wr\Z/2$ (imprimitive)};
\end{scope}
\end{tikzpicture}
\caption{The decomposition of Theorem~\ref{thm:pmn-decomp} seen on dessins. Right: the ten
edges of $D_{6,4}$, $2$-coloured by the block system of
Theorem~\ref{thm:pmn-monodromy}---solid blue edges lie over the branch
$P_{3,2}\in(0,1)$ of the intermediate map, dashed red ones over $P_{3,2}\in(-1,0)$; each
class has five edges, and the two classes are the two blocks of the imprimitive action.
Around each black vertex the colours alternate (locally $P_{3,2}\sim cz^{r}$ takes
positive and negative real values in alternating sectors), so $\sigma_{0}$ interleaves the
blocks, while the transposition $\sigma_{1}$, supported at the white vertex of degree $2$
on the central path, stays inside the blue block---exactly the configuration exploited in
Steps 1 and 3 of the proof.}
\label{fig:inflation}
\end{figure}
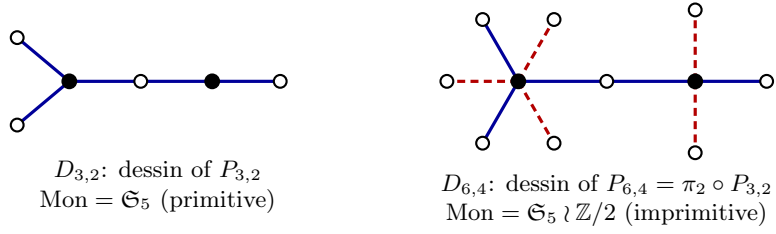

\subsection{The monodromy group}\label{ssec:pmn-monodromy}

We can now compute $\Mon(P_{m,n})$ completely. We fix the convention: for a permutation
group $C\leq\Sym_{b}$, the wreath product $\Sym_{a}\wr C$ is the semidirect product
$(\Sym_{a})^{b}\rtimes C$, with $C$ permuting the $b$ coordinates, in its
\emph{imprimitive} action on $ab=\{1,\dots,a\}\times\{1,\dots,b\}$ letters: the base group
acts within the $b$ blocks $\{1,\dots,a\}\times\{j\}$ and $C$ permutes the blocks. Thus
$\Sym_{a}\wr\Z/b$ means $\Sym_{a}\wr C$ for $C\leq\Sym_{b}$ cyclic of order $b$ generated
by a $b$-cycle; its order is $(a!)^{b}\,b$. Inside $\Sym_{ab}$, the full setwise
stabilizer of a partition into $b$ blocks of size $a$ is $\Sym_{a}\wr\Sym_{b}$, and for
$C\leq\Sym_{b}$ the preimage of $C$ under the natural surjection
$\Sym_{a}\wr\Sym_{b}\to\Sym_{b}$ is exactly $\Sym_{a}\wr C$; in particular the preimage
splits, because the full wreath product does.

\begin{lemma}[Monodromy generators]\label{lem:pmn-generators}
Label the edges of $D_{m,n}$ by $1,\dots,d$ so that $1,\dots,m$ are the edges at $b_{0}$
in counterclockwise order with $m$ the edge to $w_{\ast}$, and $m+1,\dots,d$ are the edges
at $b_{1}$ in counterclockwise order with $m+1$ the edge to $w_{\ast}$. Then
\begin{equation}\label{eq:pmn-gens}
  \sigma_{0}=(1\,2\,\cdots\,m)(m{+}1\,\cdots\,d),\qquad
  \sigma_{1}=(m,\,m{+}1),\qquad
  \sigma_{\infty}=(\sigma_{0}\sigma_{1})^{-1},
\end{equation}
and $\sigma_{\infty}$ is a $d$-cycle.
\end{lemma}

\begin{proof}
By Proposition~\ref{prop:dessin-dictionary}, $\sigma_{0}$ rotates the edges at each black
vertex and $\sigma_{1}$ those at each white vertex; the white vertices of degree $1$
contribute fixed points and $w_{\ast}$ the transposition $(m,m{+}1)$. That
$\sigma_{\infty}$ is a $d$-cycle restates total ramification over $\infty$, which holds
because $P_{m,n}$ is a polynomial; alternatively, a plane tree has a single face.
\end{proof}

\begin{theorem}[Monodromy of the Belyi polynomials]\label{thm:pmn-monodromy}
Let $k=\gcd(m,n)$ and $d'=(m+n)/k$. Then
\begin{equation}\label{eq:pmn-monodromy}
  \Mon(P_{m,n})\;\cong\;\Sym_{d'}\wr\Z/k ,
\end{equation}
acting imprimitively with exactly $k$ blocks of size $d'$. In particular
\begin{equation*}
  \Mon(P_{m,n})=\Sym_{m+n}\iff\gcd(m,n)=1 ,
\end{equation*}
and in general $\bigl|\Mon(P_{m,n})\bigr|=(d'!)^{k}\,k$.
\end{theorem}

\begin{proof}
The proof has two regimes. When $\gcd(m,n)=1$ we show no block system can exist, and
Jordan's theorem finishes at once. When $k=\gcd(m,n)>1$ the decomposition of
Theorem~\ref{thm:pmn-decomp} supplies a block system with $k$ blocks, which a counting
argument shows to be maximal; we then determine the kernel and the image of the action on
blocks separately, and identify the group inside the full block stabilizer by comparing
orders. Write $G=\Mon(P_{m,n})=\langle\sigma_{0},\sigma_{1}\rangle\leq\Sym_{d}$ with the
generators \eqref{eq:pmn-gens}; $G$ is transitive.

\emph{Step 1: any block system has at most $k$ blocks, and $\sigma_{1}$ lies inside a
block.} Let $\mathcal{B}$ be a system of $b$ blocks of size $c$, $bc=d$, $b>1$. The
transposition $\sigma_{1}=(m,m{+}1)$ preserves $\mathcal{B}$. If $m$ and $m+1$ lay in
distinct blocks $B\neq B'$, then $\sigma_{1}(B)$ is a block containing $m+1$, hence
$\sigma_{1}(B)=B'$; but $\sigma_{1}$ fixes every element of $B\smallsetminus\{m\}$, which
therefore lies in $B\cap B'=\emptyset$, forcing $c=1$. So for $c>1$ the two points $m,m+1$
lie in a common block $B$. Now $\sigma_{0}$ preserves $\mathcal{B}$ and preserves each of
the two sets $A_{1}=\{1,\dots,m\}$ and $A_{2}=\{m+1,\dots,d\}$, acting on them as an
$m$-cycle and an $n$-cycle. The blocks $\sigma_{0}^{\,j}(B)$, $j\geq0$, contain
$\sigma_{0}^{\,j}(m)$ and $\sigma_{0}^{\,j}(m{+}1)$, and as $j$ varies these exhaust
$A_{1}$ and $A_{2}$; hence the $\sigma_{0}$-orbit of $B$ covers all $d$ points and
therefore consists of all $b$ blocks. Since $\sigma_{0}$ maps blocks to blocks and
preserves $A_{1}$ and $A_{2}$, all blocks in this orbit contain the same number
$\alpha=|B\cap A_{1}|$ of points of $A_{1}$ and the same number $\beta=|B\cap A_{2}|$ of
points of $A_{2}$. Counting, $b\alpha=m$ and $b\beta=n$, so
\begin{equation}\label{eq:block-divisibility}
  b\ \bigm|\ \gcd(m,n)=k .
\end{equation}

\emph{Step 2: the case $k=1$.} By \eqref{eq:block-divisibility} no nontrivial block system
exists, so $G$ is primitive. A primitive permutation group containing a transposition is
the full symmetric group (Jordan; see \cite[Thm.~3.3A]{DixonMortimer1996}), and $G$
contains $\sigma_{1}$. Hence $G=\Sym_{d}=\Sym_{d'}\wr\Z/1$, as claimed.

\emph{Step 3: the case $k>1$.} By Theorem~\ref{thm:pmn-decomp},
$P_{m,n}=\pi_{k}\circ P_{m',n'}$. For a composite covering $t=f\circ g$ the fibre
$t^{-1}(\ast)$ decomposes into the $\deg f$ sets $g^{-1}(u)$, $u\in f^{-1}(\ast)$, and
these form a block system for $\Mon(t)$. Here this gives a system of exactly $k$ blocks of
size $d'$, so by \eqref{eq:block-divisibility} $k$ is the maximal possible number of
blocks. Fix such a maximal system $\mathcal{B}$ (with $b=k$ blocks of size $d'$), let
$\pi\colon G\to\Sym(\mathcal{B})\cong\Sym_{k}$ be the induced action on blocks, and let
$N=\ker\pi$.

\emph{Step 3a: the image $\pi(G)$ is cyclic of order $k$, transitive.} The image is
generated by $\pi(\sigma_{0})$ and $\pi(\sigma_{1})$. The transposition $\sigma_{1}$ has,
by Step 1, its support $\{m,m{+}1\}$ inside a single block $B_{0}$; it maps $B_{0}$ to
itself and fixes every point outside $B_{0}$, hence fixes every other block pointwise, so
$\pi(\sigma_{1})=1$ and $\sigma_{1}\in N$. Thus $\pi(G)=\langle\pi(\sigma_{0})\rangle$ is
cyclic; it is transitive on the $k$ blocks because $G$ is transitive on points, and a
transitive cyclic subgroup of $\Sym_{k}$ has order exactly $k$ and is generated by a
$k$-cycle.

\emph{Step 3b: block stabilizers restrict to the full $\Sym_{d'}$.} Let
$G_{B}=\pi^{-1}\bigl(\operatorname{Stab}_{\pi(G)}(B)\bigr)$ be the stabilizer of a block
$B\in\mathcal{B}$ and $H\leq\Sym(B)\cong\Sym_{d'}$ the image of the restriction
$G_{B}\to\Sym(B)$. Then $H$ is transitive ($G_{B}$ acts transitively on its block, since
$G$ is transitive and blocks are permuted transitively) and \emph{primitive}: an
$H$-invariant nontrivial partition of $B$ would transport, by transitivity of $G$ on
$\mathcal{B}$, to a $G$-invariant partition of all $d$ points refining $\mathcal{B}$ into
more than $k$ blocks, contradicting \eqref{eq:block-divisibility}. And $H$ contains a
transposition: choose $g\in G$ with $g(B_{0})=B$ (possible by Step 3a); then
$g\sigma_{1}g^{-1}$ is a transposition with support in $B$, lies in $G_{B}$, and restricts
to a transposition of $B$. By Jordan's theorem \cite[Thm.~3.3A]{DixonMortimer1996},
$H=\Sym(B)$.

\emph{Step 3c: the kernel is the full base group.} For $B\in\mathcal{B}$ let
$N_{B}\leq\Sym_{d}$ be the group of permutations supported in $B$ (acting arbitrarily on
$B$, trivially outside). Every element of $N_{B}$ maps $B$ to $B$ and fixes each other
block pointwise, so $N_{B}\cap G\subseteq N$; we must show $N_{B}\cap G$ is all of
$N_{B}\cong\Sym_{d'}$. The subgroup $N_{B}\cap G$ is normal in $G_{B}$: for $g\in G_{B}$
and $x\in N_{B}\cap G$, the conjugate $gxg^{-1}$ lies in $G$ and is supported in
$g(B)=B$. Its image under the restriction map $\rho\colon G_{B}\to\Sym(B)$ is therefore a normal
subgroup of $\rho(G_{B})=H=\Sym_{d'}$, since the image of a normal subgroup under a
homomorphism is normal in the image; and it contains the transposition
$g\sigma_{1}g^{-1}|_{B}$ constructed in Step 3b, which is supported in $B$ and hence lies
in $N_{B}\cap G$. The only normal subgroup of $\Sym_{d'}$ containing a transposition is
$\Sym_{d'}$ itself ($d'\geq2$; for $d'\geq3$ the proper normal subgroups are $A_{d'}$ and,
for $d'=4$, the Klein group, none containing a transposition). Hence the restriction of
$N_{B}\cap G$ to $B$ is all of $\Sym(B)$, and since elements of $N_{B}$ are determined by
their restriction to $B$, $N_{B}\cap G=N_{B}\cong\Sym_{d'}$. The subgroups $N_{B}$,
$B\in\mathcal{B}$, have disjoint supports, so they commute and generate their direct
product inside $G$:
\begin{equation*}
  (\Sym_{d'})^{k}\cong\prod_{B\in\mathcal{B}}N_{B}\;\subseteq\;N .
\end{equation*}
Conversely $N$ fixes each block setwise, so $N\subseteq\prod_{B}\Sym(B)=(\Sym_{d'})^{k}$;
hence $N=(\Sym_{d'})^{k}$, the full base group.

\emph{Step 3d: identification with the wreath product.} $G$ preserves $\mathcal{B}$, so
$G\leq\Sym_{d'}\wr\Sym_{k}$, the full stabilizer of the block system; and
$\pi(G)=C\cong\Z/k$ by Step 3a, so $G$ is contained in the preimage of $C$, which by the
convention fixed above is exactly $\Sym_{d'}\wr\Z/k$, of order $(d'!)^{k}k$. But
$|G|=|N|\cdot|\pi(G)|=(d'!)^{k}\,k$ by Steps 3a and 3c. Containment plus equality of
orders gives $G=\Sym_{d'}\wr\Z/k$, with the semidirect (split) structure inherited from
the ambient wreath product. This completes the proof.
\end{proof}

\begin{example}\label{ex:pmn-mon}
For $(m,n)=(2,2)$ one has $k=2$, $d'=2$ and
$\Mon(P_{2,2})\cong\Sym_{2}\wr\Z/2$, the dihedral group of order $8$---indeed
$P_{2,2}(z)=16z^{2}(1-z)^{2}=\bigl(4z(1-z)\bigr)^{2}=\pi_{2}\circ P_{1,1}$. For
$(m,n)=(4,6)$, $k=2$ and $d'=5$, giving $|{\Mon}|=(5!)^{2}\cdot2=28800$ inside
$\Sym_{10}$ of order $3628800$. For $(m,n)=(4,5)$, $\gcd=1$ and
$\Mon(P_{4,5})=\Sym_{9}$. These values agree with direct computation of the group
generated by \eqref{eq:pmn-gens}.
\end{example}

\begin{corollary}\label{cor:pmn-galois}
The covering $P_{m,n}$ is Galois if and only if $(m,n)=(1,1)$, that is, if and only if
$P_{m,n}=4z(1-z)$.
\end{corollary}

\begin{proof}
A covering of degree $d$ is Galois exactly when its monodromy group acts regularly, i.e.\
has order $d$. By Theorem~\ref{thm:pmn-monodromy} this reads $(d'!)^{k}k=kd'$, that is
$(d'!)^{k}=d'$. Since $d'=m'+n'\geq2$ we have $(d'!)^{k}\geq d'!\geq d'$, with equality
only if $k=1$ and $d'!=d'$, i.e.\ $d'=2$. Hence $k=1$ and $m'+n'=2$, so $m=n=1$.
Conversely $P_{1,1}(z)=4z(1-z)$ has monodromy $\Sym_{2}$ of order $2=\deg P_{1,1}$.
\end{proof}

\begin{remark}\label{rem:pmn-galois-factor}
The composite structure explains the near-misses: $P_{k,k}=\pi_{k}\circ P_{1,1}$ is a
composition of two Galois coverings, of degrees $k$ and $2$, but is not itself Galois for
$k\geq2$---the composite of Galois coverings need not be Galois, and
Theorem~\ref{thm:pmn-monodromy} quantifies the failure: the monodromy is
$\Sym_{2}\wr\Z/k$ of order $2^{k}k$, not $2k$.
\end{remark}
\section{Lower bounds for Belyi degrees}\label{sec:degree}

Belyi's algorithm produces \emph{a} Belyi map, usually of enormous degree. The opposite
question---how small can a Belyi map on a given curve be?---is delicate, and the invariant
\begin{equation*}
  \mathcal{B}(X)\;=\;\min\bigl\{\deg t:\ t\ \text{a Belyi map on }X\bigr\}
\end{equation*}
(the \emph{Belyi degree}) is the subject of a substantial literature. This section proves
two clean lower bounds, both sharp, which follow from \eqref{eq:belyi-RH} alone. They are
elementary; we record them with their equality cases and sharpness examples, which is what
makes them usable.

\begin{theorem}[Degree bounds]\label{thm:degree-bound}
Let $X$ be a curve of genus $g$ over $\C$ and $t\colon X\to\Pp^{1}$ a Belyi map of degree
$d$.
\begin{enumerate}[label=\textup{(\roman*)},leftmargin=2.4em]
\item\label{itm:db-general} $d\geq 2g+1$, with equality if and only if $t$ is totally
ramified over each of $0$, $1$ and $\infty$.
\item\label{itm:db-clean} If $t$ is pre-clean---in particular if $t$ is clean---then
$d\geq 4g$. For clean $t$, equality holds if and only if $t$ is totally ramified over $0$
and over $\infty$. More precisely $d\geq 4g+r$, where $r$ is the number of unramified
points of $t^{-1}(1)$.
\end{enumerate}
Consequently $\mathcal{B}(X)\geq 2g+1$ for every curve of genus $g$.
\end{theorem}

\begin{proof}
\ref{itm:db-general} By \eqref{eq:belyi-RH}, $2g-2=d-(n_{0}+n_{1}+n_{\infty})$ where
$n_{s}=\#t^{-1}(s)\geq1$. Hence $2g-2\leq d-3$, i.e.\ $d\geq 2g+1$, with equality exactly
when $n_{0}=n_{1}=n_{\infty}=1$, which is total ramification over each of the three points.

\ref{itm:db-clean} If $t$ is clean, every point of $t^{-1}(1)$ has ramification index $2$,
so $n_{1}=d/2$ (in particular $d$ is even). Then \eqref{eq:belyi-RH} reads
\begin{equation}\label{eq:clean-RH}
  2g-2\;=\;d-\frac{d}{2}-n_{0}-n_{\infty}\;=\;\frac{d}{2}-n_{0}-n_{\infty}
  \;\leq\;\frac{d}{2}-2 ,
\end{equation}
whence $d\geq 4g$, with equality precisely when $n_{0}=n_{\infty}=1$. For a pre-clean map
let $r$ be the number of points of $t^{-1}(1)$ with $\ram_{P}=1$ and $s$ the number with
$\ram_{P}=2$, so that $r+2s=d$ and $n_{1}=r+s=\frac{d+r}{2}$. Then \eqref{eq:belyi-RH}
gives
\begin{equation*}
  2g-2=d-\frac{d+r}{2}-n_{0}-n_{\infty}=\frac{d-r}{2}-n_{0}-n_{\infty}
  \leq\frac{d-r}{2}-2 ,
\end{equation*}
so $d\geq 4g+r\geq 4g$, which contains the clean case $r=0$.
\end{proof}

Both bounds are attained, and for infinitely many genera. The sharpness examples we use
are cyclic---superelliptic---covers: the constraint of total ramification over all three
points, which is what equality forces, is solved in the simplest way by the curves
$y^{N}=x^{a}(1-x)^{b}$. It must be stressed, and \S\ref{ssec:extremal} will make it
precise, that these are the simplest extremal maps, \emph{not} the only ones.

\begin{example}[Superelliptic curves; sharpness of $d\geq 2g+1$]\label{ex:superelliptic}
Fix $N\geq2$ and integers $a,b\geq1$ with $\gcd(a,b,N)=1$, and let $X_{N,a,b}$ be the
smooth projective model of
\begin{equation}\label{eq:superelliptic}
  y^{N}=x^{a}(1-x)^{b},
\end{equation}
with $t=x\colon X_{N,a,b}\to\Pp^{1}$ of degree $N$. This is a cyclic covering, Galois with
group $\Z/N$ acting by $y\mapsto\zeta_{N}y$, branched only over $0,1,\infty$, with
ramification indices
\begin{equation}\label{eq:superelliptic-ram}
  \ram_{0}=\frac{N}{\gcd(a,N)},\qquad
  \ram_{1}=\frac{N}{\gcd(b,N)},\qquad
  \ram_{\infty}=\frac{N}{\gcd(a+b,N)} ,
\end{equation}
and $n_{s}=N/\ram_{s}$ points in each fibre. Thus $t$ is a Belyi map, and
\eqref{eq:belyi-RH} gives
\begin{equation*}
  g=\frac{N-n_{0}-n_{1}-n_{\infty}+2}{2}.
\end{equation*}
Taking $a=b=1$ and $N=2g+1$ odd, all three gcd's are $1$ (as
$\gcd(2,2g+1)=1$), so $\ram_{0}=\ram_{1}=\ram_{\infty}=N$, $n_{s}=1$, and
\begin{equation*}
  y^{2g+1}=x(1-x)\quad\text{has genus }g\quad\text{and a Belyi map of degree } 2g+1 .
\end{equation*}
By Theorem~\ref{thm:degree-bound}\ref{itm:db-general} this is optimal:
$\mathcal{B}(X)=2g+1$ for these curves, and the bound $d\geq2g+1$ is sharp for every
$g\geq1$. For $g=2$ the curve is $y^{5}=x(1-x)$, with $\Pi(t)=\langle(5),(5),(5)\rangle$
and dessin the five-edge ``banana'' of Figure~\ref{fig:banana}; for $g=3$,
$y^{7}=x(1-x)$.
\end{example}

\begin{figure}[t]
\centering
\begin{tikzpicture}[scale=1.05,
  bl/.style={circle,fill=black,inner sep=2.3pt},
  wh/.style={circle,draw=black,fill=white,thick,inner sep=2.0pt}]
\coordinate (B) at (-1.9,0); \coordinate (W) at (1.9,0);
\foreach \o/\i in {60/120, 30/150, 0/180, -30/-150, -60/-120}{
  \draw[very thick] (B) to[out=\o,in=\i] (W);}
\node[bl] at (B) {}; \node[wh] at (W) {};
\node[below=5pt] at (B) {\small $t^{-1}(0)$: one point, $\ram=5$};
\node[below=5pt] at (W) {\small $t^{-1}(1)$: one point, $\ram=5$};
\node[font=\footnotesize,align=center] at (0,-1.55)
  {$\#V-\#E+\#F = 2-5+1 = -2 = 2-2g \ \Longrightarrow\ g=2$};
\end{tikzpicture}
\caption{The dessin of the extremal Belyi map $t=x$ on $y^{5}=x(1-x)$
(Example~\ref{ex:superelliptic}): five edges joining a single black to a single white
vertex---a ``banana'' graph---embedded in the genus-$2$ surface with a single face (the
unique totally ramified point over $\infty$). Drawn in the plane the five edges bound four
faces; on the curve they bound one, and the deficit is the genus, by Euler's formula. The
monodromy is $\sigma_{0}=\sigma_{1}^{-1}=$ a $5$-cycle, generating $\Z/5$: this is the
cyclic case. The unique \emph{non-cyclic} extremal dessin of degree $5$
(Corollary~\ref{cor:A5-rational}) has the same passport but its five edges are attached in
a genuinely different cyclic order, generating $A_{5}$.}
\label{fig:banana}
\end{figure}
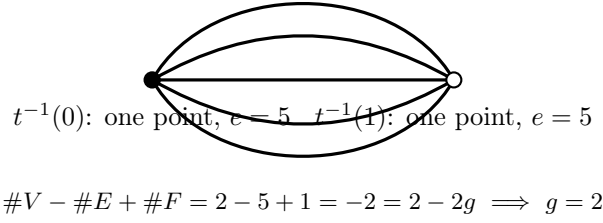

\begin{example}[Sharpness of the clean bound $d\geq4g$]\label{ex:clean-elliptic}
Take $N=4$, $a=1$, $b=2$ in \eqref{eq:superelliptic}:
\begin{equation*}
  X:\ y^{4}=x(1-x)^{2},\qquad t=x .
\end{equation*}
Here $\gcd(1,4)=1$, $\gcd(2,4)=2$, $\gcd(3,4)=1$, so by \eqref{eq:superelliptic-ram}
$\ram_{0}=4$, $\ram_{1}=2$, $\ram_{\infty}=4$, giving $n_{0}=n_{\infty}=1$, $n_{1}=2$.
Every point over $1$ has index exactly $2$: the map is \emph{clean}. Riemann--Hurwitz
gives $2g-2=4-(1+2+1)=0$, so $g=1$, and $d=4=4g$. By
Theorem~\ref{thm:degree-bound}\ref{itm:db-clean} this elliptic curve has the smallest
possible clean Belyi degree. Its dessin has one black vertex of degree $4$, two white
vertices of degree $2$, and one face---a bouquet of two loops on the torus. More
generally, $y^{4N}=x(1-x)^{2N}$ and its relatives produce clean extremal examples in
higher genus.
\end{example}

\subsection{Extremal Belyi maps}\label{ssec:extremal}

Equality in Theorem~\ref{thm:degree-bound}\ref{itm:db-general} is a strong condition: it
forces the passport to be $\langle(d),(d),(d)\rangle$ with $d=2g+1$. We can say
considerably more about the maps that achieve it.

\begin{proposition}[Parity and the alternating constraint]\label{prop:parity}
A passport $\langle(d),(d),(d)\rangle$ is realized by a covering only if $d$ is odd---so
there is no obstruction at $d=2g+1$, which is always odd---and in that case the monodromy
group of every extremal Belyi map is contained in the alternating group:
\begin{equation*}
  \Mon(t)\;\leq\;A_{d}\qquad\text{for every Belyi map attaining } d=2g+1 .
\end{equation*}
\end{proposition}

\begin{proof}
A $d$-cycle has sign $(-1)^{d-1}$. Applying $\operatorname{sgn}$ to
$\sigma_{0}\sigma_{1}\sigma_{\infty}=1$ gives $(-1)^{3(d-1)}=1$, so $d$ is odd; and then
each generator $\sigma_{0},\sigma_{1}$ is an even permutation, so the group they generate
lies in $A_{d}$.
\end{proof}

Proposition~\ref{prop:parity} already rules out $\Mon=\Sym_{d}$ for extremal maps, and is
consistent with everything found below: the groups occurring in
Theorem~\ref{thm:extremal-classification} are $\Z/d$, $A_{5}$,
$\mathrm{PSL}_{2}(\mathbb{F}_{7})\leq A_{7}$, and $A_{7}$.

\begin{proposition}[Cyclic extremal maps]\label{prop:cyclic-extremal}
Let $d=2g+1\geq3$ be odd. The Belyi maps of degree $d$ with cyclic monodromy $\Z/d$
and passport $\langle(d),(d),(d)\rangle$ are precisely the superelliptic maps
$t=x$ on $y^{d}=x^{a}(1-x)^{b}$ of Example~\ref{ex:superelliptic} with
$\gcd(a,d)=\gcd(b,d)=\gcd(a+b,d)=1$, and their number up to isomorphism is the
multiplicative function
\begin{equation}\label{eq:cyclic-count}
  N_{\mathrm{cyc}}(d)\;=\;\prod_{p^{k}\,\|\,d}p^{\,k-1}(p-2)\,,
\end{equation}
the product over the prime powers exactly dividing $d$. For $d$ prime this is
$d-2=2g-1$; e.g.\ $N_{\mathrm{cyc}}(9)=3$, $N_{\mathrm{cyc}}(15)=3$,
$N_{\mathrm{cyc}}(25)=15$.
\end{proposition}

\begin{proof}
A covering with cyclic monodromy of order $d=\deg t$ is Galois with group $\Z/d$, hence
Kummer (we are in characteristic zero with all roots of unity present):
$\C(X)=\C(x)[y]/(y^{d}-f)$ for some $f\in\C(x)^{\times}$, and being branched only over
$0,1,\infty$ forces $f=x^{a}(1-x)^{b}$ up to $d$-th powers and constants. By
\eqref{eq:superelliptic-ram}, total ramification over $0,1,\infty$ is equivalent to
$\gcd(a,d)=\gcd(b,d)=\gcd(a+b,d)=1$. Two admissible pairs give isomorphic coverings over
$\Pp^{1}$ exactly when the Kummer classes agree up to the choice of generator of $\Z/d$,
i.e.\ when $(a',b')\equiv\lambda(a,b)\pmod d$ for some $\lambda\in(\Z/d)^{\times}$; the
scaling action is free, since $\lambda a\equiv a$ with $a$ a unit forces $\lambda=1$. It
remains to count admissible pairs. By the Chinese remainder theorem the count is
multiplicative, so fix $p^{k}\,\|\,d$: the admissible pairs modulo $p^{k}$ are the
$(a,b)$ with $a,b$ units and $a+b$ a unit, i.e.\ $b\not\equiv-a\pmod p$; the units of
$\Z/p^{k}$ distribute equally over the $p-1$ nonzero residues modulo $p$, so for each of
the $\varphi(p^{k})$ choices of $a$ there are
$\varphi(p^{k})-p^{k-1}=p^{k-1}(p-2)$ choices of $b$. Multiplying over $p^{k}\,\|\,d$
gives $\varphi(d)\prod p^{k-1}(p-2)$ admissible pairs, and dividing by the free scaling
action of order $\varphi(d)$ gives \eqref{eq:cyclic-count}.
\end{proof}

The natural question is whether cyclic is the only possibility. It is not, and the
smallest counterexample is small enough to exhibit.

\begin{theorem}[Extremal Belyi maps in low degree]\label{thm:extremal-classification}
Enumerating the transitive triples of $d$-cycles with product $1$ in $\Sym_{d}$, up to
simultaneous conjugation, gives the following complete classification of the Belyi maps
attaining $d=2g+1$ for $g\leq3$:
\begin{center}\small
\begin{tabular}{@{}ccccl@{}}
\toprule
$g$ & $d=2g+1$ & \# extremal maps & monodromy groups (with multiplicities) & \\
\midrule
$1$ & $3$ & $1$  & $\Z/3$ $(1)$ & all cyclic\\
$2$ & $5$ & $4$  & $\Z/5$ $(3)$,\ \ $A_{5}$ $(1)$ & first non-cyclic example\\
$3$ & $7$ & $30$ & $\Z/7$ $(5)$,\ \ $\mathrm{PSL}_{2}(\mathbb{F}_{7})$ $(2)$,\ \
      $A_{7}$ $(23)$ & \\
\bottomrule
\end{tabular}
\end{center}
In particular:
\begin{enumerate}[label=\textup{(\roman*)},leftmargin=2.4em]
\item the cyclic counts $1,3,5$ agree with Proposition~\ref{prop:cyclic-extremal};
\item extremal Belyi maps need \emph{not} be cyclic: in genus $2$ there is exactly one
non-cyclic extremal map, of degree $5$ with monodromy group $A_{5}$;
\item in genus $3$ the extremal maps realize exactly three monodromy groups, and two of
them have the Klein group $\mathrm{PSL}_{2}(\mathbb{F}_{7})$ of order $168$ as monodromy,
acting on $7$ letters as on $\Pp^{1}(\mathbb{F}_{7})$.
\end{enumerate}
\end{theorem}

\begin{proof}
By Theorem~\ref{thm:degree-bound}\ref{itm:db-general} and
Theorem~\ref{thm:classification}, the extremal maps of degree $d$ are in bijection with
transitive triples $(\sigma_{0},\sigma_{1},\sigma_{\infty})$ of $d$-cycles satisfying
$\sigma_{0}\sigma_{1}\sigma_{\infty}=1$, modulo simultaneous conjugation; transitivity is
automatic since a $d$-cycle is already transitive. The enumeration in $\Sym_{d}$ for
$d=3,5,7$ is a finite computation, carried out over all
$\bigl((d-1)!\bigr)^{2}$ ordered pairs $(\sigma_{0},\sigma_{1})$ of $d$-cycles, retaining
those whose product is a $d$-cycle and grouping the survivors into $\Sym_{d}$-conjugacy
classes of pairs. The resulting counts and the orders of
$\langle\sigma_{0},\sigma_{1}\rangle$ are as tabulated; the groups of order $60$, $168$ and
$2520=7!/2$ are identified as $A_{5}$, $\mathrm{PSL}_{2}(\mathbb{F}_{7})$ and $A_{7}$,
the middle one being the unique simple group of order $168$ and the only transitive
subgroup of $\Sym_{7}$ of that order up to conjugacy. Statement (i) is
Proposition~\ref{prop:cyclic-extremal}; (ii) and (iii) read off the table.
\end{proof}

Representative permutation triples for every monodromy type, together with the
enumeration algorithm and complete verification scripts, are given in
Appendix~\ref{app:computations}.

The classification has an arithmetic consequence, obtained by feeding it back through the
descent theory of Section~\ref{sec:descent}: within a passport, $\GQ$ preserves not only
the cycle types but the isomorphism class of the monodromy group, so a dessin that is
alone with its monodromy type is Galois-fixed.

\begin{corollary}[The non-cyclic quintic extremal map is rational]\label{cor:A5-rational}
The unique extremal dessin of degree $5$ with monodromy $A_{5}$ has field of moduli $\Q$
and, its automorphism group being trivial, is defined over $\Q$; explicitly, it is
realized by a degree-$5$ Belyi map $t\colon X\to\Pp^{1}$ with $\Q$-coefficients on a
genus-$2$ curve over $\Q$, totally ramified over each of $0,1,\infty$.
\end{corollary}

\begin{proof}
Conjugation by $\sigma\in\GQ$ carries a Belyi map to one with the same passport
\emph{and} isomorphic monodromy group ($\Mon(t^{\sigma})$ is conjugate to $\Mon(t)$ in
$\Sym_{d}$, since the topological covering away from the branch points is unchanged up to
relabelling). By Theorem~\ref{thm:extremal-classification} the passport
$\langle(5),(5),(5)\rangle$ contains a unique dessin with monodromy $A_{5}$; its
$\GQ$-orbit is therefore a single point and its field of moduli is $\Q$
(Proposition~\ref{prop:galois-passport}). Its automorphism group is the centralizer of the
transitive subgroup $A_{5}\leq\Sym_{5}$, which is trivial (a centralizing element commutes
with a point stabilizer $A_{4}$, which has a unique fixed point, forcing the element to fix
every point). Theorem~\ref{thm:rigid-descent} then descends the pair $(X,t)$ to
$M(X,t)=\Q$.
\end{proof}

By the same argument the three cyclic quintic dessins are individually
$\GQ$-fixed---which is visible directly, since the models $y^{5}=x^{a}(1-x)^{b}$ have
rational coefficients---while the two dessins with monodromy
$\mathrm{PSL}_{2}(\mathbb{F}_{7})$ in degree $7$ form a $\GQ$-set of size at most $2$, so
their fields of moduli are at most quadratic; and the $23$ dessins with monodromy $A_{7}$
split into orbits whose sizes we do not determine here.

\subsection{A mass formula}\label{ssec:mass}

The counts $1,4,30$ of Theorem~\ref{thm:extremal-classification} are not arbitrary: they
are governed by an exact formula, obtained by combining Boccara's count of factorizations
of a $d$-cycle into two $d$-cycles \cite{Boccara1980} with the orbit--stabilizer
interpretation of dessins. The counting input is thus classical; the contribution here is
the translation into a weighted count of extremal dessins, the identification of the
automorphism weighting, and the integration with the equality case of
Theorem~\ref{thm:degree-bound}. We give a self-contained proof, deriving the
factorization count from the Frobenius formula along the way; the result provides the
conceptual argument complementing the finite enumeration, and an independent check of
it.

\begin{theorem}[Mass formula for extremal Belyi maps]\label{thm:mass-formula}
Let $d\geq3$ be odd. Then
\begin{equation}\label{eq:mass-formula}
  \sum_{\mathcal{D}}\frac{1}{|\Aut(\mathcal{D})|}\;=\;\frac{2\,(d-1)!}{d\,(d+1)}\,,
\end{equation}
the sum running over the isomorphism classes of extremal dessins of degree $d$, i.e.\
those with passport $\langle(d),(d),(d)\rangle$; here $\Aut(\mathcal{D})$ is the
automorphism group of the Belyi map, equal to the centralizer of
$\langle\sigma_{0},\sigma_{1}\rangle$ in $\Sym_{d}$. For $d$ even the sum is empty.
\end{theorem}

\begin{proof}
The plan: convert the automorphism-weighted count of dessins into a plain count of ordered
pairs of $d$-cycles with $d$-cycle product (orbit--stabilizer), evaluate that count by the
Frobenius class-multiplication formula, and observe that on $d$-cycles only the hook
characters survive, leaving an alternating sum of reciprocal binomial coefficients that
telescopes. Write $C\subset\Sym_{d}$ for the class of $d$-cycles, $|C|=(d-1)!$. By
Theorem~\ref{thm:classification}, extremal dessins correspond to pairs
$(\sigma_{0},\sigma_{1})\in C\times C$ with $(\sigma_{0}\sigma_{1})^{-1}\in C$, modulo
conjugation; transitivity is automatic. The group $\Sym_{d}$ acts on the set $\mathcal{P}$
of such pairs by simultaneous conjugation, the stabilizer of a pair being the centralizer
of the subgroup it generates, i.e.\ $\Aut(\mathcal{D})$. The orbit--stabilizer theorem, in the
form ``orbit size $=|G|/|\text{stabilizer}|$'', gives
\begin{equation}\label{eq:mass-orbit}
  |\mathcal{P}|\;=\;\sum_{\mathcal{D}}\frac{d!}{|\Aut(\mathcal{D})|}\,,
\end{equation}
so it suffices to show $|\mathcal{P}|=2\,((d-1)!)^{2}/(d+1)$, since then the mass equals
$|\mathcal{P}|/d!=2(d-1)!/d(d+1)$.

Fix a $d$-cycle $z$ and let $N(z)=\#\{(\sigma_{0},\sigma_{1})\in C\times C:
\sigma_{0}\sigma_{1}=z\}$, so that $|\mathcal{P}|=|C|\cdot N(z)=(d-1)!\,N(z)$ (summing
over the $(d-1)!$ possible values $z=(\sigma_{0}\sigma_{1})^{-1}$, with $N$ independent of
$z$ by conjugacy). The Frobenius class-multiplication formula gives
\begin{equation}\label{eq:frobenius}
  N(z)\;=\;\frac{|C|^{2}}{d!}\sum_{\chi\in\operatorname{Irr}(\Sym_{d})}
  \frac{\chi(c)^{2}\,\chi(z^{-1})}{\chi(1)}
  \;=\;\frac{((d-1)!)^{2}}{d!}\sum_{\chi}\frac{\chi(c)^{3}}{\chi(1)}\,,
\end{equation}
$c$ denoting a $d$-cycle ($z^{-1}$ is again one). By the Murnaghan--Nakayama rule,
$\chi_{\lambda}(c)=0$ unless $\lambda$ is a hook $(d-r,1^{r})$, in which case
$\chi_{\lambda}(c)=(-1)^{r}$ and $\chi_{\lambda}(1)=\binom{d-1}{r}$. Hence, with $n=d-1$,
\begin{equation}\label{eq:hook-sum}
  \sum_{\chi}\frac{\chi(c)^{3}}{\chi(1)}
  \;=\;\sum_{r=0}^{n}\frac{(-1)^{3r}}{\binom{n}{r}}
  \;=\;\sum_{r=0}^{n}\frac{(-1)^{r}}{\binom{n}{r}}
  \;=\;\begin{cases}\dfrac{2(n+1)}{n+2}=\dfrac{2d}{d+1}, & n \text{ even }(d\text{ odd}),\\[6pt]
  0, & n \text{ odd }(d\text{ even}),\end{cases}
\end{equation}
the evaluation of the alternating sum of reciprocal binomial coefficients being classical
(it follows from the recursion
$\binom{n}{r}^{-1}=\frac{n+1}{n+2}\bigl[\binom{n+1}{r}^{-1}+\binom{n+1}{r+1}^{-1}\bigr]$,
which telescopes the alternating sum). Substituting,
$N(z)=\frac{((d-1)!)^{2}}{d!}\cdot\frac{2d}{d+1}=\frac{2(d-1)!}{d+1}$---this is Boccara's
count of factorizations of a $d$-cycle into two $d$-cycles \cite{Boccara1980}---and
\begin{equation*}
  |\mathcal{P}|=(d-1)!\cdot\frac{2(d-1)!}{d+1}\,,\qquad
  \sum_{\mathcal{D}}\frac{1}{|\Aut(\mathcal{D})|}
  =\frac{|\mathcal{P}|}{d!}
  =\frac{2\,(d-1)!}{d\,(d+1)}\,. \qedhere
\end{equation*}
\end{proof}

\begin{remark}[Relation to the literature]\label{rem:mass-literature}
The ingredients of Theorem~\ref{thm:mass-formula} have distinct owners, and it is worth
recording who owns what. The count $N(z)=2(d-1)!/(d+1)$ of factorizations of a $d$-cycle
into two $d$-cycles is due to Boccara \cite{Boccara1980}, by a combinatorial argument;
character-theoretic evaluations of such cycle-product counts, of which
\eqref{eq:frobenius}--\eqref{eq:hook-sum} is an instance, go back to Frobenius and were
developed systematically by Jackson \cite{Jackson1987}. On the dessin side, Shabat and
Zvonkin \cite{ShabatZvonkin1994} initiated the arithmetic study of plane trees, and
Adrianov and Zvonkin \cite{AdrianovZvonkin1998} the study of their compositions; the
enumeration of dessins with prescribed passports, weighted or not, is a standing theme of
\cite{LandoZvonkin2004}. What is new here is the meeting point: the observation that
Boccara's count, read through the orbit--stabilizer dictionary of
Theorem~\ref{thm:classification}, is precisely a mass formula for the dessins attaining
the minimal Belyi degree $d=2g+1$ of Theorem~\ref{thm:degree-bound}, together with the
identification of the automorphism weighting and the two-sided bound of
Remark~\ref{rem:mass-asymptotics} that it yields.
\end{remark}

\begin{example}\label{ex:mass-check}
For $d=3$: mass $=\tfrac{2\cdot2}{3\cdot4}=\tfrac13$, matching the single cyclic dessin
with $|\Aut|=3$. For $d=5$: mass $=\tfrac{2\cdot24}{5\cdot6}=\tfrac85$, matching
$3\cdot\tfrac15+1=\tfrac85$ (three cyclic dessins with $|\Aut|=5$, one $A_{5}$-dessin with
trivial automorphisms). For $d=7$: mass $=\tfrac{2\cdot720}{7\cdot8}=\tfrac{180}{7}$,
matching $5\cdot\tfrac17+2+23=\tfrac{180}{7}$. The formula thus reproves the enumeration
totals of Theorem~\ref{thm:extremal-classification} up to the determination of the
automorphism groups. In degree $9$ it gives a mass of
$\tfrac{2\cdot 8!}{9\cdot10}=896$; the \emph{number} of extremal dessins of degree $9$
would equal $896$ exactly if every one of them had trivial automorphism group, and in any
case lies between $896$ and $8064$ by Remark~\ref{rem:mass-asymptotics}. The cyclic
subfamily contributes a known sliver: by Proposition~\ref{prop:cyclic-extremal} there are
$N_{\mathrm{cyc}}(9)=3$ cyclic extremal dessins, each with $|\Aut|=9$, accounting for
mass $\tfrac13$; the remaining mass $896-\tfrac13$ is carried by non-cyclic dessins.
\end{example}

\begin{remark}[Mass versus number]\label{rem:mass-asymptotics}
Let $N_{d}$ denote the number of extremal dessins of degree $d$ and $M_{d}$ the mass
\eqref{eq:mass-formula}. For a transitive dessin the automorphism group---the centralizer
of a transitive subgroup---acts semiregularly on the $d$ edges (a centralizing element
fixing one edge commutes with the transitive action, hence fixes all), so
$|\Aut(\mathcal{D})|$ divides $d$; thus $1\leq|\Aut(\mathcal{D})|\leq d$ and
\begin{equation*}
  \frac{2\,(d-1)!}{d(d+1)}\;=\;M_{d}\;\leq\;N_{d}\;\leq\;d\,M_{d}\;=\;\frac{2\,(d-1)!}{d+1}\,.
\end{equation*}
The mass formula alone does not determine which end of this range $N_{d}$ approaches: that
depends on the proportion of dessins with nontrivial automorphisms. The data of
Theorem~\ref{thm:extremal-classification}---all automorphism groups trivial except on the
thin cyclic subfamily---suggest that almost all extremal dessins have trivial automorphism
group, in which case $N_{d}$ would be asymptotic to $M_{d}$; establishing this requires an
independent argument and is posed as part of Problem~\ref{prob:extremal}.
\end{remark}

\begin{remark}\label{rem:extremal-curves}
Theorem~\ref{thm:extremal-classification} classifies the \emph{coverings}, not the curves.
Each extremal covering has a genus-$g$ source curve, defined over $\Qbar$ by Belyi's
theorem---over $\Q$ itself in the case of Corollary~\ref{cor:A5-rational} and of the
cyclic family---and the curves so obtained are exactly the genus-$g$ curves of minimal
possible Belyi degree. Identifying them explicitly is a concrete problem: it is natural to
ask whether the two genus-$3$ extremal maps with monodromy
$\mathrm{PSL}_{2}(\mathbb{F}_{7})$ have the Klein quartic as source---the Klein quartic is
the genus-$3$ curve whose automorphism group is that very group
(Example~\ref{ex:klein})---which would give $\mathcal{B}(\text{Klein})=7$ rather than the
$168$ suggested by the quotient map. We do not settle this here; see
Problem~\ref{prob:extremal}.
\end{remark}

\begin{remark}[Comparison and context]\label{rem:degree-context}
The bound $d\geq 2g+1$ is best possible only for very special curves: equality forces the
passport $\langle(d),(d),(d)\rangle$, a highly restrictive condition---but it does
\emph{not} force the covering to be cyclic. The cyclic extremal coverings form the
superelliptic subfamily of Proposition~\ref{prop:cyclic-extremal}, while
Theorem~\ref{thm:extremal-classification} exhibits non-cyclic extremal coverings already
in degree $5$. For a general curve of genus $g$ the Belyi degree is far larger than
$2g+1$; the known upper bounds are of a completely different order, being at least
exponential in the height of the defining data \cite{Khadjavi2002}. The gap between
$2g+1$ and the best known upper bounds is enormous, and closing it---even for a single
explicit family---is one of the concrete open problems in the area
(Problem~\ref{prob:belyi-degree}).
\end{remark}

\begin{table}[t]
\centering
\small
\begin{tabular}{@{}llccccc@{}}
\toprule
Curve & Belyi map $t$ & $\deg t$ & $g$ & passport & $2g+1$ & clean? \\
\midrule
$y^{5}=x(1-x)$    & $x$ & $5$ & $2$ & $\langle(5),(5),(5)\rangle$ & $5$ & no \\
$y^{7}=x(1-x)$    & $x$ & $7$ & $3$ & $\langle(7),(7),(7)\rangle$ & $7$ & no \\
$y^{4}=x(1-x)^{2}$& $x$ & $4$ & $1$ & $\langle(4),(2,2),(4)\rangle$ & $3$ & yes \\
$\Pp^{1}$         & $P_{m,n}$ & $m{+}n$ & $0$ &
   $\langle(m,n),(2,1^{d-2}),(d)\rangle$ & $1$ & no\\
$\Pp^{1}$ ($\lambda$-line) & $j/1728$ & $6$ & $0$ &
   $\langle(3,3),(2,2,2),(2,2,2)\rangle$ & $1$ & yes\\
Fermat $F_{N}$    & $(x/z)^{N}$ & $N^{2}$ & $\binom{N-1}{2}$ &
   $\langle(N^{N}),(N^{N}),(N^{N})\rangle$ & $N^{2}{-}3N{+}3$ & no\\
Klein quartic     & quotient by $\mathrm{PSL}_{2}(\mathbb{F}_{7})$ & $168$ & $3$ &
   $\langle(2^{84}),(3^{56}),(7^{24})\rangle$ & $7$ & yes\\
\bottomrule
\end{tabular}
\caption{The examples computed in this article, with their invariants. The notation
$(N^{N})$ means $N$ cycles of length $N$. All entries satisfy
\eqref{eq:belyi-RH} and Theorem~\ref{thm:degree-bound}; rows one to three are extremal.}
\label{tab:examples}
\end{table}
\section{Worked examples}\label{sec:examples}

Belyi's theorem is an existence statement, and existence statements are best understood
against explicit instances. The four examples of this section are chosen because each
illustrates a different mechanism: a Fermat curve, where the Belyi map is a coordinate and
the curve has large genus; a family of cyclic covers realizing the extremal degrees of
Section~\ref{sec:degree}; the $\lambda\mapsto j$ map, where the Belyi structure is the
classical modular one; and the Klein quartic, where it comes from a large automorphism
group. All ramification data below are computed, not quoted.

\subsection{Fermat curves}\label{ssec:fermat}

\begin{example}[Fermat curves]\label{ex:fermat}
Let $F_{N}\subset\Pp^{2}$ be the Fermat curve $x^{N}+y^{N}=z^{N}$, smooth of genus
$g=\binom{N-1}{2}=\frac{(N-1)(N-2)}{2}$, and set
\begin{equation}\label{eq:fermat-belyi}
  t\colon F_{N}\longrightarrow\Pp^{1},\qquad
  t\bigl([x:y:z]\bigr)=\Bigl(\frac{x}{z}\Bigr)^{\!N} .
\end{equation}
Writing $u=x/z$, $v=y/z$, so that $u^{N}+v^{N}=1$, the map is $t=u^{N}=1-v^{N}$, of degree
$N^{2}$: it is the composition of the degree-$N$ projection $[x:y:z]\mapsto u$ with
$u\mapsto u^{N}$, or equivalently $F_{N}\to\Pp^{1}$ is the fibre product of two cyclic
covers. Its fibres are
\begin{align*}
  t^{-1}(0)&=\{u=0\}=\bigl\{[0:v:1]:v^{N}=1\bigr\}: &&N\ \text{points, each of index }N,\\
  t^{-1}(1)&=\{v=0\}=\bigl\{[u:0:1]:u^{N}=1\bigr\}: &&N\ \text{points, each of index }N,\\
  t^{-1}(\infty)&=\{z=0\}=\bigl\{[u:v:0]:u^{N}+v^{N}=0\bigr\}: &&N\ \text{points, each of index }N,
\end{align*}
so $n_{0}=n_{1}=n_{\infty}=N$ and $t$ is a Belyi map with passport
$\langle (N^{N}),(N^{N}),(N^{N})\rangle$. Riemann--Hurwitz \eqref{eq:belyi-RH} confirms the
genus:
\begin{equation*}
  2g-2=-2N^{2}+3N(N-1)=N^{2}-3N \quad\Longrightarrow\quad g=\frac{N^{2}-3N+2}{2}=\frac{(N-1)(N-2)}{2},
\end{equation*}
matching the degree--genus formula. Thus $\mathcal{B}(F_{N})\leq N^{2}$, while
Theorem~\ref{thm:degree-bound} gives $\mathcal{B}(F_{N})\geq 2g+1=N^{2}-3N+3$: the
coordinate Belyi map is within $3N-3$ of optimal, which for the Fermat curves is a
remarkably tight sandwich. The dessin of \eqref{eq:fermat-belyi} is the ``$N\times N$
grid'' on the genus-$g$ surface, with $N$ black and $N$ white vertices all of degree $N$
and $N$ faces.
\end{example}

\subsection{Cyclic covers and the extremal degrees}\label{ssec:cyclic}

Example~\ref{ex:superelliptic} already exhibited the extremal family. It is worth
recording the general ramification computation once, since it is the standard source of
explicit Belyi maps.

\begin{proposition}\label{prop:cyclic-ram}
Let $N\geq2$ and $a,b\geq1$ with $\gcd(a,b,N)=1$, and let $X_{N,a,b}$ be the smooth model
of $y^{N}=x^{a}(1-x)^{b}$ with $t=x$. Then $t$ is a Galois Belyi map with group $\Z/N$,
ramification indices \eqref{eq:superelliptic-ram} and genus
\begin{equation}\label{eq:cyclic-genus}
  g=1+\frac{1}{2}\Bigl(N-\gcd(a,N)-\gcd(b,N)-\gcd(a+b,N)\Bigr).
\end{equation}
\end{proposition}

\begin{proof}
The extension $\C(x,y)/\C(x)$ is Kummer of degree $N$, cyclic with $y\mapsto\zeta_{N}y$;
it is unramified away from the zeros and poles of $x^{a}(1-x)^{b}$, i.e.\ away from
$0,1,\infty$. At $x=0$ the valuation of $x^{a}(1-x)^{b}$ is $a$, so the local extension is
$\C((x))\bigl[y\bigr]/(y^{N}-x^{a}\cdot\text{unit})$, which splits into $\gcd(a,N)$ places
each of ramification index $N/\gcd(a,N)$; similarly at $x=1$ with $b$, and at $x=\infty$
where the valuation is $-(a+b)$. Substituting $n_{s}=\gcd(\cdot,N)$ into
\eqref{eq:belyi-RH} gives \eqref{eq:cyclic-genus}.
\end{proof}

\begin{table}[h]
\centering\small
\begin{tabular}{@{}cccccccc@{}}
\toprule
$N$ & $a$ & $b$ & $\ram_{0}$ & $\ram_{1}$ & $\ram_{\infty}$ & $g$ & remark\\
\midrule
$3$&$1$&$1$&$3$&$3$&$3$&$1$& extremal, $d=2g+1$\\
$5$&$1$&$1$&$5$&$5$&$5$&$2$& extremal, $d=2g+1$\\
$7$&$1$&$1$&$7$&$7$&$7$&$3$& extremal, $d=2g+1$\\
$9$&$1$&$1$&$9$&$9$&$9$&$4$& extremal, $d=2g+1$\\
$4$&$1$&$2$&$4$&$2$&$4$&$1$& clean, extremal, $d=4g$\\
$8$&$1$&$2$&$8$&$4$&$8$&$3$& $d=8$, $4g=12$: not clean\\
$6$&$1$&$2$&$6$&$3$&$2$&$1$& $n_{\infty}=3$\\
\bottomrule
\end{tabular}
\caption{Cyclic Belyi covers $y^{N}=x^{a}(1-x)^{b}$ computed from
Proposition~\ref{prop:cyclic-ram}.}
\label{tab:cyclic}
\end{table}

\subsection{The modular example: $\lambda\mapsto j$}\label{ssec:lambda-j}

\begin{example}[The degree-six modular Belyi map]\label{ex:lambda-j}
The $\lambda$-line parametrizes elliptic curves with full level-$2$ structure via the
Legendre family $E_{\lambda}\colon y^{2}=x(x-1)(x-\lambda)$, and the forgetful map to the
$j$-line is
\begin{equation}\label{eq:j-lambda}
  j(\lambda)=256\,\frac{(\lambda^{2}-\lambda+1)^{3}}{\lambda^{2}(\lambda-1)^{2}} .
\end{equation}
Set $t=j/1728$. We claim that
$t\colon\Pp^{1}_{\lambda}\to\Pp^{1}_{t}$ is a \emph{clean} Belyi map of degree $6$, with
passport $\langle(3,3),(2,2,2),(2,2,2)\rangle$.

Indeed, from \eqref{eq:j-lambda},
\begin{equation}\label{eq:t-lambda}
  t=\frac{4\,(\lambda^{2}-\lambda+1)^{3}}{27\,\lambda^{2}(\lambda-1)^{2}},
  \qquad
  t-1=\frac{\bigl((\lambda+1)(2\lambda-1)(\lambda-2)\bigr)^{2}}
            {27\,\lambda^{2}(\lambda-1)^{2}} ,
\end{equation}
the second identity being equivalent to the classical relation
\begin{equation}\label{eq:disc-identity}
  4\,(\lambda^{2}-\lambda+1)^{3}-\bigl((\lambda+1)(2\lambda-1)(\lambda-2)\bigr)^{2}
  \;=\;27\,\lambda^{2}(\lambda-1)^{2},
\end{equation}
which one verifies by expanding both sides. Reading off \eqref{eq:t-lambda}:
\begin{itemize}[leftmargin=1.6em]
\item over $t=0$: the two roots of $\lambda^{2}-\lambda+1$, i.e.\ the primitive sixth
roots of unity, each with multiplicity $3$. So $n_{0}=2$, indices $(3,3)$.
\item over $t=1$: the three roots $\lambda=-1,\ \tfrac12,\ 2$, each with multiplicity
$2$. So $n_{1}=3$, indices $(2,2,2)$: the map is clean.
\item over $t=\infty$: the poles $\lambda=0,1$ of order $2$, and $\lambda=\infty$, where
the numerator has degree $6$ and the denominator degree $4$, so again order $2$. So
$n_{\infty}=3$, indices $(2,2,2)$.
\end{itemize}
The degree is $6$ in each case, and \eqref{eq:belyi-RH} gives
$2g-2=6-(2+3+3)=-2$, i.e.\ $g=0$, as it must. The three special fibres are exactly the
elliptic points of order $3$ ($j=0$), the elliptic points of order $2$ ($j=1728$) and the
cusps ($j=\infty$) of $\mathrm{SL}_{2}(\Z)$, and the six sheets are the six cosets of
$\Gamma(2)$ in $\mathrm{SL}_{2}(\Z)/\{\pm1\}$, permuted by
$\mathrm{SL}_{2}(\Z/2)\cong\Sym_{3}$; the values $\lambda=-1,\tfrac12,2$ and
$\lambda=0,1,\infty$ are the two orbits of size $3$ of this $\Sym_{3}$-action on the
$\lambda$-line, and the primitive sixth roots of unity form the orbit of size $2$.
Since $\Mon(t)$ is transitive of degree $6$ and $t$ is the quotient by
$\Sym_{3}$ composed with nothing else, $t$ is a Galois covering with group $\Sym_{3}$,
$|{\Mon}(t)|=6=\deg t$.
\end{example}

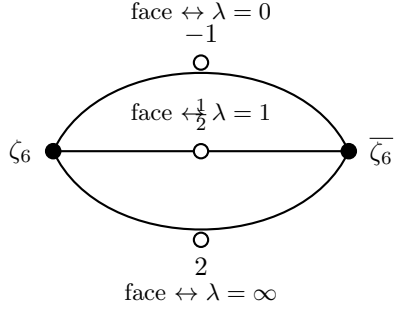
\begin{figure}[t]
\centering
\begin{tikzpicture}[scale=1.15,
  bl/.style={circle,fill=black,inner sep=2.2pt},
  wh/.style={circle,draw=black,fill=white,thick,inner sep=1.9pt},
  ed/.style={thick}]
\coordinate (B1) at (-1.7,0);
\coordinate (B2) at (1.7,0);
\draw[ed] (B1) to[out=65,in=115] (B2);
\draw[ed] (B1) -- (B2);
\draw[ed] (B1) to[out=-65,in=-115] (B2);
\node[wh] at (0,1.02) {};
\node[wh] at (0,0) {};
\node[wh] at (0,-1.02) {};
\node[bl] at (B1) {};
\node[bl] at (B2) {};
\node[left=4pt] at (B1) {\small $\zeta_{6}$};
\node[right=4pt] at (B2) {\small $\overline{\zeta_{6}}$};
\node[above=3pt] at (0,1.02) {\small $-1$};
\node[above=3pt] at (0,0.03) {\small $\tfrac12$};
\node[below=3pt] at (0,-1.02) {\small $2$};
\node[font=\footnotesize] at (0,1.62) {face $\leftrightarrow\lambda=0$};
\node[font=\footnotesize] at (0,0.45) {face $\leftrightarrow\lambda=1$};
\node[font=\footnotesize] at (0,-1.62) {face $\leftrightarrow\lambda=\infty$};
\end{tikzpicture}
\caption{The dessin of the modular Belyi map $t=j/1728$ of Example~\ref{ex:lambda-j},
drawn on the Riemann sphere: the ``theta graph'' with two black vertices of degree $3$
(the primitive sixth roots of unity, lying over $j=0$), three white vertices of degree
$2$ (the points $-1,\tfrac12,2$ over $j=1728$), and three faces (the cusps
$\lambda=0,1,\infty$). Euler: $5-6+3=2$, so $g=0$, in agreement with
\eqref{eq:belyi-RH}.}
\label{fig:theta-dessin}
\end{figure}

\subsection{The Klein quartic}\label{ssec:klein}

\begin{example}[Klein quartic; the Hurwitz bound as a passport]\label{ex:klein}
Let $K\subset\Pp^{2}$ be the Klein quartic $x^{3}y+y^{3}z+z^{3}x=0$, a smooth plane
quartic, hence of genus $g=\binom{4-1}{2}=3$. Its automorphism group is
$G=\mathrm{PSL}_{2}(\mathbb{F}_{7})$ of order $168=84(g-1)$, the maximum allowed by
Hurwitz's theorem. The quotient map
\begin{equation*}
  t\colon K\longrightarrow K/G\cong\Pp^{1}
\end{equation*}
is Galois of degree $168$, and is a Belyi map: by Riemann--Hurwitz for a Galois quotient
with signature $(0;p,q,r)$,
\begin{equation*}
  2g-2=|G|\Bigl(-2+\bigl(1-\tfrac1p\bigr)+\bigl(1-\tfrac1q\bigr)+\bigl(1-\tfrac1r\bigr)\Bigr),
\end{equation*}
and $(p,q,r)=(2,3,7)$ gives
$168\bigl(-2+\tfrac12+\tfrac23+\tfrac67\bigr)=168\cdot\tfrac{1}{42}=4=2\cdot3-2$,
confirming $g=3$. The passport is
$\langle(2^{84}),(3^{56}),(7^{24})\rangle$, with $n_{0}=84$, $n_{1}=56$,
$n_{\infty}=24$; note $\eqref{eq:belyi-RH}$: $168-(84+56+24)=4=2g-2$. Since all indices
over $1$ would have to be $2$ for cleanliness, the map as written is not clean; the
labelling can be permuted so that the index-$2$ points lie over $1$, and then it is.
The triangle group $\Delta(2,3,7)$ is the one of smallest positive area, which is
precisely why $84(g-1)$ is the Hurwitz bound---the Belyi/triangle-group picture makes the
bound a statement about hyperbolic area rather than about curves.
\end{example}
\section{The Galois action on dessins: an explicit orbit over
\texorpdfstring{$\Q(\sqrt[3]{2})$}{Q(cbrt 2)}}
\label{sec:galois-dessins}

Belyi's theorem gives $\GQ$ a set to act on. If $t\colon X\to\Pp^{1}$ is a Belyi map
defined over $\Qbar$ and $\sigma\in\GQ$, then $t^{\sigma}\colon X^{\sigma}\to\Pp^{1}$ is
again a Belyi map, of the same degree and with the same ramification indices---conjugation
cannot change the combinatorics of a fibre. Hence:

\begin{proposition}\label{prop:galois-passport}
$\GQ$ acts on the set of isomorphism classes of dessins, preserving the passport
\eqref{eq:passport}. The stabilizer of a dessin $\mathcal{D}$ is an open subgroup, and its
fixed field is the field of moduli $M(X,t)$; the orbit of $\mathcal{D}$ is finite of size
$[M(X,t):\Q]$.
\end{proposition}

\begin{proof}
Invariance of the passport is the remark above. By Proposition~\ref{prop:finiteness} the
set of dessins of given degree is finite, so all orbits are finite and stabilizers have
finite index; the stabilizer is $U(X,t)\cap\GQ$ by
Definition~\ref{def:moduli-field}, whose fixed field is $M(X,t)$ by
Lemma~\ref{lem:aut-C}, and the orbit--stabilizer theorem gives the size.
\end{proof}

\begin{corollary}[Passport bound on the field of moduli]\label{cor:passport-bound}
Let $t$ be a Belyi map with passport $\Pi$ and let $N(\Pi)$ be the number of isomorphism
classes of dessins with passport $\Pi$. Then
\begin{equation*}
  \bigl[M(X,t):\Q\bigr]\;\leq\;N(\Pi).
\end{equation*}
In particular, if the passport contains a single dessin, then $M(X,t)=\Q$; this
uniqueness is analogous to---though not identical with---the rigidity phenomenon of
inverse Galois theory, where the relevant condition is uniqueness of a generating tuple in
prescribed rational conjugacy classes of the monodromy group up to conjugacy
\cite{MalleMatzat1999,Serre1992}.
\end{corollary}

\begin{proof}
By Proposition~\ref{prop:galois-passport} the degree of the field of moduli is the size of
the $\GQ$-orbit, which is contained in the set of dessins with passport $\Pi$.
\end{proof}

Corollary~\ref{cor:passport-bound} explains Corollary~\ref{cor:pmn-rigid} after the fact:
the passport \eqref{eq:pmn-passport} of $P_{m,n}$ has $N(\Pi)=1$, so $P_{m,n}$ must be
definable over $\Q$---as of course it visibly is. The interest of the bound is in the
opposite direction, when $N(\Pi)>1$ and one wants to know how large the field of moduli
can be; the computations below realize the bound exactly, with $N(\Pi)=3$ and a cubic
field, and with $N(\Pi)=2$ and a quadratic field.

The action is faithful, and remains so on plane trees \cite{Schneps1994,LandoZvonkin2004};
computing orbits is therefore a way of touching $\GQ$ with bare hands. Passports are the
first invariant, but they do not separate orbits, and the interesting phenomena begin
where a single passport contains several dessins. We now carry out such a computation in
full, in the smallest degree where the answer is a cubic field.

\subsection{Setting up}\label{ssec:orbit-setup}

By \S\ref{ssec:dessins}, plane trees with $d$ edges correspond to polynomial Belyi maps of
degree $d$, uniquely up to the affine substitutions $z\mapsto\alpha z+\beta$. Fix the
passport
\begin{equation}\label{eq:passport-321}
  \Pi=\bigl\langle(3,2,1),\ (2,2,1,1),\ (6)\bigr\rangle ,
\end{equation}
so $d=6$: three black vertices of degrees $3,2,1$, four white vertices of degrees
$2,2,1,1$, and one face. There are $3+4=7$ vertices and $6$ edges, consistent with a tree,
and \eqref{eq:belyi-RH} gives $2g-2=6-(3+4+1)=-2$, so $g=0$.

Normalize the black vertex of degree $3$ to $z=0$ and the one of degree $2$ to $z=1$; the
remaining black vertex is at an unknown $z=a$. Then
\begin{equation}\label{eq:p-normalized}
  p(z)=c\,z^{3}(z-1)^{2}(z-a),
\end{equation}
and the white vertices are the roots of $p-1$, with multiplicities $2,2,1,1$. Since
$p'$ vanishes to order $\alpha-1$ at a black vertex of degree $\alpha$ and to order
$\beta-1$ at a white vertex of degree $\beta$, we get
\begin{equation}\label{eq:pprime}
  p'(z)=6c\,z^{2}(z-1)\bigl(z^{2}-\tfrac{4+5a}{6}z+\tfrac{a}{2}\bigr)
       =c\,z^{2}(z-1)\bigl(6z^{2}-(4+5a)z+3a\bigr),
\end{equation}
as one checks by differentiating \eqref{eq:p-normalized}. Thus the two white vertices of
degree $2$ are the roots $u,v$ of the quadratic factor:
\begin{equation}\label{eq:uv}
  u+v=\frac{4+5a}{6},\qquad uv=\frac{a}{2},
\end{equation}
and the conditions defining a Belyi map are exactly
\begin{equation}\label{eq:belyi-conditions}
  p(u)=p(v)=1 .
\end{equation}

\begin{lemma}\label{lem:orbit-equations}
Conditions \eqref{eq:belyi-conditions} are equivalent to the system
\begin{align}
  \bigl(3125a^{6}-7500a^{5}+4800a^{4}-112a^{3}-96a^{2}+1536a-1024\bigr)\,c
  &= -93312, \label{eq:E1}\\[2pt]
  a^{4}(a-1)^{3}\,c^{2}&=432 . \label{eq:E2}
\end{align}
Moreover $c$ is determined by $a$ through the linear equation \eqref{eq:E1}; in particular
every solution satisfies $c\in\Q(a)$, so the corresponding Belyi map is defined over
$\Q(a)$.
\end{lemma}

\begin{proof}
The pair of conditions \eqref{eq:belyi-conditions} is equivalent to
$p(u)+p(v)=2$ and $p(u)p(v)=1$. Both left-hand sides are symmetric polynomials in $u,v$,
hence polynomials in $u+v$ and $uv$; substituting \eqref{eq:uv} and clearing denominators
gives \eqref{eq:E1} and \eqref{eq:E2} respectively. Equation \eqref{eq:E1} is linear in
$c$ with a nonzero coefficient for the solutions in question, whence $c\in\Q(a)$.
\end{proof}

The elimination must be certified: clearing denominators and taking resultants can create
extraneous solutions, and the normalization \eqref{eq:p-normalized} must be shown to
biject solutions with isomorphism classes. The following lemma, whose verifications are
documented line by line in Appendix~\ref{app:computations}, does both.

\begin{lemma}[No extraneous solutions; bijection with dessins]\label{lem:orbit-bijection}
Let $R(a)=-432\,(25a^{2}-32a+16)^{3}(25a^{3}-12a^{2}-24a-16)^{2}$ be the resultant
\eqref{eq:resultant}.
\begin{enumerate}[label=\textup{(\roman*)},leftmargin=2.4em]
\item The coefficient of $c$ in \eqref{eq:E1} is coprime to both irreducible factors of
$R$: it vanishes at no root of either. Consequently every root $a$ of $R$ determines a
unique $c=c(a)\in\Q(a)^{\times}$, and the pair $(a,c(a))$ satisfies both \eqref{eq:E1}
and \eqref{eq:E2}; no solutions are lost or gained in the elimination.
\item Neither factor of $R$ vanishes at $a\in\{0,1\}$, so every solution
$p=c\,z^{3}(z-1)^{2}(z-a)$ is a genuine degree-$6$ polynomial with three distinct black
vertices of degrees $3,2,1$, and its critical values are exactly $\{0,1\}$: it is a Belyi
polynomial.
\item The discriminant $(4+5a)^{2}-72a$ of the quadratic factor of $p'$ in
\eqref{eq:pprime} vanishes identically modulo $25a^{2}-32a+16$ and is coprime to
$25a^{3}-12a^{2}-24a-16$. Hence the roots of the quadratic factor of $R$ give white
degrees $(3,1,1,1)$ and those of the cubic factor give $(2,2,1,1)$: the two passports
separate exactly along the factorization of $R$.
\item The normalization is rigid: an affine map fixing $0$ and $1$ is the identity, and
the black vertices of degrees $3,2,1$ are distinguishable, so distinct solutions
$(a,c)$ give non-isomorphic Belyi polynomials. Together with
Proposition~\ref{prop:dessin-dictionary} and the matching combinatorial counts
($3$ and $2$), the solutions of each factor biject with the isomorphism classes of trees
of the corresponding passport.
\end{enumerate}
\end{lemma}

\begin{proof}
(i) Compute $\gcd$ of the degree-$6$ coefficient polynomial of $c$ in \eqref{eq:E1} with
each factor of $R$: both gcds are $1$ over $\Q$ (Appendix~\ref{app:computations}). Since
\eqref{eq:E1} is linear in $c$, each root $a$ of $R$ gives exactly one $c(a)$, nonzero
because the constant term $93312\neq0$; substituting into \eqref{eq:E2} and using that
$R$ is precisely the resultant shows the pair satisfies both equations.
(ii) Evaluate: the cubic takes values $-16$ and $-27$ at $a=0,1$, the quadratic $16$ and
$9$; all nonzero. That the critical values are $\{0,1\}$ is then the defining property of
the system \eqref{eq:belyi-conditions}, verified symbolically at each root.
(iii) Polynomial division: the discriminant reduces to $0$ modulo the quadratic factor and
has gcd $1$ with the cubic factor.
(iv) An affine substitution preserving the normalization must fix the black vertices of
degrees $3$ and $2$, i.e.\ fix $0$ and $1$, hence be the identity; so the assignment
$(a,c)\mapsto p$ is injective on isomorphism classes. Surjectivity onto each passport
holds because every tree of the passport, normalized this way, satisfies
\eqref{eq:E1}--\eqref{eq:E2}, hence appears among the roots of $R$; the counts match
($3+2$ solutions, $3+2$ trees), so the correspondence is a bijection passport by
passport.
\end{proof}

\subsection{The cubic orbit}\label{ssec:cubic-orbit}

\begin{theorem}[{A $\GQ$-orbit of three plane trees over $\Q(\sqrt[3]{2})$}]
\label{thm:cubic-orbit}
Eliminating $c$ from \eqref{eq:E1}--\eqref{eq:E2} gives
\begin{equation}\label{eq:resultant}
  \operatorname{Res}_{c}\bigl(\eqref{eq:E1},\eqref{eq:E2}\bigr)
  \;=\;-432\,\bigl(25a^{2}-32a+16\bigr)^{3}\bigl(25a^{3}-12a^{2}-24a-16\bigr)^{2}.
\end{equation}
The passport \eqref{eq:passport-321} corresponds exactly to the cubic factor: it contains
exactly three plane trees $T_{1},T_{2},T_{3}$, forming a single $\GQ$-orbit, indexed by
the three roots $a_{1},a_{2},a_{3}$ of the irreducible cubic
\begin{equation}\label{eq:cubic-field}
  25a^{3}-12a^{2}-24a-16=0 .
\end{equation}
The field of moduli of the tree attached to $a_{i}$ is the embedded cubic field
$\Q(a_{i})\subset\C$, which is also a field of definition for it. The three fields
$\Q(a_{1}),\Q(a_{2}),\Q(a_{3})$ are distinct subfields of $\C$---the conjugate embeddings
of one abstract \emph{non-Galois} cubic field, abstractly isomorphic to
$\Q(\sqrt[3]{2})$---and they are permuted by $\GQ$ along with the trees:
$M(T^{\sigma})=\sigma\bigl(M(T)\bigr)$, since
$\operatorname{Stab}(T^{\sigma})=\sigma\operatorname{Stab}(T)\sigma^{-1}$. Explicitly,
the real root is
\begin{equation}\label{eq:a-explicit}
  a=\frac{4+18\sqrt[3]{2}+6\sqrt[3]{4}}{25}
  \;\approx\;1.448119408,
\end{equation}
and the corresponding Belyi polynomial is \eqref{eq:p-normalized} with
\begin{equation}\label{eq:c-explicit}
  c=\frac{-78125}{1968\sqrt[3]{2}+681\sqrt[3]{4}-1196}\;\approx\;-33.040186803 .
\end{equation}
The three trees are drawn in Figure~\ref{fig:three-trees}.
\end{theorem}

\begin{proof}
The resultant \eqref{eq:resultant} is a direct computation from \eqref{eq:E1} and
\eqref{eq:E2}, and its factorization over $\Q$ is as displayed. We must match the two
irreducible factors with the two possible passports.

The quadratic factor $25a^{2}-32a+16$ has roots $a=(16\pm12i)/25$. For these, the
discriminant of the quadratic in \eqref{eq:pprime} is
$(4+5a)^{2}-72a$, which at $a=(16\pm12i)/25$ evaluates to
$\bigl(\tfrac{36\pm12i}{5}\bigr)^{2}-\tfrac{72(16\pm12i)}{25}=0$: the roots $u,v$
\emph{coincide}. Then $p-1$ has a triple root at $u=v$, so the white degrees are
$(3,1,1,1)$, not $(2,2,1,1)$: this factor belongs to the neighbouring passport
$\langle(3,2,1),(3,1,1,1),(6)\rangle$ of Theorem~\ref{thm:quadratic-orbit}, not to
\eqref{eq:passport-321}.

The cubic factor therefore carries the passport \eqref{eq:passport-321}. It is
irreducible over $\Q$ (no rational root, by the rational root test), so its three roots
form a single $\GQ$-orbit. One checks that
$a_{0}=(4+18\sqrt[3]{2}+6\sqrt[3]{4})/25$ satisfies the cubic, so
$\Q(a_{0})\subseteq\Q(\sqrt[3]{2})$, with equality by comparing degrees. Solving
\eqref{eq:E1} for $c$ at $a=a_{0}$ gives \eqref{eq:c-explicit}, and a direct evaluation
confirms that the resulting $p$ of \eqref{eq:p-normalized} has critical values exactly
$\{0,1\}$, i.e.\ is a Belyi polynomial; the analogous verification at the other two roots
is part of Lemma~\ref{lem:orbit-bijection}.

It remains to see that the count is right---that the passport contains exactly three
dessins, no more. This is a purely combinatorial check, independent of the algebra: by
Theorem~\ref{thm:classification} the dessins with passport \eqref{eq:passport-321} are the
transitive triples $\sigma_{0}\sigma_{1}\sigma_{\infty}=1$ in $\Sym_{6}$ with cycle types
$(3,2,1)$, $(2,2,1,1)$, $(6)$, taken up to simultaneous conjugation. Enumerating these
gives exactly three classes, represented by
\begin{align*}
  \sigma_{0}&=(1\,3\,6)(2\,4), &\sigma_{1}&=(1\,4)(3\,5),\\
  \sigma_{0}&=(1\,3\,6)(2\,4), &\sigma_{1}&=(2\,3)(4\,5),\\
  \sigma_{0}&=(1\,3\,6)(2\,4), &\sigma_{1}&=(1\,2)(5\,6),
\end{align*}
(labels $1,\dots,6$ for the edges). Three trees, three roots of an irreducible cubic, a
single orbit: the normalization \eqref{eq:p-normalized} sets up a bijection between the
trees of the passport and the roots of the cubic (Lemma~\ref{lem:orbit-bijection} below
makes this precise), and $\GQ$ acts compatibly, permuting the roots transitively---so the
orbit is single, and $\GQ$ moves the tree attached to $a_{i}$ to the one attached to
$\sigma(a_{i})$. Finally, fix $i$ and write $T_{i}$ for the tree of $a_{i}$. By
Lemma~\ref{lem:orbit-equations}, $T_{i}$ is defined over $\Q(a_{i})$, so
$M(T_{i})\subseteq\Q(a_{i})$; by Proposition~\ref{prop:galois-passport},
$[M(T_{i}):\Q]$ equals the orbit size $3=[\Q(a_{i}):\Q]$; hence
$M(T_{i})=\Q(a_{i})$, an embedded cubic subfield of $\C$ which is simultaneously the
field of moduli and a field of definition. Applying $\sigma\in\GQ$ conjugates the
stabilizer, so $M(T_{i}^{\sigma})=\sigma(\Q(a_{i}))=\Q(\sigma(a_{i}))$: the three trees
carry the three conjugate embeddings, one field each, not one common field. The splitting
field of \eqref{eq:cubic-field} has Galois group $\Sym_{3}$ (discriminant $-5038848$, not
a square), so the abstract field is non-Galois; and
$a_{0}=(4+18\sqrt[3]{2}+6\sqrt[3]{4})/25$ shows the real embedding equals
$\Q(\sqrt[3]{2})$.
\end{proof}

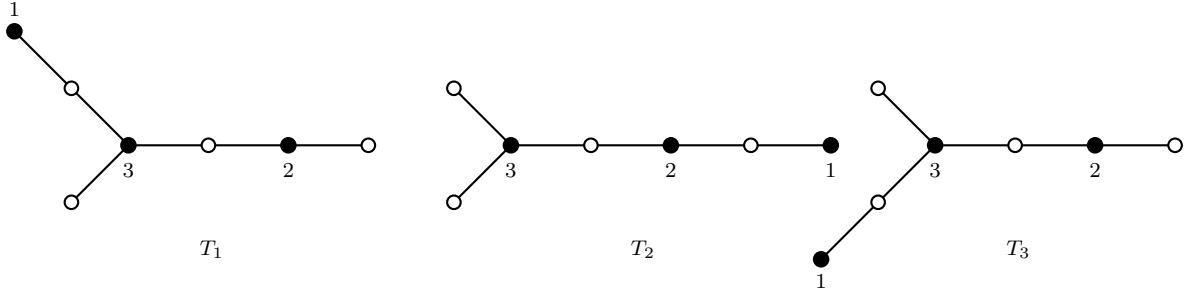
\begin{figure}[t]
\centering
\begin{tikzpicture}[scale=0.92,
  bl/.style={circle,fill=black,inner sep=2.2pt},
  wh/.style={circle,draw=black,fill=white,thick,inner sep=1.8pt},
  ed/.style={thick},
  lbl/.style={font=\scriptsize}]
\begin{scope}[xshift=0cm]
  \coordinate (B3) at (0,0);
  \coordinate (W1) at (1.15,0);
  \coordinate (B2) at (2.3,0);
  \coordinate (L1) at (3.45,0);
  \coordinate (W2) at (-0.82,0.82);
  \coordinate (B1) at (-1.64,1.64);
  \coordinate (L2) at (-0.82,-0.82);
  \draw[ed] (B3)--(W1)--(B2)--(L1);
  \draw[ed] (B3)--(W2)--(B1);
  \draw[ed] (B3)--(L2);
  \node[bl] at (B3) {}; \node[bl] at (B2) {}; \node[bl] at (B1) {};
  \node[wh] at (W1) {}; \node[wh] at (W2) {};
  \node[wh] at (L1) {}; \node[wh] at (L2) {};
  \node[lbl,below=3pt] at (B3) {$3$};
  \node[lbl,below=3pt] at (B2) {$2$};
  \node[lbl,above=2pt] at (B1) {$1$};
  \node[lbl] at (1.2,-1.5) {$T_{1}$};
\end{scope}
\begin{scope}[xshift=5.5cm]
  \coordinate (B3) at (0,0);
  \coordinate (Wa) at (1.15,0);
  \coordinate (B2) at (2.3,0);
  \coordinate (Wb) at (3.45,0);
  \coordinate (B1) at (4.6,0);
  \coordinate (L1) at (-0.82,0.82);
  \coordinate (L2) at (-0.82,-0.82);
  \draw[ed] (B3)--(Wa)--(B2)--(Wb)--(B1);
  \draw[ed] (B3)--(L1); \draw[ed] (B3)--(L2);
  \node[bl] at (B3) {}; \node[bl] at (B2) {}; \node[bl] at (B1) {};
  \node[wh] at (Wa) {}; \node[wh] at (Wb) {};
  \node[wh] at (L1) {}; \node[wh] at (L2) {};
  \node[lbl,below=3pt] at (B3) {$3$};
  \node[lbl,below=3pt] at (B2) {$2$};
  \node[lbl,below=3pt] at (B1) {$1$};
  \node[lbl] at (1.9,-1.5) {$T_{2}$};
\end{scope}
\begin{scope}[xshift=11.6cm]
  \coordinate (B3) at (0,0);
  \coordinate (W1) at (1.15,0);
  \coordinate (B2) at (2.3,0);
  \coordinate (L1) at (3.45,0);
  \coordinate (W2) at (-0.82,-0.82);
  \coordinate (B1) at (-1.64,-1.64);
  \coordinate (L2) at (-0.82,0.82);
  \draw[ed] (B3)--(W1)--(B2)--(L1);
  \draw[ed] (B3)--(W2)--(B1);
  \draw[ed] (B3)--(L2);
  \node[bl] at (B3) {}; \node[bl] at (B2) {}; \node[bl] at (B1) {};
  \node[wh] at (W1) {}; \node[wh] at (W2) {};
  \node[wh] at (L1) {}; \node[wh] at (L2) {};
  \node[lbl,below=3pt] at (B3) {$3$};
  \node[lbl,below=3pt] at (B2) {$2$};
  \node[lbl,below=2pt] at (B1) {$1$};
  \node[lbl] at (1.2,-1.5) {$T_{3}$};
\end{scope}
\end{tikzpicture}
\caption{The three plane trees of the passport
$\langle(3,2,1),(2,2,1,1),(6)\rangle$, permuted transitively by $\GQ$; the tree attached
to the root $a_{i}$ of $25a^{3}-12a^{2}-24a-16$ has field of moduli the embedded cubic
field $\Q(a_{i})$, the three of which are the conjugate embeddings of the abstract
non-Galois field $\Q(\sqrt[3]{2})$ (Theorem~\ref{thm:cubic-orbit}). Numbers are
black-vertex degrees. As
abstract graphs $T_{1}$ and $T_{3}$ coincide; they differ as \emph{plane} trees, by the
cyclic order at the degree-$3$ vertex---they are mirror images. $T_{2}$ is a different
abstract tree. No purely graph-theoretic invariant of the underlying abstract trees can
therefore separate this orbit; the planar structure is essential.}
\label{fig:three-trees}
\end{figure}

\subsection{The neighbouring quadratic orbit}\label{ssec:quadratic-orbit}

The discarded factor of \eqref{eq:resultant} is not a degeneration but a second, smaller
orbit.

\begin{theorem}\label{thm:quadratic-orbit}
The passport $\bigl\langle(3,2,1),(3,1,1,1),(6)\bigr\rangle$ contains exactly two plane
trees. They form a single $\GQ$-orbit; the field of moduli of each is the imaginary
quadratic field $\Q(i)$---here a genuinely common embedded field, since $\Q(i)/\Q$ is
Galois and hence stable under conjugation of stabilizers. In the normalization
\eqref{eq:p-normalized} they are given by the two roots
\begin{equation*}
  a=\frac{16\pm 12\,i}{25}
\end{equation*}
of the quadratic factor $25a^{2}-32a+16$ of \eqref{eq:resultant}.
\end{theorem}

\begin{proof}
As shown in the proof of Theorem~\ref{thm:cubic-orbit}, at these values of $a$ the two
critical points $u,v$ of \eqref{eq:pprime} collide, so $p-1$ acquires a triple root and
the white degrees become $(3,1,1,1)$: the passport is the stated one, with
$3+4=7$ vertices and $6$ edges, again a tree of genus $0$. The quadratic
$25a^{2}-32a+16$ is irreducible over $\Q$ (discriminant $32^{2}-4\cdot25\cdot16=-576<0$),
so its two roots are $\GQ$-conjugate and generate $\Q(i)$, since
$a=\frac{16+12i}{25}$ gives $i=\frac{25a-16}{12}$. The combinatorial count agrees:
enumerating transitive triples of cycle types $(3,2,1)$, $(3,1,1,1)$, $(6)$ in $\Sym_{6}$
up to simultaneous conjugation yields exactly two classes.
\end{proof}

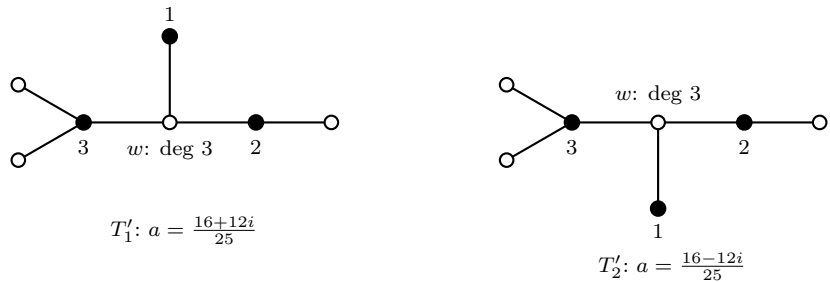
\begin{figure}[t]
\centering
\begin{tikzpicture}[scale=0.95,
  bl/.style={circle,fill=black,inner sep=2.2pt},
  wh/.style={circle,draw=black,fill=white,thick,inner sep=1.8pt},
  ed/.style={thick},
  lbl/.style={font=\scriptsize}]
\begin{scope}[xshift=0cm]
  \coordinate (W3) at (0,0);
  \coordinate (B3) at (-1.2,0);
  \coordinate (B2) at (1.2,0);
  \coordinate (B1) at (0,1.2);
  \coordinate (LA) at (-1.2,0)  ; 
  \draw[ed] (B3)--(W3)--(B2);
  \draw[ed] (W3)--(B1);
  \draw[ed] (B3)--++(150:1.05); \node[wh] at ($(B3)+(150:1.05)$) {};
  \draw[ed] (B3)--++(210:1.05); \node[wh] at ($(B3)+(210:1.05)$) {};
  \draw[ed] (B2)--++(0:1.05);   \node[wh] at ($(B2)+(0:1.05)$) {};
  \node[bl] at (B3) {}; \node[bl] at (B2) {}; \node[bl] at (B1) {};
  \node[wh] at (W3) {};
  \node[lbl,below=3pt] at (B3) {$3$};
  \node[lbl,below=3pt] at (B2) {$2$};
  \node[lbl,above=2pt] at (B1) {$1$};
  \node[lbl,below=3pt] at (W3) {$w$: deg $3$};
  \node[lbl] at (0.2,-1.5) {$T'_{1}$:\ $a=\tfrac{16+12i}{25}$};
\end{scope}
\begin{scope}[xshift=6.8cm]
  \coordinate (W3) at (0,0);
  \coordinate (B3) at (-1.2,0);
  \coordinate (B2) at (1.2,0);
  \coordinate (B1) at (0,-1.2);
  \draw[ed] (B3)--(W3)--(B2);
  \draw[ed] (W3)--(B1);
  \draw[ed] (B3)--++(150:1.05); \node[wh] at ($(B3)+(150:1.05)$) {};
  \draw[ed] (B3)--++(210:1.05); \node[wh] at ($(B3)+(210:1.05)$) {};
  \draw[ed] (B2)--++(0:1.05);   \node[wh] at ($(B2)+(0:1.05)$) {};
  \node[bl] at (B3) {}; \node[bl] at (B2) {}; \node[bl] at (B1) {};
  \node[wh] at (W3) {};
  \node[lbl,below=3pt] at (B3) {$3$};
  \node[lbl,below=3pt] at (B2) {$2$};
  \node[lbl,below=2pt] at (B1) {$1$};
  \node[lbl,above=3pt] at (W3) {$w$: deg $3$};
  \node[lbl] at (0.2,-2.0) {$T'_{2}$:\ $a=\tfrac{16-12i}{25}$};
\end{scope}
\end{tikzpicture}
\caption{The two plane trees of the passport
$\langle(3,2,1),(3,1,1,1),(6)\rangle$ (Theorem~\ref{thm:quadratic-orbit}): a white vertex
of degree $3$ joins all three black vertices, and the remaining three edges are pendants.
The two trees are mirror images---they differ by the cyclic order at the white
vertex---and complex conjugation, which reverses orientation, exchanges them; their
fields of moduli are the two conjugate embeddings $\Q(i)$ and $\Q(-i)$ of the Gaussian
field, which as subfields of $\C$ coincide: here, in contrast with the cubic orbit, the
abstract field is Galois and the embedded field really is common to the orbit.}
\label{fig:two-trees}
\end{figure}

\begin{remark}[What the computation illustrates]\label{rem:orbit-lessons}
Three features deserve emphasis.
\begin{enumerate}[label=\textup{(\roman*)},leftmargin=2.4em]
\item \emph{Fields of moduli of dessins are genuinely arithmetic.} The trees of
Theorem~\ref{thm:cubic-orbit} have fields of moduli the conjugate embeddings of
$\Q(\sqrt[3]{2})$, a non-Galois extension---the very example used in
\S\ref{ssec:fields} to show that the Galois correspondence can fail. A $\GQ$-set of size
$3$ whose point stabilizers are non-normal open subgroups is exactly what a non-Galois
cubic field looks like, and here it is realized by three drawings; the non-normality is
reflected in the fact that the three stabilizers, hence the three embedded fields, are
distinct.
\item \emph{The passport is a coarse invariant, but not a complete one.} Here it happens
to be complete---each of our two passports is a single orbit---but the two orbits sit
inside a single resultant \eqref{eq:resultant}, and are separated only by whether two
critical points collide. Passports that split into several orbits exist and are the
source of the search for finer invariants.
\item \emph{Planarity matters.} $T_{1}$ and $T_{3}$ are isomorphic as abstract trees.
Any invariant computed from the abstract graph---degree sequence, diameter, spectrum---is
blind to the distinction, yet $\GQ$ moves one to the other. Dessins are embedded graphs,
and the embedding is where the arithmetic lives.
\end{enumerate}
\end{remark}
\section{Synthesis and open problems}\label{sec:outlook}

\subsection{What the proof is made of}\label{ssec:synthesis}

It is worth stating plainly which ingredients carry which half of
Theorem~\ref{thm:belyi-intro}, because they are so different in kind.

The implication \emph{three branch points $\Rightarrow$ arithmetic} rests on exactly two
pillars. The first is topological: $\pi_{1}$ of the thrice-punctured sphere is free of
rank $2$, hence finitely generated, hence has finitely many subgroups of each finite index
(Proposition~\ref{prop:finiteness}). The second is Galois-theoretic: a subgroup of
$\Aut(\C)$ of finite index in $\Aut(\C/K)$ has fixed field a finite extension of $K$
(Lemma~\ref{lem:index}). Everything else---the field of moduli, the Riemann--Roch
normalization of Theorem~\ref{thm:moduli-covering}, the semilinear descent of
Lemma~\ref{lem:galois-descent}---is plumbing between the two, needed because the field of
moduli is only \emph{a priori} a field of moduli and not a field of definition.

The converse implication rests on nothing but the chain rule
\eqref{eq:branch-composition} and the arithmetic of the two families of maps: minimal
polynomials in Stage~1, the Belyi polynomials $P_{m,n}$ in Stage~2. It is this asymmetry
that makes the converse the constructive half and the source of every explicit example in
Sections~\ref{sec:belyi-poly}--\ref{sec:galois-dessins}. The two halves also feed back
into each other: the descent theory, applied to the classification of extremal dessins,
produced the rationality statement of Corollary~\ref{cor:A5-rational}, and the character
theory of $\Sym_{d}$, applied to the equality case of an elementary inequality, produced
the mass formula of Theorem~\ref{thm:mass-formula}.

\begin{figure}[t]
\centering
\begin{tikzpicture}[
  box/.style={draw,rounded corners=2pt,align=center,font=\scriptsize,
              inner sep=3.5pt,minimum height=8mm},
  arr/.style={-{Stealth[length=4.5pt]},thick,gray!70!black}]
\node[box,fill=blue!6] (PI) at (0,0)
   {$\pi_{1}(\Pp^{1}\!\smallsetminus\!\{0,1,\infty\})$ free of rank $2$\\ \eqref{eq:pi1}};
\node[box,fill=blue!6] (AUT) at (7.4,0)
   {$\Aut(\C)$: finite index\\ $\Rightarrow$ finite extension (Lem.~\ref{lem:index})};
\node[box] (CLASS) at (0,-1.5)
   {classification by triples\\ (Thm.~\ref{thm:classification})};
\node[box] (FIN) at (3.7,-2.6)
   {finiteness\\ (Prop.~\ref{prop:finiteness})};
\node[box] (DESC) at (7.4,-1.5)
   {semilinear descent\\ (Lem.~\ref{lem:galois-descent}, Thm.~\ref{thm:weil})};
\node[box] (MOD) at (7.4,-3.1)
   {field of moduli\\ (Thms.~\ref{thm:rigid-descent}, \ref{thm:moduli-covering})};
\node[box,fill=red!7] (BELYI) at (3.7,-4.3)
   {\textbf{Belyi's theorem} (Thm.~\ref{thm:belyi-intro})};
\node[box,fill=green!8] (ALGO) at (0,-3.1)
   {Belyi's algorithm\\ $p$, then $P_{m,n}$ (\S\ref{ssec:algorithm})};
\node[box,fill=green!8] (PMN) at (0,-4.6)
   {structure of $P_{m,n}$\\ (Thms.~\ref{thm:pmn-decomp}, \ref{thm:pmn-monodromy})};
\node[box] (DEG) at (7.4,-4.6)
   {degree bounds \& mass formula\\ (Thms.~\ref{thm:degree-bound}, \ref{thm:mass-formula})};
\node[box] (ORB) at (3.7,-5.7)
   {Galois orbits of dessins\\ (Thms.~\ref{thm:cubic-orbit}, \ref{thm:quadratic-orbit})};
\draw[arr] (PI) -- (CLASS);
\draw[arr] (CLASS) -- (FIN);
\draw[arr] (AUT) -- (DESC);
\draw[arr] (DESC) -- (MOD);
\draw[arr] (FIN) -- (BELYI);
\draw[arr] (MOD) -- (BELYI);
\draw[arr] (ALGO) -- (BELYI);
\draw[arr] (ALGO) -- (PMN);
\draw[arr] (BELYI) -- (ORB);
\draw[arr] (CLASS) to[out=-180,in=180,looseness=1.1] (ORB.west);
\draw[arr] (BELYI) -- (DEG);
\end{tikzpicture}
\caption{The logical architecture. The left column is topological and constructive, the
right column arithmetic and descent-theoretic; Belyi's theorem is where they meet, and the
objects studied in Sections~\ref{sec:belyi-poly}--\ref{sec:galois-dessins} are what it
leaves behind.}
\label{fig:architecture}
\end{figure}
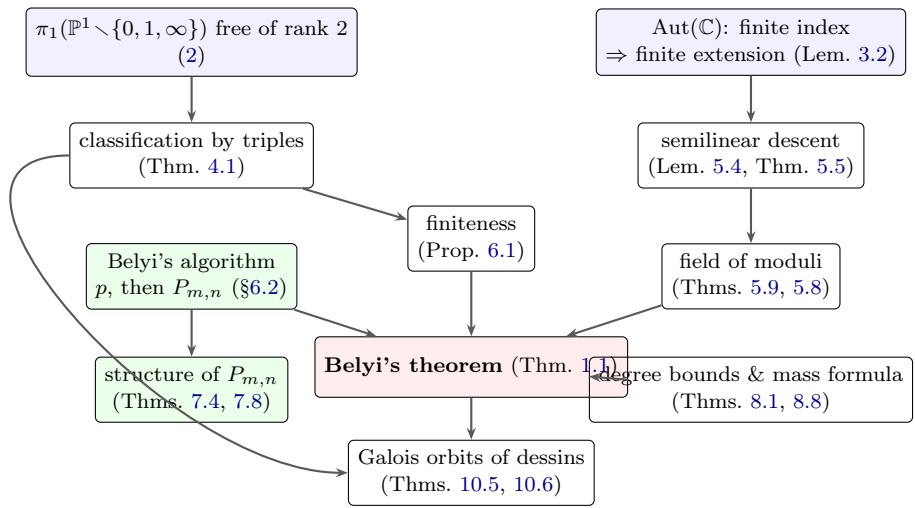

\subsection{Further perspectives}\label{ssec:perspectives}

Three directions in which the objects of this article open onto other areas deserve
mention, beyond the open problems below.

\emph{(1) Hurwitz theory and mathematical physics.} The mass formula
\eqref{eq:mass-formula} is, in the language of enumerative geometry, a Hurwitz number: it
counts degree-$d$ coverings of $\Pp^{1}$ with three prescribed total ramification points,
weighted by automorphisms---the genus-$g$, three-branch-point case of the counts that
appear in the Gromov--Witten theory of curves. Dessins themselves are the ribbon graphs
of matrix-model perturbation theory: the generating functions of maps on surfaces are
Gaussian matrix integrals, an identification developed at length in
\cite[Ch.~3]{LandoZvonkin2004}, and through it weighted dessin counts such as
\eqref{eq:mass-formula} sit inside the same circle of ideas as the Harer--Zagier and
topological-recursion formulas. It would be interesting to know whether the extremal
condition---all three partitions maximally degenerate---admits a natural matrix-model or
topological-recursion interpretation, and whether the exact evaluation
$2(d-1)!/d(d+1)$ is the shadow of a polynomiality statement for a family of triple
Hurwitz numbers.

\emph{(2) Computation of Belyi maps.} The certified elimination of
Section~\ref{sec:galois-dessins}---normalize, reduce to a resultant, prove that no
extraneous solutions arise, and match against an independent combinatorial count---is a
template that scales, in principle, to any passport, and it is exactly the kind of
groundwork the computational theory of Belyi maps rests on; see Sijsling--Voight
\cite{SijslingVoight2014} for the state of the art, where Gr\"obner, complex-analytic,
and $p$-adic methods are compared. The certification step (Lemma~\ref{lem:orbit-bijection})
is the part most often left implicit in the literature, and making it routine would turn
tables of Belyi maps into verified data.

\emph{(3) Regular maps and hypermaps.} A Belyi map is Galois precisely when its dessin is
a \emph{regular hypermap}---one whose automorphism group acts transitively on edges---and
the interplay between regularity, fields of moduli, and defining fields is a developed
subject in its own right \cite{JonesWolfart2016}. Corollary~\ref{cor:pmn-galois} locates
the unique regular member of the family $P_{m,n}$; the quotient structure
$P_{k,k}=\pi_{k}\circ P_{1,1}$ of Remark~\ref{rem:pmn-galois-factor} is the simplest
instance of the general machinery of hypermap quotients, and the wreath products of
Theorem~\ref{thm:pmn-monodromy} are exactly the monodromy groups that the covering-space
composition theory of \cite{AdrianovZvonkin1998} predicts should appear.

\subsection{Open problems}\label{ssec:problems}

\begin{problem}[Belyi degree of an explicit family]\label{prob:belyi-degree}
Theorem~\ref{thm:degree-bound} gives $\mathcal{B}(X)\geq2g+1$, attained by
$y^{2g+1}=x(1-x)$. Determine $\mathcal{B}(X)$ exactly for a family where it is
\emph{not} attained---for instance for the Fermat curves, where
Example~\ref{ex:fermat} gives $N^{2}-3N+3\leq\mathcal{B}(F_{N})\leq N^{2}$. Is
$\mathcal{B}(F_{N})=N^{2}$ for $N\geq4$? Even the case $N=4$ (genus $3$, so
$\mathcal{B}\geq7$ against the upper bound $16$) appears to be open.
\end{problem}

\begin{problem}[Extremal curves and orbits]\label{prob:extremal}
Theorem~\ref{thm:extremal-classification} classifies the extremal Belyi
\emph{coverings} for $g\leq3$ and Proposition~\ref{prop:cyclic-extremal} counts the cyclic
ones in prime degree. Two questions remain. First, identify the extremal \emph{curves}:
which genus-$g$ curve carries the non-cyclic degree-$5$ map with monodromy $A_{5}$, and do
the two degree-$7$ maps with monodromy $\mathrm{PSL}_{2}(\mathbb{F}_{7})$ live on the
Klein quartic---which would give $\mathcal{B}(\mathrm{Klein})=7$
(Remark~\ref{rem:extremal-curves})? Second, refine the mass formula: Theorem~\ref{thm:mass-formula} pins the weighted count
$M_{d}$ at $2(d-1)!/d(d+1)$, and Remark~\ref{rem:mass-asymptotics} bounds the number
$N_{d}$ of extremal dessins between $M_{d}$ and $dM_{d}$. Prove that
$N_{d}\sim M_{d}$---equivalently, that asymptotically almost all extremal dessins have
trivial automorphism group (the data for $d\leq7$ suggest this, the exceptions being
exactly the thin cyclic subfamily)---and determine the $\GQ$-orbit structure on the
$A_{d}$ majority, in particular whether orbits of unbounded size occur among extremal
dessins.
\end{problem}

\begin{problem}[Wreath monodromy beyond $P_{m,n}$]\label{prob:wreath}
Theorem~\ref{thm:pmn-monodromy} computes $\Mon(P_{m,n})=\Sym_{d'}\wr\Z/k$. The proof used
only that the dessin is a double star. Determine the monodromy group of the general
``multiple star'' $z^{m_{1}}(z-\alpha_{2})^{m_{2}}\cdots(z-\alpha_{r})^{m_{r}}$ normalized
to be Belyi: is it again the full symmetric group whenever
$\gcd(m_{1},\dots,m_{r})=1$, and a wreath product otherwise?
\end{problem}

\begin{problem}[Separating invariants]\label{prob:invariants}
In Theorem~\ref{thm:cubic-orbit} the passport determined the orbit, and the two trees
$T_{1},T_{3}$ of Figure~\ref{fig:three-trees} are distinguished only by their planar
embedding. Find a computable invariant of \emph{plane} trees, finer than the passport,
which is constant on $\GQ$-orbits and which separates the known examples of passports
containing more than one orbit.
\end{problem}
\appendix
\section{Documentation of the computations}\label{app:computations}

Every finite enumeration and every elimination used in the body of the article is
documented here, in a form that can be re-executed verbatim. All scripts are in Python
using SymPy (version $\geq1.12$); each runs in well under a minute on commodity hardware
except the degree-$7$ enumeration ($\sim$ minutes). Analogous code in GAP or Magma is a
direct transcription.

\subsection{Enumeration of dessins with prescribed passport}
\label{app:enumeration}

The algorithm implements Theorem~\ref{thm:classification} literally.

\begin{center}
\fbox{\parbox{0.92\textwidth}{\footnotesize
\textbf{Algorithm} (dessins with passport
$\langle\tau_{0},\tau_{1},\tau_{\infty}\rangle$ in degree $d$).
\begin{enumerate}[label=\arabic*.,leftmargin=1.8em,itemsep=1pt]
\item List all $\sigma_{0}\in\Sym_{d}$ of cycle type $\tau_{0}$ and all
$\sigma_{1}\in\Sym_{d}$ of type $\tau_{1}$.
\item Retain the pairs with $(\sigma_{0}\sigma_{1})^{-1}$ of type $\tau_{\infty}$ and
$\langle\sigma_{0},\sigma_{1}\rangle$ transitive.
\item Partition the surviving pairs into orbits under simultaneous conjugation by
$\Sym_{d}$; each orbit is one dessin. Record, per orbit, the order of
$\langle\sigma_{0},\sigma_{1}\rangle$ (the monodromy group) and
$d!/|\text{orbit}|$ (the automorphism group order, i.e.\ the centralizer).
\end{enumerate}}}
\end{center}

\begin{verbatim}
from sympy.combinatorics import Permutation, PermutationGroup
from sympy.combinatorics.named_groups import SymmetricGroup
from math import factorial

def ctype(p):
    return tuple(sorted((len(c) for c in p.full_cyclic_form), reverse=True))

def dessins(d, t0, t1, tinf):
    E = list(SymmetricGroup(d).elements)
    pairs = [(s0, s1) for s0 in E if ctype(s0) == t0
                      for s1 in E if ctype(s1) == t1
             if ctype((s0*s1)**-1) == tinf
             and PermutationGroup([s0, s1]).is_transitive()]
    seen, classes = set(), []
    for (s0, s1) in pairs:
        if (s0, s1) in seen: continue
        orbit = {(g*s0*g**-1, g*s1*g**-1) for g in E}
        G = PermutationGroup([s0, s1])
        classes.append((s0, s1, G.order(), factorial(d)//len(orbit)))
        seen |= orbit
    return classes
\end{verbatim}

\subsection{Outputs used in the text}\label{app:outputs}

\emph{Extremal passports} $\langle(d),(d),(d)\rangle$
(Theorem~\ref{thm:extremal-classification}). Running
\texttt{dessins(d,(d,),(d,),(d,))}:

\begin{center}\small
\begin{tabular}{@{}clll@{}}
\toprule
$d$ & representative $(\sigma_{0},\ \sigma_{1})$ & $\Mon$, $|\Mon|$ & $|\Aut|$ \\
\midrule
$3$ & $(1\,2\,3),\ (1\,2\,3)$ & $\Z/3$, $3$ & $3$\\
$5$ & $(1\,3\,2\,4\,5),\ (1\,3\,2\,4\,5)$ & $\Z/5$, $5$ \ [$3$ classes] & $5$\\
$5$ & $(1\,3\,2\,4\,5),\ (1\,2\,4\,3\,5)$ & $A_{5}$, $60$ \ [$1$ class] & $1$\\
$7$ & $(1\,7\,4\,5\,3\,2\,6),\ (1\,7\,4\,5\,3\,2\,6)$ & $\Z/7$, $7$ \ [$5$ classes] & $7$\\
$7$ & $(1\,7\,4\,5\,3\,2\,6),\ (1\,7\,3\,6\,4\,2\,5)$ &
  $\mathrm{PSL}_{2}(\mathbb{F}_{7})$, $168$ \ [$2$ classes] & $1$\\
$7$ & $(1\,7\,4\,5\,3\,2\,6),\ (1\,7\,2\,5\,6\,4\,3)$ & $A_{7}$, $2520$ \ [$23$ classes]
  & $1$\\
\bottomrule
\end{tabular}
\end{center}

Totals: $1,4,30$ classes in degrees $3,5,7$; total ordered pairs $2,192,129600$, matching
$|C|\cdot 2(d-1)!/(d+1)=(d-1)!\cdot2(d-1)!/(d+1)$ as in
Theorem~\ref{thm:mass-formula}; masses
$\sum1/|\Aut|=\tfrac13,\tfrac85,\tfrac{180}{7}$, matching $2(d-1)!/d(d+1)$. The group of
order $168$ is identified as $\mathrm{PSL}_{2}(\mathbb{F}_{7})$ by order and transitivity:
it is the unique conjugacy class of transitive subgroups of $\Sym_{7}$ of order $168$
(e.g.\ by \texttt{PermutationGroup.is\_primitive} plus the classification of transitive
groups of degree $7$, or in GAP via \texttt{TransitiveIdentification}).

\emph{Degree-six passports} (Theorems~\ref{thm:cubic-orbit},
\ref{thm:quadratic-orbit}). Running the algorithm with
$\tau_{\infty}=(6)$: the passport $\langle(3,2,1),(2,2,1,1)\rangle$ yields $3$ classes,
all with a common $\sigma_{0}=(1\,3\,6)(2\,4)$ and
$\sigma_{1}\in\{(1\,4)(3\,5),\,(2\,3)(4\,5),\,(1\,2)(5\,6)\}$; the passport
$\langle(3,2,1),(3,1,1,1)\rangle$ yields $2$ classes.

\subsection{The elimination of Section~\ref{sec:galois-dessins}}
\label{app:elimination}

The system \eqref{eq:E1}--\eqref{eq:E2} and its certification
(Lemma~\ref{lem:orbit-bijection}):

\begin{verbatim}
import sympy as sp
z, a, c = sp.symbols('z a c')
p  = c*z**3*(z-1)**2*(z-a)
u, v = sp.symbols('u v')            # roots of the quadratic factor of p'
s1 = sp.Rational(5,6)*a + sp.Rational(2,3)   # u+v
s2 = a/2                                     # uv
# symmetric reduction of p(u)+p(v)-2 and p(u)p(v)-1:
from sympy.polys.polyfuncs import symmetrize
E1s, r1, _ = symmetrize(sp.expand(p.subs(z,u)+p.subs(z,v)-2), [u,v], formal=True)
E2s, r2, _ = symmetrize(sp.expand(p.subs(z,u)*p.subs(z,v)-1), [u,v], formal=True)
assert r1 == 0 and r2 == 0
S1, S2 = sp.symbols('s1 s2')
E1 = sp.expand(46656*E1s.subs({S1:s1, S2:s2}))   # -> eq. (E1), linear in c
E2 = sp.expand(  432*E2s.subs({S1:s1, S2:s2}))   # -> eq. (E2)
R  = sp.factor(sp.resultant(E1, E2, c))
# R = -432*(25*a**2-32*a+16)**3*(25*a**3-12*a**2-24*a-16)**2
cub, quad = 25*a**3-12*a**2-24*a-16, 25*a**2-32*a+16
lead = sp.Poly(E1, c).all_coeffs()[0]            # coefficient of c
assert sp.gcd(sp.Poly(lead,a), sp.Poly(cub,a)) == 1     # Lemma (i)
assert sp.gcd(sp.Poly(lead,a), sp.Poly(quad,a)) == 1
assert cub.subs(a,0)*cub.subs(a,1)*quad.subs(a,0)*quad.subs(a,1) != 0   # (ii)
disc = sp.expand((4+5*a)**2 - 72*a)
assert sp.rem(disc, quad, a) == 0                        # (iii): collision
assert sp.gcd(sp.Poly(disc,a), sp.Poly(cub,a)) == 1      # (iii): no collision
a0 = (4 + 18*2**sp.Rational(1,3) + 6*2**sp.Rational(2,3))/25
assert sp.simplify(cub.subs(a, a0)) == 0
c0 = sp.radsimp(sp.solve(E1.subs(a, a0), c)[0])
crit = sp.solve(sp.diff(p.subs({a:a0, c:c0}), z), z)
assert {sp.nsimplify(sp.simplify(p.subs({a:a0,c:c0,z:r}))) for r in crit} == {0,1}
\end{verbatim}

The final assertion verifies that the explicit polynomial with
\eqref{eq:a-explicit}--\eqref{eq:c-explicit} has critical values exactly $\{0,1\}$; the
same loop over the complex roots (using \texttt{CRootOf}) verifies the remaining four
solutions.

\subsection{Verification of the mass formula and the character identity}
\label{app:mass}

The identity \eqref{eq:hook-sum} is checked for $n\leq10$ by
\begin{verbatim}
for n in range(2, 11, 2):
    s = sum(sp.Rational((-1)**r, sp.binomial(n, r)) for r in range(n+1))
    assert s == sp.Rational(2*(n+1), n+2)
\end{verbatim}
and the masses of \S\ref{app:outputs} match \eqref{eq:mass-formula} for $d=3,5,7$; for
$d$ even the enumeration returns no classes, matching the vanishing of
\eqref{eq:hook-sum}.

\subsection{Verification of the cyclic count}\label{app:cyclic}

The count \eqref{eq:cyclic-count} of Proposition~\ref{prop:cyclic-extremal} is verified by
brute force for $d\in\{3,5,7,9,15,21,25,27,45\}$:
\begin{verbatim}
from math import gcd
from sympy import factorint, totient
def brute(d):
    pairs = [(a,b) for a in range(1,d) for b in range(1,d)
             if gcd(a,d)==gcd(b,d)==gcd(a+b,d)==1]
    units = [u for u in range(1,d) if gcd(u,d)==1]
    seen, cl = set(), 0
    for p in pairs:
        if p in seen: continue
        cl += 1
        seen |= {((u*p[0])%d,(u*p[1])%d) for u in units}
    return cl
def formula(d):
    r = 1
    for p,k in factorint(d).items(): r *= p**(k-1)*(p-2)
    return r
for d in (3,5,7,9,15,21,25,27,45):
    assert brute(d) == formula(d)
\end{verbatim}
In each case the number of admissible pairs also equals $\varphi(d)$ times the class
count, confirming that the scaling action is free; and for $d=9$ a dessin-level
enumeration of cyclic triples $(\sigma_{0},\sigma_{1})=(c^{a},c^{b})$, $c$ a fixed
$9$-cycle, up to conjugation, independently returns $3$ classes with $|\Aut|=9$.

\end{document}